\documentclass{amsart}

\usepackage[all]{xy}
\usepackage{verbatim}

\numberwithin{equation}{section}

\theoremstyle{plain}
\newtheorem{theorem}{Theorem}[section]
\newtheorem{corollary}[theorem]{Corollary}
\newtheorem{lemma}[theorem]{Lemma}
\newtheorem{proposition}[theorem]{Proposition}
\theoremstyle{remark}

\theoremstyle{definition}
\newtheorem{definition}[theorem]{Definition}

\makeatletter

\def\R{\mathbb{R}}
\def\C{\mathbb{C}}
\def\Z{\mathbb{Z}}

\newcommand{\lie}[1]{\mathfrak{#1}}
\newcommand{\lier}[1]{\mathfrak{#1}_{0}}

\newcommand{\Span}[1]{\mathrm{Span}\langle{#1}\rangle}

\def\gK{\lie{g},K}
\def\brgK{(\gK)}

\def\Hom{\mathrm{Hom}}
\def\Ext{\mathrm{Ext}}

\def\modulo{\mathrm{mod}\,}

\def\tr{\mathrm{tr}}
\def\diag{\mathrm{diag}}

\def\Ind{\mathrm{Ind}}
\def\Ad{\mathrm{Ad}}

\def\pr{\mathrm{pr}}

\def\sgn{\mathrm{sgn}}

\def\I{\sqrt{-1}}
\def\ds{\displaystyle}

\allowdisplaybreaks[3]

\title[Socle filtrations of representations of $SL(3,\R)$ and 
$Sp(2,\R)$]
{The socle filtrations of principal series representations of 
$SL(3,\R)$ and $Sp(2,\R)$}

\author{Naoki Hashimoto}
\address{
Department of Physics and Mathematics, 
Aoyama Gakuin University, 
5-10-1, Fuchinobe, Chuo-ku, Sagamihara, Kanagawa 252-5258, Japan. }
\email{nhashimoto0514@gmail.com}

\author{Kenji Taniguchi}
\address{
Department of Physics and Mathematics, 
Aoyama Gakuin University, 
5-10-1, Fuchinobe, Chuo-ku, Sagamihara, Kanagawa 252-5258, Japan. }
\email{taniken@gem.aoyama.ac.jp}

\author{Go Yamanaka}
\address{
Department of Physics and Mathematics, 
Aoyama Gakuin University, 
5-10-1, Fuchinobe, Chuo-ku, Sagamihara, Kanagawa 252-5258, Japan. }
\email{tohma17@gmail.com}

\subjclass[2010]{Primary 22E46}
\keywords{principal series representation, socle filtration, 
$SL(3,\R)$, $Sp(2,\R)$}

\begin{document}

\begin{abstract}
We study the structure of the $\brgK$-modules of 
the principal series representations of $SL(3,\R)$ and $Sp(2,\R)$ 
induced from minimal parabolic subgroups, 
in the case when the infinitesimal character is nonsingular. 
The composition factors of these modules are known by 
Kazhdan-Lusztig-Vogan conjecture. 
In this paper, we give complete descriptions of 
the socle filtrations of these modules. 

\end{abstract}

\maketitle

\section{Introduction}
\label{section:introduction}

Given a representation of a group or an algebra, 
the determination of its composition series is a natural problem. 
In \cite{KL}, Kazhdan and Lusztig conjectured that 
the composition factors of Verma modules of complex reductive Lie algebras 
are determined by the value of so-called Kazhdan-Lusztig polynomials at 
$q=1$. 
This conjecture is called a Kazhdan-Lusztig conjecture, 
and it is proved independently by Beilinson and Bernstein (\cite{BB}) 
and by Brylinski and Kashiwara (\cite{BK}). 

Let $G$ be a real reductive Lie group, 
and $K$ a maximal compact subgroup of $G$. 
The complexified Lie algebra of $G$ is denoted by $\lie{g}$. 
By Harish-Chandra's subquotient theorem, 
every irreducible $\brgK$-module, 
is realized as a subquotient module 
of some principal series module. 
Moreover, Langlands (\cite{L}) and Mili\v{c}i\'{c} 
classified the irreducible $\brgK$-modules 
in terms of standard modules, namely some kind of generalized 
principal series modules. 
As for the composition factor problem for standard modules, 
Vogan generalized Kazhdan-Lusztig conjecture to the standard modules 
of $G$ (\cite{VIC-II}), which we call the Kazhdan-Lusztig-Vogan (KLV) conjecture. 
This conjecture was proved by himself (\cite{VIC-III}). 

As introduced above, the composition factor problem for standard modules is 
completely solved in 1980's. 
On the other hand, the composition series problem, 
namely the complete determination of the socle filtrations 
of standard modules, 
is very difficult. 
Here, for a $\brgK$-module $V$ of finite length, 
the socle is defined as the largest semisimple submodule of $V$. 
Let $V_{1}$ be the socle of $V$, 
and let $V_{2}$ be the socle of $V/V_{1}$, and so on. 
The filtration of $V$ so obtained is called the socle filtration of $V$. 

For the groups of real rank one, 
this problem is completely solved by Collingwood (\cite{C}). 
As for higher rank groups, there are many researches on 
the structure of degenerate principal series 
(for example \cite{HT}, \cite{LeeS-T}, \cite{LL-1}, \cite{LL-2}). 
But, as far as the authors know, for the principal series representations 
induced from minimal parabolic subgroups, 
there are few complete results. 
The intertwining operators 
between principal series representations are explored 
in the cases of $G=SL(3,\C)$ by Tsuchikawa \cite{Tsu}, 
$G=SL(3,\R)$ by A. I. Fomin \cite{F} 
and $G=Sp(2,\R)$ by Mui\'{c} \cite{Muic}, 
and they obtained some partial results on the structure of 
principal series representations. 
The problem of complete determination of socle filtrations of 
principal series representations is still open for groups 
with higher real rank. 
The main objective of this paper is to describe completely 
the socle filtrations of the principal series modules 
induced from minimal parabolic subgroups with nonsingular integral 
infinitesimal characters, 
in the cases of the groups $SL(3,\R)$ and $Sp(2,\R)$. 
The authors hope that the results of this paper become 
cornerstones for solving the composition series problem generally. 

One application of the results of this paper is 
the description of the socle filtration of the standard Whittaker 
$\brgK$-modules \cite{T2} defined by the second author. 
Since the structures of these modules resemble those of 
the principal series modules in some parts, 
the results of this paper play important roles 
when we determine the socle filtrations of 
the standard Whittaker $\brgK$-modules of $SL(3,\R)$ and $Sp(2,\R)$. 
Details of this study on Whittaker modules will be reported elsewhere. 

We explain the contents of this paper. 
In Section~\ref{section:Langlands classification}, 
we review the Langlands classification of 
irreducible $\brgK$-modules 
after Vogan's book \cite{V-Green}. 
In Section~\ref{section:outline}, 
we explain the outline of our computation. 
We frequently use the shift operators of $K$-types, 
so we also explain them in this section. 
Section~\ref{section:other tools} consists of the explanation of 
other tools used in this paper. 
Section~\ref{section:representation of K} 
consists of the review of the well-known structure of 
the finite dimensional representations of $\lie{sl}(2)$, 
since the maximal compact subgroups of $SL(3,\R)$ and $Sp(2,\R)$ are 
$SO(3)$ and $U(2)$, respectively.

We investigate the structure of principal series modules of 
$SL(3,\R)$ in Sections~\ref{section:irreducible modules. SL(3)} -- 
\ref{section:main results for SL(3)}. 
In Section~\ref{section:irreducible modules. SL(3)}, 
we review the classification of irreducible $\brgK$-modules of $SL(3,\R)$, 
and the $K$-spectra of these irreducible modules are presented 
in Section~\ref{section:K-spectra, SL(3)}. 
In Section~\ref{section:shift operators, SL(3)}, 
we write down the shift operators of $K$-types for $SL(3,\R)$ explicitly. 
These operators are used for determination of candidates of irreducible 
submodules in principal series modules. 
This is done in Section~\ref{section:P_{-1}=p_{-2}=0} and 
\ref{section:candidates for irr sub. SL(3)}. 
After these preparations, 
we determine the socle filtrations of principal series modules 
of $SL(3,\R)$ completely in Section~\ref{section:main results for SL(3)}. 
The main result is Theorem~\ref{theorem:main results for SL(3)}. 

We treat the case of $Sp(2,\R)$ in 
Sections~\ref{section:the group Sp(2)} 
-- 
\ref{section:determination, X(11)}. 
We review the structure of the group $Sp(2,\R)$ 
in Section~\ref{section:the group Sp(2)}, 
and that of irreducible modules of it 
in Section~\ref{section:irreducible modules of Sp(2)}. 
$K$-spectra of irreducible modules are presented in 
Section~\ref{section:K-spectra, Sp(2)}. 
In Section~\ref{section:shift operators, Sp(2)}, 
we write down the shift operators of $K$-types explicitly. 
In Section~\ref{section:candidates for submodules of PS}, 
The candidates of irreducible submodules are obtained, 
and the socle filtrations of the principal series modules 
in the block $PSO(4,1)$ are determined. 
Those of the principal series modules in the block 
$PSO(3,2)$ are determined in 
Sections~\ref{section:determination, X(00)} and 
\ref{section:determination, X(11)}. 
The main results for the case of $Sp(2,\R)$ are 
Theorems~\ref{theorem:mail result, PSO(4,1)}, 
\ref{theorem:main results for Sp(2), X(gamma_{10})} 
and 
\ref{theorem:main results for Sp(2), X(gamma_{11})}.

Before going ahead, we fix some notation. 
For a Lie group $L$, its Lie algebra is denoted by 
the corresponding German letter 
$\lier{l}$ and its complexification by $\lie{l}$. 
The complex dual space of $\lie{l}$ is denoted by $\lie{l}^{\ast}$. 
For a compact Lie group $L$, the set of equivalence classes of 
irreducible representations of $L$ is denoted by $\widehat{L}$. 
For $\tau \in \widehat{L}$, 
its representation space is denoted by $V_{\tau}^{L}$. 
If $\tau$ is specified by a highest weight $\lambda$ of it, 
we also denote it by $V_{\lambda}^{L}$. 
This notation will be used for irreducible finite dimensional 
representations of $\lie{sl}(2)$. 
Let $A$ be a Lie group and $B$ its closed subgroup. 
For a representation $(\pi,V)$ of $A$ and an irreducible representation 
$(\tau,W)$ of $B$, 
the $\tau$-isotypic subspace of $V$ is denoted by $V(\tau)$. 

We denote the matrix unit $(\delta_{i,k}\, \delta_{j,l})_{i,j}$ by 
$E_{k,l}$. 
The diagonal matrix with the entries $a_{1}, \dots, a_{k}$ is 
denoted by $\diag(a_{1},\dots,a_{k})$. 
Once a basis $\{v_{i}\}$ of a vector space $V$ is fixed, 
we often denote an element of $V$ by its coordinate with respect to 
this basis. 
For example, 
once a basis $\{e_{i}\}$ of a Cartan subalgebra $\lie{t}$ of $K$ is 
fixed, we write the element $\alpha = \sum_{i} a_{i} e_{i}$ 
as $(a_{1}, \dots, a_{k})$. 

Next, we fix the notation of $\brgK$-modules. 
Let $G$ be a real reductive linear Lie group in the sense of 
\cite{V-Green}. 
Choose a Cartan involution $\theta$ of $G$, 
and let $K$ be the maximal compact subgroup of $G$ consisting of 
fixed points of $\theta$. 
The corresponding Cartan decomposition of $\lier{g}$ is 
denoted by $\lier{g} = \lier{k} + \lier{s}$. 
For an admissible $\brgK$-module $X$, 
the {\it $K$-spectrum} of $X$, 
namely the set 
$\{V_{\tau}^{K} \in \widehat{K} 
\mid 
\Hom_{K}(V_{\tau}^{K}, X|_{K}) \not= 0\}$ 
of $K$-types of $X$, 
is denoted by $\widehat{K}(X)$. 

Finally, we fix the notation of principal series modules. 
Let $(\lie{a}_{m})_{0}$ be a maximal commutative subspace of $\lier{s}$. 
Define $A_{m} := \exp (\lie{a}_{m})_{0}$ and $M_{m} := Z_{K}(A)$. 
Choose a positive system $\Sigma^{+}$ of 
the restricted root system $\Sigma(\lier{g}, (\lie{a}_{m})_{0})$. 
As usual, half the sum of elements in $\Sigma^{+}$ is denoted by $\rho_{m}$. 
Define the nilpotent subgroup $N_{m}$ of $G$ which corresponds 
to the positive system $\Sigma^{+}$. 
Then, $P_{m} := M_{m} A_{m} N_{m}$ is a minimal parabolic subgroup of $G$. 
For a representation of $M$ and a linear character $\nu \in \lie{a}^{\ast}$, 
define the principal series module by 
\[
I(\sigma, \nu) 
:= 
\mathrm{Ind}_{P_{m}}^{G}
(\sigma \otimes e^{\nu+\rho_{m}} \otimes 1)_{\mbox{\tiny $K$-finite}}.
\]
As we explained in the introduction, 
the objective of this paper is 
to determine the socle filtration of $I(\sigma, \nu)$ explicitly.

\noindent
\textbf{Acknowledgment.} 
The second author was supported by JSPS KAKENHI Grant Number JP24540027.

\section{Regular characters and the Langlands classification}
\label{section:Langlands classification}

In this section, we review 
the Langlands classification of $\brgK$-modules 
by means of regular characters after Vogan's book \cite{V-Green}. 

Let $H$ be a Cartan subgroup of $G$. 
An element $\Lambda$ of $\lie{h}^{\ast}$ defines 
an infinitesimal character $\chi_{\Lambda}$ of $Z(\lie{g})$. 
If a $\brgK$-module $V$ admits the infinitesimal character $\chi_{\Lambda}$, 
we also say $V$ admits the infinitesimal character $\Lambda$.

Suppose that $H$ is $\theta$-stable. 
let $T := H \cap K$ and $A := H \cap \exp \lier{s}$. 
Then $H = T A$. 
The centralizer of $A$ in $G$ is denoted by $L$, 
and its Langlands decomposition by $L = M A$. 
Then, $T$ is a compact Cartan subgroup of $M$. 

A {\it regular character} (\cite[Definition~6.6.1]{V-Green}) 
is an ordered pair $\gamma = (\Gamma, \overline{\gamma})$ with 
$\Gamma$ an ordinary character of $H$ and 
$\overline{\gamma} \in \lie{h}^{\ast}$ satisfying 
\begin{enumerate}
\item
if $\alpha \in \Delta(\lie{m}, \lie{t})$, 
then $\langle \alpha, \overline{\gamma} \rangle$ is 
real and non-zero. 
\item
The differential of $\Gamma$ is 
\[
d \Gamma 
= \overline{\gamma} + \rho_{\lie{m}} - 2 \rho_{\lie{m} \cap \lie{k}}, 
\]
where $\rho_{\lie{m}} \in \lie{h}^{\ast}$ is half the sum of roots in 
$\Delta^{+}(\lie{m}, \lie{t}) 
:= 
\{\alpha \in \Delta(\lie{m}, \lie{t}) 
\mid 
\langle \alpha, \overline{\gamma} \rangle > 0\}$, 
and $\rho_{\lie{m} \cap \lie{k}}$ is half the sum of 
compact roots in $\Delta^{+}(\lie{m}, \lie{t})$.  
\end{enumerate}
The set of regular characters of $H$ is denoted by $\widehat{H}'$.

Choose a regular character 
$\gamma = (\Gamma, \overline{\gamma}) 
\in \widehat{H}'$. 
There exists a discrete series module $\sigma$ of $M$ 
such that the highest weight of its minimal $M \cap K$-type 
is $\Gamma|_{T}$. 
Put $\nu := \overline{\gamma}|_{\lie{a}}$. 
Let $P = M A N$ a parabolic subgroup of $G$ with Levi factor $MA$, 
and let $\rho$ be half the sum of roots in $\Delta(\lie{n}, \lie{a})$. 
Denote by $X(\gamma)$ the $\brgK$-modules of 
$\mathrm{Ind}_{P}^{G} (\sigma \otimes e^{\nu+\rho} \otimes 1)$, 
and call it 
the {\it standard modules} with parameter 
$(H, \gamma)$. 
Note that this module admits the infinitesimal character $\overline{\gamma}$.

Since we only treat $\brgK$-modules with nonsingular infinitesimal characters 
in this paper, 
we assume that $\overline{\gamma}$ is nonsingular, for simplicity. 
If we choose $N$ to be negative with respect to $\nu$, 
then $X(\gamma)$ has a unique irreducible submodule. 
We call it the {\it Langlands submodule of $X(\gamma)$} 
and we denote it by $\overline{X}(\gamma)$. 

\begin{theorem}[Langlands classification]
Let $\Lambda$ be a nonsingular infinitesimal character. 
Define $\widehat{H}_{\Lambda}'$ 
to be the set of regular characters 
$\gamma = (\Gamma, \overline{\gamma}) \in \widehat{H}'$ 
which satisfies $\chi_{\overline{\gamma}} = \chi_{\Lambda}$.  
\begin{enumerate}
\item
For an irreducible $\brgK$-module $\pi$ 
with the infinitesimal character $\Lambda$, 
there exists a Cartan subgroup $H$ of $G$ and a regular character 
$\gamma \in \widehat{H}_{\Lambda}'$ such that 
$\pi$ is isomorphic to $\overline{X}(\gamma)$. 
\item
If $\overline{X}(\gamma_{1}) \simeq \overline{X}(\gamma_{2})$, 
$\gamma_{i} \in (\widehat{H_{i}})_{\Lambda}'$ ($i=1,2$), 
then $(H_{1}, \gamma_{1})$ and $(H_{2}, \gamma_{2})$ are conjugate 
under the action of $K$. 
\end{enumerate}
\end{theorem}

\section{Outline of computation}\label{section:outline}

Most of this paper consists of direct computation, 
so it may be helpful for readers to write the outline of computation here. 

Firstly, we divide the set of irreducible modules into 
blocks (\cite[Definition~9.2.1]{V-Green}). 
Block equivalence of irreducible $\brgK$-modules is the equivalence relation 
generated by 
\[
X \sim Y \quad \Leftrightarrow \quad 
\Ext_{\gK}^{1}(X, Y) \not= 0.
\]
The equivalence classes are called {\it blocks}. 
By this definition, we may restrict our consideration 
within a block. 

Secondly, by the translation principle, we may restrict our interest in 
the case of a special infinitesimal character, 
especially in the case when the infinitesimal character is trivial.

The most important part of our computation is 
to seek the candidates for the irreducible submodules 
of principal series modules in question. 
This is done by use of shift operators of $K$-types. 
We explain this part in detail. 
The KLV-conjecture tells us 
the composition factors of a principal series module. 
By the Blattner formula of $K$-spectra of discrete series modules and 
the $K$-spectra of principal series modules, 
we can compute the $K$-spectra of all irreducible modules. 
Suppose that $\tau_{1}$ is a $K$-type of 
an irreducible $\brgK$-module $\pi$, 
but another irreducible $K$-representation $\tau_{2}$ is not. 
Suppose moreover that there exists a shift operator of $K$-types $P$ 
in a principal series module, which sends 
elements of $\tau_{1}$ to $\tau_{2}$. 
If $\pi$ is in the socle of this principal series module, 
then the kernel of $P$ is non-trivial. 
In this way, we determine the candidates for irreducible factors 
in the socle of the principal series module. 
This method is used for the determination of 
the embedding of discrete series modules 
into induced modules in \cite{Y}. 
We apply this method to other irreducible modules. 

The shift operators of $K$-types are used also for 
the determination of the socle filtration. 
Suppose that a non-zero vector $v_{1}$ is known 
to be contained in an irreducible factor $V_{1}$, and 
another non-zero vector $v_{2}$ is known to be contained 
in an irreducible factor $V_{2}$ but not in $V_{1}$. 
If there exists a shift operator of $K$-types 
which sends $v_{1}$ to $v_{2}$, 
then the irreducible factor $V_{1}$ lies in a floor higher than $V_{2}$. 

Given a vector of the principal series module, 
it is in general hard to specify the irreducible factor in which 
this vector is contained, 
but in some cases it is easy. 
For example, if the multiplicity of a $K$-type $\tau$ 
of a principal series module in question is one, 
we can specify the irreducible factor in which this $K$-type $\tau$ is 
contained from the information of $K$-spectra of irreducible factors. 

We also use some known facts on the structure of $\brgK$-modules. 
These are summarized in the next section.

We explain the shift operator of $K$-types briefly. 
Let $(\tau, V_{\tau}^{K})$ be an irreducible representation of $K$, 
and $(\tau^{\ast}, (V_{\tau}^{K})^{\ast})$ be its 
contragredient representation. 
There is a natural identification
\begin{align*}
\Hom_{K}(\tau, I(\sigma, \nu)) 
& \simeq 
C_{\tau^{\ast}}^{\infty} 
(K \backslash G / M_{m} A_{m} N_{m}; \sigma \otimes e^{\nu+\rho_{m}}) 
\\
& := 
\{\Phi : G \overset{C^{\infty}}{\longrightarrow} 
(V_{\tau}^{K})^{\ast} \otimes V_{\sigma}^{M_{m}} \mid 
\\
& \qquad \qquad 
\Phi(kgman) = a^{-\nu-\rho_{m}} \tau^{\ast}(k) \otimes \sigma(m)^{-1} \, 
\Phi(g)
\\
& \qquad \qquad \qquad 
k \in K, g \in G, m \in M_{m}, a \in A_{m}, n \in N_{m}\}. 
\end{align*}
The correspondence is given by 
\[
\varphi(v)(g) = \langle \Phi(g), v \rangle, 
\qquad 
v \in V_{\tau}^{K}, 
\]
where $\varphi \in \Hom_{K}(\tau, I(\sigma, \nu))$, 
$\Phi \in 
C_{\tau^{\ast}}^{\infty} 
(K \backslash G / M_{m} A_{m} N_{m}; \sigma \otimes e^{\nu+\rho_{m}})$ 
and $\langle \ , \ \rangle$ is the paring of $(V_{\tau}^{K})^{\ast}$ 
and $V_{\tau}^{K}$. 

Let $\{X_{i}\}$ be a basis of $\lie{s}$ 
and let $\{X^{i}\}$ be the dual basis of $\lie{s}$ 
with respect to a non-degenerate invariant bilinear form on $\lie{g}$. 
The adjoint representation of $K$ on $\lie{s}$ is denoted by 
$(\Ad_{\lie{s}}, \lie{s})$. 
Define a $K$-equivariant map 
\[
\nabla : C_{\tau^{\ast}}^{\infty} 
(K \backslash G / M_{m} A_{m} N_{m}; \sigma \otimes e^{\nu+\rho_{m}})
\rightarrow 
C_{\tau^{\ast} \otimes \Ad_{\lie{s}}}^{\infty} 
(K \backslash G / M_{m} A_{m} N_{m}; \sigma \otimes e^{\nu+\rho_{m}}), 
\]
by 
\[
\nabla \Phi(g) 
:= 
\sum_{i} L(X_{i})\, \Phi(g) \otimes X^{i}.
\]
Here, $L( \ast )$ denotes the left translation. 

Let $\Delta(\lie{s}, \lie{t})$ be the weight space of the 
adjoint representation$(\Ad_{\lie{s}}, \lie{s})$ 
with respect to a Cartan subalgebra $\lie{t}$ of $\lie{k}$. 
For the groups $G=SL(3,\R)$ and $Sp(2,\R)$, 
the multiplicities of each weights are all one. 
If $\tau = \tau_{\lambda}$, namely if the highest weight of 
$\tau$ is $\lambda$, 
then the irreducible decomposition of 
$\tau_{\lambda} \otimes \Ad_{\lie{s}}$ is given by 
\[
\tau_{\lambda} \otimes \Ad_{\lie{s}} 
\simeq 
\bigoplus_{\alpha \in \Delta(\lie{s},\lie{t})} 
m(\alpha)\, \tau_{\lambda+\alpha}, 
\qquad 
m(\alpha) = 0 \mbox{ or } 1. 
\]
Let 
\[
\mathrm{pr}_{\alpha} : 
(\tau_{\lambda})^{\ast} \otimes \Ad_{\lie{s}} 
\rightarrow (\tau_{\lambda+\alpha})^{\ast}
\]
be the natural projection along this decomposition. 
Then 
\begin{align*}
\mathcal{P}_{\alpha} 
:= 
\pr_{\alpha} \circ \nabla 
\, : \, &
\Hom_{K}(\tau_{\lambda}, I(\sigma, \nu)) 
\simeq 
C_{(\tau_{\lambda})^{\ast}}^{\infty} 
(K \backslash G / M_{m} A_{m} N_{m}; \sigma \otimes e^{\nu+\rho_{m}}) 
\\
& \rightarrow 
C_{(\tau_{\lambda+\alpha})^{\ast}}^{\infty} 
(K \backslash G / M_{m} A_{m} N_{m}; \sigma \otimes e^{\nu+\rho_{m}}) 
\simeq 
\Hom_{K}(\tau_{\lambda+\alpha}, I(\sigma, \nu)) 
\end{align*}
is a $K$-equivariant map obtained from the action of $\lie{g}$ on 
$I(\sigma, \nu)$. 
We call these operators the {\it shift operators of $K$-types}.

\section{Other tools}
\label{section:other tools}

Our direct method explained in the last section 
is not sufficient for our purpose. 
We use several known results on the representations 
of real reductive groups. 

\subsection{Dual principal series}
\label{subsection:dual principal series}

It is well known that there is a non-degenerate invariant pairing 
\begin{align*}
& 
I(\sigma, \nu) \times I(\sigma^{\ast}, -\nu) 
\rightarrow \C, 
&
&
\langle f_{1}, f_{2} \rangle 
= 
\int_{K} f_{1}(k) \, f_{2}(k) \, dk. 
\end{align*}
By this, we can compare the socle filtrations of these principal series 
modules. 

\subsection{Integral intertwining operators between principal series}

We also use the well known integral intertwining operators 
between principal series modules and their factorizations 
(\cite{Schiffmann}). 
Fix a minimal parabolic subgroup $P_{m} = M_{m} A_{m} N_{m}$. 
Let $\Sigma^{+}$ be the positive system of the root system 
$\Sigma(\lier{g}, \lier{a})$ corresponding to $N_{m}$. 
Choose $\nu \in \lie{a}^{\ast}$ so that $\mathrm{Re}\, \nu$ is 
positive with respect to $\Sigma^{+}$. 
Denote by $w^{\circ}$ the longest element of the Weyl group $W(G, A_{m})$ 
with respect to this positive system. 
Let $w^{\circ} = r_{1} r_{2} \cdots r_{\ell}$ be a reduced expression 
of $w^{\circ}$, and put $w_{k} := r_{k} \cdots r_{\ell}$, 
$k = 1, 2, \dots, \ell$. 
Then, there are series of intertwining operators 
\begin{align*}
I(\sigma, \nu) 
& \rightarrow 
I(w_{\ell} \cdot (\sigma, \nu)) 
\rightarrow 
I(w_{\ell-1} \cdot (\sigma, \nu)) 
\rightarrow \cdots 
\\
& 
\rightarrow 
I(w_{k} \cdot (\sigma, \nu)) 
\rightarrow \cdots 
\rightarrow 
I(w^{\circ} \cdot (\sigma, \nu)). 
\end{align*}
It is known that the composition of these operators is not zero. 

These operators help us in some cases. 
For example, 
suppose that the socles of 
$I(w_{k} \cdot (\sigma, \nu))$ 
and 
$I(w_{k-1} \cdot (\sigma, \nu))$ 
are known to be identical, and they consist of one irreducible factor. 
Moreover, suppose that the multiplicity of 
this irreducible factor in these principal series is one. 
Then, since there exists a non-trivial intertwining operator between 
these principal series modules, they are isomorphic. 

\subsection{Horizontal symmetry}
\label{subsection:horizontal symmetry} 
In some case, 
there is a symmetric structure in the socle filtration of 
a principal series module. 

\begin{theorem}\label{theorem:Vogan, Borel-Wallach}
\rm{(}Vogan, Borel-Wallach, \cite[Chapter I, Corollary~7.5]{BW}\rm{)}. 
Let $G$ be a connected real semisimple linear Lie group whose complexification 
is simply connected. 
Then there exists an element $\mu$ of $\mathrm{Aut}\, G$ which satisfies 
the following properties: 
\begin{enumerate}
\item
$\mu (K) = K$, 
\item
For any irreducible admissible $\brgK$-module $(\pi,V)$, 
the twisted module  
\[
(\pi^{\mu},V^{\mu}) 
:= (\pi \circ \mu^{-1}, V)
\] 
is isomorphic to the contragredient $\brgK$-module 
$(\pi^{*},V^{*})$. 
\end{enumerate}
\end{theorem}

In the setting of Theorem~\ref{theorem:Vogan, Borel-Wallach}, 
if $\mu$ stabilizes $A_{m}$, 
then $I(\sigma, \nu)^{\mu}$ is also a principal series 
module. 
So if we know the structure of $I(\sigma, \nu)$, 
then that of $I(\sigma, \nu)^{\mu}$ 
is determined. 
This is a matter of course. 
But in a special situation, 
this theorem implies a symmetric structure in $I(\sigma, \nu)$. 

In order to state it simply, we need a definition. 

\begin{definition}
Let $\pi_{1}$ and $\pi_{2}$ be admissible $\brgK$-modules 
with nonsingular infinitesimal characters. 
These $\brgK$-modules are called {\it quasi-isomorphic} 
if they lie in the same coherent family 
and lie in the same open Weyl chamber. 
The module $\pi_{2}$ is {\it quasi-dual} to $\pi_{1}$ 
if $\pi_{2}$ is quasi-isomorphic to the contragredient module 
of $\pi_{1}$. 
\end{definition}

\begin{corollary}[Horizontal symmetry]
\label{corollary:horizontal symmetry}
In the setting of Theorem~\ref{theorem:Vogan, Borel-Wallach}, 
suppose that $I(\sigma, \nu)^{\mu}$ is quasi-isomorphic to $I(\sigma, \nu)$. 
Then an irreducible factor $\pi$ of $I(\sigma,\nu)$ 
and its quasi-dual module $\pi'$, with the same infinitesimal character 
as that of $\pi$,  appear in $I(\sigma, \nu)$ as a pair $\pi \oplus \pi'$. 
In other words, they appear in the same floor of the socle 
filtration of $I(\sigma,\nu)$. 
\end{corollary}
\begin{proof}
Put $I(\sigma', \nu') \simeq I(\sigma, \nu)^{\mu}$. 
In the situation of this corollary, 
$\psi_{\nu'}^{\nu} I(\sigma', \nu') 
\simeq I(\sigma, \nu)$, 
where $\psi_{\nu'}^{\nu}$ is the translation functor. 
By Theorem~\ref{theorem:Vogan, Borel-Wallach}, 
if $\pi$ is in the $k$-th floor of the socle filtration of $I(\sigma, \nu)$, 
so is the contragredient module $\pi^{\ast}$ of $\pi$ 
in that of $I(\sigma', \nu')$. 
By the translation principle, 
$\pi'$ is in the $k$-th floor of $I(\sigma, \nu)$. 
\end{proof}

\subsection{Parity of length}
\label{subsection:parity of length}

In \cite[Definition~8.1.4]{V-Green}, 
Vogan defined the (integral) length $\ell(\gamma)$ of a regular character 
$\gamma$. 
The following theorem is a consequence of KLV-conjecture. 
\begin{theorem}\rm{(}Vogan, \cite[Theorem~9.5.1]{V-Green}\rm{)}. 
Let $\gamma_{1}$, $\gamma_{2}$ be regular characters 
of the same nonsingular infinitesimal character. 
Suppose that they are not conjugate under $K$. 
Then 
\[
\Ext_{\gK}^{1}(\overline{X}(\gamma_{1}), \overline{X}(\gamma_{2})) 
\not = 0
\]
only if 
\[
\ell(\gamma_{1}) - \ell(\gamma_{2}) \equiv 1 
\quad 
(\mathrm{mod}\, 2). 
\]
\end{theorem}

This theorem enables us to narrow down the candidates 
for irreducible factors lying in each floor. 
We use it in a situation as below.

\begin{corollary}[Parity condition]
\label{corollary:parity condition}
Suppose that $V$ is a $\brgK$-module of finite length. 
If the lengths of the irreducible factors in the 
$k$-th floor of the socle filtration of $V$ 
are all even (resp. odd), 
then those of the factors in $(k+1)$-st floor are all odd (resp. even). 
\end{corollary}

\section{Finite dimensional representations of $\lie{sl}_{2}$}
\label{section:representation of K}
In order to write down the shift operator of $K$-types 
(Section~\ref{section:outline}) explicitly, 
it is needed to write the irreducible decomposition of 
the tensor product $\tau_{\lambda} \otimes \Ad_{\lie{s}}$. 
Since the maximal compact subgroups $K$ of $G = SL(3,\R)$ 
and $Sp(2, \R)$ are 
$SO(3)$ and $U(2)$, 
respectively, 
we need the irreducible decomposition of the tensor product of 
representations of $\lie{sl}_{2}$. 

Let $\{H, E, F\}$ be the standard $\lie{sl}_{2}$-triple, 
namely they satisfy 
\begin{align*}
& 
[H, E] = 2 E, 
&
&
[H, F] = -2 F, 
&
&
[E, F] = H. 
\end{align*}

For $\lambda \in \frac{1}{2} \Z_{\geq 0}$, 
let 
\[
\{v_{q}^{\lambda} 
\mid 
-\lambda \leq q \leq \lambda, 
\
\lambda - q \in \Z
\}
\]
be the basis of the irreducible representation of $\lie{sl}_{2}$ 
with the highest weight $2 \lambda$, 
which satisfies 
\begin{align*}
& 
H v_{q}^{\lambda} 
= 
2 q\, v_{q}^{\lambda}, 
&
&
E v_{q}^{\lambda} 
= 
(\lambda-q)\, v_{q+1}^{\lambda}, 
&
&
F v_{q}^{\lambda} 
= 
(\lambda+q)\, v_{q-1}^{\lambda}. 
\end{align*}
Note that the highest weight is not $\lambda$ but $2\lambda$. 
For later purpose, 
we need the actions of some elements in $SU(2)$. 
\begin{lemma}
\label{lemma:the action of elements on repn of SU(2)}
For 
$\begin{pmatrix}
a & 0 \\ 0 & a^{-1} 
\end{pmatrix}, 
\begin{pmatrix}
0 & -1 \\ 1 & 0 
\end{pmatrix}
\in SU(2)$, 
\begin{align*}
&
\tau_{2\lambda}
(
\begin{pmatrix}
a & 0 \\ 0 & a^{-1} 
\end{pmatrix}
)
\, 
v_{q}^{\lambda} 
= 
a^{2q} 
\, 
v_{q}^{\lambda}, 
&
&
\tau_{2\lambda}
(
\begin{pmatrix}
0 & -1 \\ 1 & 0 
\end{pmatrix}
)
\, 
v_{q}^{\lambda} 
= 
(-1)^{\lambda-q} 
\, 
v_{-q}^{\lambda}. 
\end{align*}
\end{lemma}

\begin{lemma}[Irreducible decomposition of 
$\tau_{2\lambda} \otimes \tau_{2\mu}$]
\label{lemma:irred. decomp. of tau otimes Ad}
For $\lambda, \mu \in \frac{1}{2} \Z_{\geq 0}$ with $\lambda \geq \mu$, 
the irreducible decomposition of the tensor product 
$(\tau_{2\lambda}, V_{2\lambda}^{\lie{sl}_{2}}) 
\otimes 
(\tau_{2\mu}, V_{2\mu}^{\lie{sl}_{2}})$ is given by 
\begin{align*}
v_{q}^{\lambda} \otimes v_{i}^{\mu} 
= 
\sum_{|j| \leq \mu, \ \mu-j \in \Z}
c^{\lambda, \mu}(q, i; j)\, v_{q+i}^{\lambda+j}, 
\end{align*}
where 
\begin{align*}
c^{\lambda, \mu}(q, i; j) 
= 
\frac{(\mu+i)!}{(2\mu)!} 
\sum_{p=0}^{\mu-i} 
\begin{pmatrix}
\mu-i 
\\ 
p
\end{pmatrix}
&\prod_{a=0}^{p-1} (j-\mu+a) 
\prod_{a=p+i+1}^{\mu}(j+a) 
\\
& \times 
\frac{(\lambda+q)!}{(\lambda+q-p)!}
\frac{(\lambda-q)!}{(\lambda-q-\mu+j+p)!}.
\end{align*}
\end{lemma}

\section{Irreducible $(\gK)$-modules of $SL(3, \R)$}
\label{section:irreducible modules. SL(3)}

From this section to section \ref{section:main results for SL(3)}, 
we set $G = SL(3, \R)$ and 
investigate the socle filtrations of 
the principal series modules for $G$. 
In this section, 
we review the classification of 
the irreducible $(\gK)$-modules of $G$. 

$G = SL(3, \R)$ has two conjugacy classes of Cartan subgroups. 
One is split and the other is fundamental. 
Let $H_{s} = M_{m} A_{m}$ 
be the split Cartan subgroup, 
where 
\begin{align*}
& 
A_{m} 
:= 
\{\diag(a_{1}, a_{2}, a_{3}) \in G \mid a_{i} > 0\},
\\
& 
M_{m} 
:= Z_{K}(A_{m}) 
= \{I, m_{-1,1}^{sl}, m_{1,-1}^{sl}, m_{-1,-1}^{sl}\}, 
\quad 
m_{i,j}^{sl} := \diag(i,j,ij). 
\end{align*}
The Weyl group $W(G, H_{s})$ is isomorphic to 
$W(\lie{g}, \lie{h}_{s}) \simeq \mathfrak{S}_{3}$, 
which acts on $H_{s}$ by the permutation 
\begin{align}
& 
s \cdot \diag(h_{1}, h_{2}, h_{3}) 
= 
\diag(h_{s^{-1}(1)}, h_{s^{-1}(2)}, h_{s^{-1}(3)}), 
\label{eq:action of W on h_s}
\\
& \qquad 
s \in \mathfrak{S}_{3},  
\quad 
\diag(h_{1}, h_{2}, h_{3}) \in H_{s}. 
\notag
\end{align}

Let $\widehat{M_{m}}$ be the set of equivalence classes of 
irreducible representations of $M_{m}$. 
Define $\sigma_{i,j} \in \widehat{M_{m}}$, $i,j \in \Z$, by 
\[
\sigma_{i,j}(m_{-1,1}^{sl}) = (-1)^{i}, 
\qquad 
\sigma_{i,j}(m_{1,-1}^{sl}) = (-1)^{j}. 
\]
Then $\widehat{M_{m}}$ consists of four elements $\sigma_{i,j}$, 
$i,j \in \{0,1\}$. 
The action of $W(\lie{g},\lie{h}_{s}) \simeq \mathfrak{S}_{3}$ 
on $\widehat{M_{m}}$ is given by 
\begin{align}
& 
r_{1,2} \cdot \sigma_{i,j} = \sigma_{j,i}, 
&
&
r_{2,3} \cdot \sigma_{i,j} = \sigma_{i+j,j}, 
\label{eq:action of W on hat{M}}
\end{align}
where, $r_{p,q}$ is the permutation of $p$ and $q$. 

Let $f_{i}$ be the elements of $\lie{h}_{s}^{\ast}$ defined by 
\begin{equation}\label{eq:e_i, SL(3)}
f_{i}(\diag(x_{1}, x_{2}, x_{3})) = x_{i}, 
\quad 
\diag(x_{1}, x_{2}, x_{3}) \in \lie{h}_{s}. 
\end{equation}
Then the root system $\Delta(\lie{g}, \lie{h}_{s})$ is 
$\{f_{i} - f_{j} \mid 1 \leq i \not= j \leq 3\}$. 

In this paper, we specify regular characters by the numbers 
used in ATLAS \cite {Atlas}. 

Let 
\begin{align}
& 
\Lambda = \sum_{i=1}^{3} \Lambda_{i} f_{i}, 
& 
& 
\Lambda_{1}-\Lambda_{2}, \ \Lambda_{2}-\Lambda_{3} \in \Z_{>0}
\label{eq:Lambda}
\end{align} 
be a nonsingular integral infinitesimal character. 
Then there are four conjugacy classes of regular characters of $H_{m}$ 
with the infinitesimal character $\Lambda$. 
These four correspond to the ATLAS number "3", "4", "5" in the block $PU(2,1)$ 
and "0" in the block $PU(3)$. 
According to these numbers, 
we define a regular character $\gamma_{3}$, $\gamma_{4}$, $\gamma_{5}$ 
and $\gamma_{0'}$ by 
\begin{align}
\gamma_{3} 
&= 
(\sigma_{\Lambda_{1}-\Lambda_{3}, \Lambda_{2}-\Lambda_{3}+1}, 
(\Lambda_{1},\Lambda_{2},\Lambda_{3})) 
\sim 
(\sigma_{\Lambda_{2}-\Lambda_{3}+1, \Lambda_{1}-\Lambda_{3}}, 
(\Lambda_{2},\Lambda_{1},\Lambda_{3})) 
\notag
\\
&
\sim 
(\sigma_{\Lambda_{1}-\Lambda_{2}+1, \Lambda_{1}-\Lambda_{3}}, 
(\Lambda_{2},\Lambda_{3},\Lambda_{1})) 
\sim 
(\sigma_{\Lambda_{1}-\Lambda_{3}, \Lambda_{1}-\Lambda_{2}+1}, 
(\Lambda_{3},\Lambda_{2},\Lambda_{1})) 
\label{eq:gamma_3}
\\
&
\sim 
(\sigma_{\Lambda_{2}-\Lambda_{3}+1, \Lambda_{1}-\Lambda_{2}+1}, 
(\Lambda_{3},\Lambda_{1},\Lambda_{2})) 
\sim 
(\sigma_{\Lambda_{1}-\Lambda_{2}+1, \Lambda_{2}-\Lambda_{3}+1}, 
(\Lambda_{1},\Lambda_{3},\Lambda_{2})), 
\notag
\end{align}
\begin{align}
\gamma_{4} 
&= 
(\sigma_{\Lambda_{2}-\Lambda_{3}, \Lambda_{1}-\Lambda_{3}+1}, 
(\Lambda_{2},\Lambda_{1},\Lambda_{3})) 
\sim 
(\sigma_{\Lambda_{1}-\Lambda_{3}+1, \Lambda_{2}-\Lambda_{3}}, 
(\Lambda_{1},\Lambda_{2},\Lambda_{3})) 
\notag
\\
&
\sim 
(\sigma_{\Lambda_{1}-\Lambda_{2}+1, \Lambda_{2}-\Lambda_{3}}, 
(\Lambda_{1},\Lambda_{3},\Lambda_{2})) 
\sim 
(\sigma_{\Lambda_{2}-\Lambda_{3}, \Lambda_{1}-\Lambda_{2}+1}, 
(\Lambda_{3},\Lambda_{1},\Lambda_{2})) 
\label{eq:gamma_4}
\\
&
\sim 
(\sigma_{\Lambda_{1}-\Lambda_{3}+1, \Lambda_{1}-\Lambda_{2}+1}, 
(\Lambda_{3},\Lambda_{2},\Lambda_{1})) 
\sim 
(\sigma_{\Lambda_{1}-\Lambda_{2}+1, \Lambda_{1}-\Lambda_{3}+1}, 
(\Lambda_{2},\Lambda_{3},\Lambda_{1})), 
\notag
\end{align}
\begin{align}
\gamma_{5} 
&= 
(\sigma_{\Lambda_{1}-\Lambda_{2}, \Lambda_{2}-\Lambda_{3}+1}, 
(\Lambda_{1},\Lambda_{3},\Lambda_{2})) 
\sim 
(\sigma_{\Lambda_{2}-\Lambda_{3}+1, \Lambda_{1}-\Lambda_{2}}, 
(\Lambda_{3},\Lambda_{1},\Lambda_{2})) 
\notag
\\
& 
\sim 
(\sigma_{\Lambda_{1}-\Lambda_{3}+1, \Lambda_{1}-\Lambda_{2}}, 
(\Lambda_{3},\Lambda_{2},\Lambda_{1})) 
\sim 
(\sigma_{\Lambda_{1}-\Lambda_{2}, \Lambda_{1}-\Lambda_{3}+1}, 
(\Lambda_{2},\Lambda_{3},\Lambda_{1})) 
\label{eq:gamma_5}
\\
& 
\sim 
(\sigma_{\Lambda_{2}-\Lambda_{3}+1, \Lambda_{1}-\Lambda_{3}+1}, 
(\Lambda_{2},\Lambda_{1},\Lambda_{3})) 
\sim 
(\sigma_{\Lambda_{1}-\Lambda_{3}+1, \Lambda_{2}-\Lambda_{3}+1}, 
(\Lambda_{1},\Lambda_{2},\Lambda_{3})), 
\notag
\end{align}
and 
\[
\gamma_{0'} 
= 
(\sigma_{\Lambda_{1}-\Lambda_{3}, \Lambda_{2}-\Lambda_{3}}, 
(\Lambda_{1},\Lambda_{2},\Lambda_{3})). 
\]
Here, $\sim$ means the $K$-conjugacy, 
and we denoted regular characters in such a way as  
$\gamma=(\sigma, \nu) \in \widehat{M_{m}} \times \lie{a}_{m}^{\ast}$, 
since a regular character of $H_{m}$ is determined by 
its restriction to $M_{m}$ and $\lie{a}_{m}$. 
Such notation will be applied to regular characters of other Cartan subgroups. 
The lengths of these regular characters are all two. 

Next, consider the fundamental Cartan subgroup. 
Let 
\begin{align*}
& 
H_{f} = T_{f} A_{f}, 
\\
& 
T_{f} 
:= 
\left\{
\left. 
\begin{pmatrix}
\cos \theta & -\sin \theta & 
\\
\sin \theta & \cos \theta & 
\\
& & 1 
\end{pmatrix}
\right| 
\theta \in \R
\right\}, 
\quad 
A_{f} 
:= 
\{\diag(a,a,a^{-2}) \mid a > 0\}
\end{align*}
be a fundamental Cartan subgroup of $G$. 
Define $L_{f}$ to be the centralizer of $A_{f}$ in $G$. 
Then 
\[
L_{f} 
= 
\left\{
\left. 
\begin{pmatrix}
A & 
\\
& (\det A)^{-1} 
\end{pmatrix}
\right| 
A \in GL(2,\R)
\right\}. 
\]
Therefore, $L_{f} = M_{f} A_{f}$, where 
\[
M_{f} = 
\left\{
\left. 
\begin{pmatrix}
A & 
\\
& (\det A)^{-1} 
\end{pmatrix}
\right| 
A \in SL^{\pm}(2,\R)
\right\}. 
\]
The order of the Weyl group $W(G, H_{f})$ is two, 
and this group is generated by $\diag(1,-1,-1)$. 
Since $H_{f}$ is connected, 
there are 
$\# W(\lie{g}, \lie{h}_{f}) / W(G, H_{f}) 
= 3$ equivalence classes of 
irreducible $(\gK)$-modules 
with a given integral infinitesimal character, 
which correspond to the ATLAS numbers "0", "1" and "2" in the block 
$PU(2,1)$. 
According to these  numbers, 
we define regular characters $\gamma_{0}$, $\gamma_{1}$ and $\gamma_{2}$ 
of $H_{f}$ by 
\begin{align*}
\gamma_{0} 
&= 
(\Lambda_{1}-\Lambda_{3}+1,\Lambda_{1}+\Lambda_{3}-2\Lambda_{2}) 
\in \widehat{SO(2)} \times \lie{a}_{f}^{\ast}, 
\\
\gamma_{1} 
&= 
(\Lambda_{1}-\Lambda_{2}+1,\Lambda_{1}+\Lambda_{2}-2\Lambda_{3}) 
\in \widehat{SO(2)} \times \lie{a}_{f}^{\ast}, 
\\
\gamma_{2} 
&= 
(\Lambda_{2}-\Lambda_{3}+1,\Lambda_{2}+\Lambda_{3}-2\Lambda_{1}) 
\in \widehat{SO(2)} \times \lie{a}_{f}^{\ast}. 
\end{align*}
Here, we denoted the restriction of $\gamma_{j}$ to $T_{f} \simeq SO(2)$ 
by its differential. 
The length of $\gamma_{0}$ is zero and those of $\gamma_{1}, \gamma_{2}$ 
are one. 

The composition factors of $X(\gamma_{i})$, $i=1,\dots,5,0'$ 
are known by KLV-conjecture. 

\begin{theorem}
\label{theorem:composition factors for SL(3)}
In the Grothendieck group, $X(\gamma_{i})$, $i=1,\dots,5,0'$, decomposes 
into irreducible modules as follows: 
\begin{align*}
&
X(\gamma_{0}) = \overline{X}(\gamma_{0}), 
\qquad
X(\gamma_{1}) = \overline{X}(\gamma_{0}) + \overline{X}(\gamma_{1}), 
\qquad
X(\gamma_{2}) = \overline{X}(\gamma_{0}) + \overline{X}(\gamma_{2}), 
\\
&
X(\gamma_{3}) 
= 
\overline{X}(\gamma_{0})+\overline{X}(\gamma_{1})
+\overline{X}(\gamma_{2})+\overline{X}(\gamma_{3}), 
\\
&
X(\gamma_{4})
= 
\overline{X}(\gamma_{0})+\overline{X}(\gamma_{1})+\overline{X}(\gamma_{4}), 
\qquad
X(\gamma_{5})
= 
\overline{X}(\gamma_{0})+\overline{X}(\gamma_{2})+\overline{X}(\gamma_{5}), 
\\
& X(\gamma_{0'}) = \overline{X}(\gamma_{0'}). 
\end{align*}
\end{theorem}

For $G=SL(3,\R)$, the automorphism $\mu$ 
in Section~\ref{subsection:horizontal symmetry} is given by the involution 
\[
\mu : G \to G, 
\qquad 
\mu(g) = {}^{t}g^{-1}. 
\]
This automorphism stabilizes the Cartan subgroups 
$H_{m}$ and $H_{f}$. 
For two regular characters $\gamma$ and $\gamma'$, 
we write $\gamma \dot{\sim} \gamma'$ if $X(\gamma)$ and $X(\gamma')$ 
are in the same coherent family. 
Then by easy calculation, 
$\gamma_{0} \circ \mu \dot{\sim} \gamma_{0}$, 
$\gamma_{1} \circ \mu \dot{\sim} \gamma_{2}$, 
$\gamma_{3} \circ \mu \dot{\sim} \gamma_{3}$ 
and $\gamma_{4} \circ \mu \dot{\sim} \gamma_{5}$. 
It follows that 
\begin{lemma}
\label{lemma:quasi-dual modules, SL(3)}
$\overline{X}(\gamma_{0})$ and $\overline{X}(\gamma_{3})$ are quasi-self dual. 
$\overline{X}(\gamma_{1})$ is quasi-dual to $\overline{X}(\gamma_{2})$, 
and  
$\overline{X}(\gamma_{4})$ to $\overline{X}(\gamma_{5})$. 
\end{lemma}

Put $\mu' = w^{\circ} \mu (w^{\circ})^{-1}$, 
where $w^{\circ}$ is the longest element of the Weyl group $W(G,A_{m})$. 
This automorphism stabilizes $K$, $M_{m}$, $A_{m}$ and $N_{m}$ 
respectively. 
For $\sigma_{i,j} \in \widehat{M_{m}}$ and 
$\nu=(\nu_{1},\nu_{2},\nu_{3}) \in \lie{a}^{\ast}$, 
\begin{align*}
&
\sigma_{i,j} \circ (\mu')^{-1} 
= \sigma_{i,i+j}, 
&
&
\nu \circ (\mu')^{-1} = (-\nu_{3},-\nu_{2},-\nu_{1}) 
\in \lie{a}^{\ast}. 
\end{align*}
It follows that 
\begin{equation}
\label{eq:principal series twisted by mu}
I(\sigma_{i,j}, (\nu_{1},\nu_{2},\nu_{3}))^{\mu}
\simeq
I(\sigma_{i,i+j}, (-\nu_{3},-\nu_{2},-\nu_{1})). 
\end{equation}

\section{$K$-types of the standard modules of $SL(3,\R)$}
\label{section:K-spectra, SL(3)} 

For our calculation, we need the $K$-spectra of the irreducible modules 
$\overline{X}(\gamma_{i})$, $i=0,1,\dots,5$. 
In fact, we only need the information of the minimal $K$-types 
of $\overline{X}(\gamma_{i})$, $i=0,1,2$ for general $\Lambda$ and 
that of the $K$-spectra of $\overline{X}(\gamma_{i})$, $i=0,1,\dots,5$, 
for the trivial infinitesimal character $\Lambda = \rho_{m} = (1,0,-1)$. 

Before going ahead, 
we fix the identification 
$\lie{k} = \lie{so}(3) \otimes \C \simeq \lie{sl}_{2}$. 

Let $A_{ij} = E_{ij} - E_{ji}$ and $S_{ij} := E_{ij} + E_{ji}$. 
The commutation relations of them are 
$[A_{ij}, A_{jk}] = A_{ik}$, 
$[A_{ij}, S_{jk}] = S_{ik}$ for $i \not= j \not= k \not= i$, 
$[A_{ij}, S_{ij}] = 2(E_{ii} - E_{jj})$ and 
$[A_{ij}, E_{jj}] = S_{ij}$.

As a basis of $\lie{k}$, choose 
\begin{align}
& 
D := \I A_{21}, 
&
& 
Y_{+} := \I A_{32} + A_{31}, 
&
&
Y_{-} := \I A_{32} - A_{31}. 
\label{eq:sl_2 triple for so(3)}
\end{align}
Then they satisfy 
$[2D,Y_{+}] = 2Y_{+}$, 
$[2D,Y_{-}] = - 2Y_{-}$ 
and $[Y_{+},Y_{-}] = 2D$, 
namely $\{2D, Y_{+}, Y_{-}\}$ is a standard $\lie{sl}_{2}$-triple. 
Then, the irreducible representation $V_{\lambda}^{SO(3)}$, 
where $\lambda$ is the highest weight with respect to $D$, is 
identified with $V_{2\lambda}^{\lie{sl}_{2}}$ as a $SO(3)$-module. 
Under this identification 
$\lie{k} \simeq \lie{sl}_{2} = \lie{su}(2) \otimes \C$, 
the surjection map 
$SU(2) \rightarrow SO(3)$ sends the elements 
$\begin{pmatrix}
0 & -1 \\ 1 & 0 
\end{pmatrix}$ 
and 
$\begin{pmatrix}
0 & \I \\ \I & 0 
\end{pmatrix}$ in $SU(2)$ 
to $m_{-1,1}$ and $m_{1,-1}$ in $SO(3)$, respectively. 
It follows from Lemma~\ref{lemma:the action of elements on repn of SU(2)} 
that the elements 
$m_{-1,1,}, m_{1,-1} \in M_{m} \subset SO(3)$ act on 
$v_{q}^{\lambda} \in V_{\lambda}^{SO(3)}$ by 
\begin{align}
& 
\tau_{\lambda}(m_{-1,1}) v_{q}^{\lambda} 
= 
(-1)^{\lambda-q} v_{-q}^{\lambda}
&
&
\mbox{and} 
&
&
\tau_{\lambda}(m_{1,-1}) v_{q}^{\lambda} 
= (-1)^{\lambda} v_{-q}^{\lambda}. 
\label{eq:action of M, SL(3)} 
\end{align}

Firstly, consider the standard modules $X(\gamma_{i})$, $i=0,1,2$. 
Let $\pi_{DS}^{SL_{2}^{\pm}}(\mu)$ be the $(\gK)$-module of 
the discrete series representation of $M_{f} \simeq SL^{\pm}(2,\R)$ 
with the minimal $K$-type $V_{\mu}^{O(2)}$. 
Let 
\[
P_{f} = L_{f} N_{f} = M_{f} A_{f} N_{f} 
\]
be a parabolic subgroup of $G$. 
We choose an appropriate basis of $\lie{a}_{f}$, 
and we identify an element of $\lie{a}_{f}$ 
with a complex number. 
Then, the standard module $X(\gamma_{i})$, $i=0,1,2$, is 
the space of $K$-finite vectors of 
\begin{align*}
& 
\Ind_{P_{f}}^{G}(\pi_{DS}^{SL_{2}^{\pm}}(\Lambda_{p}-\Lambda_{q}+1) 
\otimes e^{\Lambda_{p}+\Lambda_{q}-2\Lambda_{r}+3} 
\otimes 1), 
&
&
(p,q,r) 
= 
\left\{
\begin{matrix}
(1,3,2) & (i=0), 
\\
(1,2,3) & (i=1), 
\\
(2,3,1) & (i=2). 
\end{matrix}
\right.
\end{align*}

From this, we know that the $K \cap M_{f} =O(2)$-spectrum of 
$\pi_{DS}^{SL_{2}^{\pm}}(\mu)$ is 
\[
\{V_{\mu+2k}^{O(2)} \mid k = 0,1,2,\dots\},
\]
and the multiplicities of these are all one. 
By Frobenius reciprocity, 
the $K$-spectra of $X(\gamma_{i})$, $i=0,1,2$, are 
\begin{align}
& 
\widehat{K}(X(\gamma_{i})) 
= 
\Z_{\geq \Lambda_{p}-\Lambda_{q}+1}, 
&
&
m_{\lambda} 
= 
\lfloor 
\frac{\lambda-\Lambda_{p}+\Lambda_{q}+1}{2}
\rfloor, 
\label{eq:K-spectra, X(gamma_i), i=0,1,2}
\end{align}
where, 
$(p,q) 
= 
(1,3)$ for $i=0$, 
$(1,2)$ for $i=1$ and 
$(2,3)$ for $i=2$, respectively, 
and $m_{\lambda}$ is the multiplicity of $V_{\lambda}^{SO(3)}$ 
in $X(\gamma_{i})$. 
Note that the minimal $K$-types of $\overline{X}(\gamma_{0})$, 
$\overline{X}(\gamma_{1})$ and $\overline{X}(\gamma_{2})$ 
are 
$V_{\Lambda_{1}-\Lambda_{3}+1}^{SO(3)}$, 
$V_{\Lambda_{1}-\Lambda_{2}+1}^{SO(3)}$ and 
$V_{\Lambda_{2}-\Lambda_{3}+1}^{SO(3)}$, respectively, 
and the multiplicities of them are all one.

For the standard modules $X(\gamma_{i})$, $i=3,4,5$, 
we need the information of $K$-types 
when the infinitesimal character is trivial. 
In this case, 
\begin{align*}
X(\gamma_{i}) 
= \Ind_{P_{m}}^{G}(\sigma_{p,q} \otimes e^{\rho_{m}+\rho_{m}} \otimes 1), 
&
&
(p,q) 
= 
\left\{ 
\begin{matrix}
(0,0) & (i=3), 
\\
(1,1) & (i=4), 
\\
(1,0) & (i=5). 
\end{matrix}
\right. 
\end{align*}

Let us calculate the $K=SO(3)$ spectra of these modules. 
By \eqref{eq:action of M, SL(3)}  and the Frobenius reciprocity, 
the $K$-spectra of $X(\gamma_{i})$, $i=3,4,5$, are as follows:  
\begin{align*}
& 
\widehat{K}(X(\gamma_{3})) 
= 
\Z_{\geq 0}, 
\qquad 
m_{\lambda} 
= 
\begin{cases}
\frac{\lambda}{2}+1 \quad \mbox{(if $\lambda$ is even)}
\\
\frac{\lambda-1}{2} \quad \mbox{(if $\lambda$ is odd)}
\end{cases}, 
\\
& 
\widehat{K}(X(\gamma_{4})) 
=
\widehat{K}(X(\gamma_{5})) 
= 
\Z_{\geq 1}, 
\qquad 
m_{\lambda} 
= 
\begin{cases}
\frac{\lambda}{2} \quad \mbox{(if $\lambda$ is even)}
\\
\frac{\lambda+1}{2} \quad \mbox{(if $\lambda$ is odd)}
\end{cases}.
\end{align*}

From these informations on the $K$-spectra of standard modules 
$X(\gamma_{i})$, $i=0,1,\dots,5$, 
and Theorem~\ref{theorem:composition factors for SL(3)}, 
we obtain the following lemma. 

\begin{lemma}
\label{lemma:K-spectra of irreducibles, SL(3)}
\begin{enumerate}
\item
If the nonsingular infinitesimal character is 
$\Lambda = (\Lambda_{1}, \Lambda_{2}, \Lambda_{3})$, 
$\Lambda_{1}-\Lambda_{2}, \Lambda_{2}-\Lambda_{3} \in \Z_{>0}$, 
then the minimal $K$-types of the irreducible modules 
$\overline{X}(\gamma_{i})$, $i=0,1,2$, are 
$\Lambda_{1}-\Lambda_{3}+1$, $\Lambda_{1}-\Lambda_{2}+1$ 
and $\Lambda_{2}-\Lambda_{3}+1$, respectively. 
\item
If the infinitesimal character $\Lambda$ is trivial, 
namely $\Lambda = \rho_{m} =(1,0,-1)$, 
then the multiplicities of small $K$-types in 
the irreducible modules are as follows; 
\begin{center}
\begin{tabular}{|c||c|c|c|c|c|c|}
\hline
$K$-type $\lambda$ & $0$ & $1$ & $2$ & $3$ & $4$ & $5$ 
\\\hline\hline
$\overline{X}(\gamma_{0})$ & $0$ & $0$ & $0$ & $1$ & $1$ & $2$ 
\\\hline
$\overline{X}(\gamma_{1})$ & $0$ & $0$ & $1$ & $0$ & $1$ & $0$ 
\\\hline
$\overline{X}(\gamma_{2})$ & $0$ & $0$ & $1$ & $0$ & $1$ & $0$ 
\\\hline
$\overline{X}(\gamma_{3})$ & $1$ & $0$ & $0$ & $0$ & $0$ & $0$ 
\\\hline
$\overline{X}(\gamma_{4})$ & $0$ & $1$ & $0$ & $1$ & $0$ & $1$ 
\\\hline
$\overline{X}(\gamma_{5})$ & $0$ & $1$ & $0$ & $1$ & $0$ & $1$ 
\\\hline
\end{tabular}
\end{center}
Here, each row exhibits the multiplicities of $V_{\lambda}^{SO(3)}$ 
in the irreducible module $\overline{X}(\gamma_{i})$. 
\end{enumerate}
\end{lemma}

Note that, if $\Lambda=\rho_{m}$, 
the irreducible module $\overline{X}(\gamma_{3})$ is 
the trivial representation, which is isomorphic
to $V_{0}^{SO(3)}$ as a $K$-module, 
and the minimal $K$-types of  $X(\gamma_{i})$, $i=4,5$, 
are $V_{1}^{SO(3)}$ with multiplicity one.


\section{Shift operators for $SL(3,\R)$}
\label{section:shift operators, SL(3)}

In this section, we write down the shift operators of $K$-types explicitly, 
in the case $G=SL(3,\R)$. 
Choose $A_{m}$ as in section~\ref{section:irreducible modules. SL(3)}. 
Define 
\begin{align*}
N_{m} &= \{
\begin{pmatrix}
1 & x_{12} & x_{13} 
\\
& 1 & x_{23} 
\\
& & 1 
\end{pmatrix} 
\mid x_{ij} \in \R\}. 
\end{align*}
Then 
$G=KA_{m} N_{m}$ is an Iwasawa decomposition of $G$.

The adjoint representation $(\lie{s}, \Ad_{\lie{s}})$ of $K$ 
is irreducible and 
$\lie{s} \simeq V_{2}^{SO(3)} = V_{4}^{\lie{sl}_{2}}$. 
By \eqref{eq:sl_2 triple for so(3)}, This isomorphism is given by 
\begin{align*}
& 
v_{\pm 2}^{2} 
= 
E_{11} - E_{22} \pm \I S_{12} 
= 
E_{11} - E_{22} \pm D \pm 2 \I E_{12}, 
\\
&
v_{\pm 1}^{2}
= 
\frac{1}{2}(\mp S_{13} - \I S_{23}) 
= -\frac{1}{2} Y_{\pm} 
\mp E_{13}-\I E_{23}, 
\\
&
v_{0}^{2}
= 
-\frac{1}{3} 
(E_{11} + E_{22} - 2 E_{33}). 
\end{align*}

Define an invariant bilinear form $\langle \ , \ \rangle$ on $\lie{g}$ 
by 
$\langle X, Y \rangle = \tr(XY)$. 
Then 
\begin{align*}
&
\langle 
v_{2}^{2}, 
v_{-2}^{2}
\rangle 
= 4, 
&
&
\langle 
v_{1}^{2}, v_{-1}^{2}
\rangle 
= 
-1, 
&
&
\langle 
v_{0}^{2}, v_{0}^{2}
\rangle 
= \frac{2}{3}. 
\end{align*}

The operator $\nabla$ defined in Section~\ref{section:outline} is 
given by 
\begin{align*}
4 \nabla 
\phi 
:= 
& 
L(E_{11} - E_{22} - \I S_{12}) \phi 
\otimes 
v_{2}^{2}
-
2 \, 
L(S_{13} - \I S_{23}) \phi 
\otimes 
v_{1}^{2} 
\\
&
- 2 \, 
L(E_{11} + E_{22} - 2 E_{33}) \phi 
\otimes 
v_{0}^{2} 
\\
& 
+ 2 \,
L(S_{13} + \I S_{23}) \phi 
\otimes 
v_{-1}^{2}
+ 
L(E_{11} - E_{22} + \I S_{12}) \phi 
\otimes 
v_{-2}^{2}.
\end{align*}

For $(\tau_{\lambda}, V_{\lambda}^{SO(3)}) \in \widehat{K}$ 
and $\nu = (\nu_{1}, \nu_{2}, \nu_{3}) \in \lie{a}_{m}^{\ast}$, 
we define 
\begin{align*}
C_{(\tau_{\lambda})^{\ast}}^{\infty}
&(K \backslash G / A_{m} N_{m}; 
e^{\nu+\rho_{m}})
\\
&=
\{\phi_{\lambda} 
: G \overset{C^{\infty}}{\longrightarrow} V_{(\tau_{\lambda})^{\ast}} 
\mid 
\\
& \hspace{15mm} 
\phi_{\lambda}(kan) 
= 
a^{-\nu-\rho_{m}} (\tau_{\lambda})^{\ast}(k)\, \phi_{\lambda}(e), 
\ 
k \in K, a \in A_{m}, n \in N_{m}\}. 
\end{align*}
Note that we may and do replace $(\tau_{\lambda})^{\ast}$ 
by $\tau_{\lambda}$, 
since $(\tau_{\lambda}, V_{\lambda}^{SO(3)})$ is self-dual. 

For $j=-2,-1,0,1,2$, let $\mathrm{pr}_{j}$ be the natural projection 
\[
V_{\lambda}^{SO(3)} \otimes \lie{s} 
\simeq 
\bigoplus_{j=-2}^{2} V_{\lambda+j}^{SO(3)} 
\rightarrow 
V_{\lambda+j}^{SO(3)}. 
\]
We define the shift operators of $K$-types $P_{j}$ ($j=-2, -1, 0, 1, 2$) by 
\[
P_{j} = \mathrm{pr}_{j} \circ \nabla : 
C^{\infty}_{\tau_{\lambda}}(K \backslash G / A_{m} N_{m}; e^{\nu+\rho_{m}}) 
\rightarrow 
C^{\infty}_{\tau_{\lambda+j}}(K \backslash G / A_{m} N_{m}; e^{\nu+\rho_{m}}). 
\]

For 
$\phi 
\in C_{\tau}^{\infty}
(K \backslash G / A_{m} N_{m}; e^{\nu+\rho_{m}})$, 
the action of elements of Lie algebra is as follows; 
\begin{align*}
&
(L(E_{ii}) \phi)(e) 
= 
\langle e_{i}, \nu + \rho_{m} \rangle \phi(e), 
\\
&
(L(X) \phi)(e) 
= 
-\tau_{\lambda}(X) \phi(e) \quad (X \in \lie{k}), 
\\
&
(L(Y) \phi)(e) = 0 \quad (Y \in \lie{n}). 
\end{align*}
Put 
\[
\phi_{\lambda}(kan) 
= a^{-\nu-\rho_{m}} 
\tau_{\lambda}(k)^{-1} 
\sum_{q=-\lambda}^{\lambda} 
c(q)\, v_{q}^{\lambda} 
\in 
C_{\tau_{\lambda}}^{\infty}
(K \backslash G / A_{m} N_{m}; e^{\nu+\rho_{m}}), 
\qquad 
c(q) \in \C. 
\]
For notational convenience, we regard $c(q)$ to be zero if $|q| > \lambda$. 
For $\alpha \in \C$ and non-negative integer $k$, 
we denote 
\[
(\alpha)_{\downarrow (k)} 
:= 
\alpha(\alpha-1)\cdots(\alpha-k+1) 
= 
\frac{\varGamma(\alpha+1)}{\varGamma(\alpha-k+1)}. 
\]
The shift operators of $K$-types for $G=SL(3,\R)$ 
are as follows. 
\begin{align}
(P_{2} \phi_{\lambda})(e) 
= 
\sum_{q=-\lambda-2}^{\lambda+2} 
\Big\{ 
&
\left(\nu_{1} - \nu_{2} + q - 1 \right)\, 
c(q-2)
- 
2 \left(\nu_{1} + \nu_{2} - 2 \nu_{3} + 2 \lambda + 3
\right)\, 
c(q) 
\notag
\\
&
\quad 
+ 
\left(\nu_{1} - \nu_{2} - q - 1 
\right)\, 
c(q+2) 
\Big\}
v_{q}^{\lambda+2}, 
\label{eq:P_2, SL(3)}
\end{align}
\begin{align}
(P_{1} \phi_{\lambda})(e) 
= 
\sum_{q=-\lambda-1}^{\lambda+1} 
\Big\{ 
&
(\lambda-q+2) 
\left(\nu_{1} - \nu_{2} + q - 1 
\right)\, 
c(q-2)
\notag
\\
& \quad 
+
2 q \left(\nu_{1} + \nu_{2} - 2 \nu_{3} + \lambda + 1 
\right)\, 
c(q)
\label{eq:P_1, SL(3)}
\\
& 
\quad 
- 
(\lambda+q+2) 
\left(\nu_{1} - \nu_{2} - q - 1 
\right)\, 
c(q+2) 
\Big\}
v_{q}^{\lambda+1}, 
\notag
\end{align}
\begin{align}
(P_{0} \phi_{\lambda})(e) 
= 
\sum_{q=-\lambda}^{\lambda} 
\Big\{ 
&
(\lambda-q+2)_{\downarrow (2)} 
\left(\nu_{1} - \nu_{2} + q - 1 
\right)\, 
c(q-2) 
\notag
\\
& 
\quad 
-  
\frac{2}{3} 
(3q^{2}-\lambda^{2}-\lambda) 
\left(\nu_{1} + \nu_{2} - 2 \nu_{3} \right)\, 
c(q)
\label{eq:P_0, SL(3)}
\\
& 
\quad 
+ 
(\lambda+q+2)_{\downarrow (2)} 
\left(\nu_{1} - \nu_{2} - q - 1 
\right)\, 
c(q+2)
\Big\} 
v_{q}^{\lambda}, 
\notag
\end{align}
\begin{align}
(P_{-1} \phi_{\lambda})(e) 
= 
\sum_{q=-\lambda+1}^{\lambda-1} 
\Big\{
&
(\lambda-q+2)_{\downarrow (3)} 
\left(\nu_{1} - \nu_{2} + q - 1 
\right)\, 
c(q-2) 
\notag
\\
& 
\quad 
- 
2 q (\lambda-q) (\lambda+q) 
\left(\nu_{1} + \nu_{2} - 2 \nu_{3} - \lambda 
\right)\, 
c(q) 
\label{eq:P_{-1}, SL(3)} 
\\
& \quad 
- 
(\lambda+q+2)_{\downarrow (3)} 
\left(\nu_{1} - \nu_{2} - q - 1 
\right)\, 
c(q+2)
\Big\} 
v_{q}^{\lambda-1}, 
\notag
\end{align}
\begin{align}
(P_{-2} \phi_{\lambda})(e) 
= 
\sum_{q=-\lambda+2}^{\lambda-2} 
\Big\{
& 
(\lambda-q+2)_{\downarrow (4)} 
\left(\nu_{1} - \nu_{2} + q - 1 
\right)\, 
c(q-2) 
\notag
\\
& 
\quad 
- 
2 (\lambda+q)_{\downarrow (2)} (\lambda-q)_{\downarrow (2)} 
\left(\nu_{1} + \nu_{2} - 2 \nu_{3} - 2 \lambda + 1
\right)\, 
c(q) 
\label{eq:P_{-2}, SL(3)}
\\
& 
\quad 
+ 
(\lambda+q+2)_{\downarrow (4)} 
\left(\nu_{1} - \nu_{2} - q - 1 
\right)\, 
c(q+2) 
\Big\}
v_{q}^{\lambda-2}. 
\notag
\end{align}

\section{Solutions of the system of equations 
$P_{-1} \phi_{\lambda} = 0$, $P_{-2} \phi_{\lambda} = 0$}
\label{section:P_{-1}=p_{-2}=0}

We will solve the system of equations $P_{-1} \phi_{\lambda}(e) = 0$, 
$P_{-2} \phi_{\lambda}(e) = 0$. 
By \eqref{eq:P_{-1}, SL(3)} and \eqref{eq:P_{-2}, SL(3)}, 
the system of equations $P_{-1} \phi_{\lambda}(e) = 0$, 
$P_{-2} \phi_{\lambda}(e) = 0$ is equivalent to 
\begin{align}
& 
(\lambda-q+2)_{\downarrow (2)} 
(\nu_{1} - \nu_{2} + q -1)\, c(q-2) 
\label{eq:c(q-2) and c(q)}\\
& 
- (\lambda+q)_{\downarrow (2)} 
(\nu_{1} + \nu_{2} - 2 \nu_{3} - 2 \lambda + q + 1)\, 
c(q) = 0, 
\notag
\\
& \qquad 
(-\lambda+2 \leq q \leq \lambda-1), 
\notag
\\
& 
(\lambda-q)_{(2)} 
(\nu_{1} + \nu_{2} - 2 \nu_{3} - 2 \lambda - q + 1)\, c(q) 
\label{eq:c(q) and c(q+2)}\\
& 
- (\lambda+q+2)_{\downarrow (2)} 
(\nu_{1} - \nu_{2} - q - 1)\, c(q+2) 
= 0 
\notag
\\
& \qquad 
(-\lambda+1 \leq q \leq \lambda-2). 
\notag
\end{align}
By replacing the parameter $q$ of \eqref{eq:c(q) and c(q+2)} by $q-2$, 
we get 
\begin{align}
& 
(\lambda-q+2)_{\downarrow (2)} 
(\nu_{1} + \nu_{2} - 2 \nu_{3} - 2 \lambda - q + 3)\, c(q-2) 
\notag
\\
& 
- (\lambda+q)_{\downarrow (2)} 
(\nu_{1} - \nu_{2} - q + 1)\, c(q) 
= 0 
\label{eq:c(q) and c(q+2)'}\\
& \qquad 
(-\lambda+3 \leq q \leq \lambda). 
\notag
\end{align}
Suppose $-\lambda+3 \leq \lambda-1$, i.e. $\lambda \geq 2$. 
If 
\begin{align*}
& 
\begin{vmatrix} 
\nu_{1} + \nu_{2} - 2 \nu_{3} - 2 \lambda - q + 3 
&
-(\nu_{1} - \nu_{2} - q + 1) 
\\
-(\nu_{1} - \nu_{2} + q - 1) 
&
\nu_{1} + \nu_{2} - 2 \nu_{3} - 2 \lambda + q + 1 
\end{vmatrix}
\\
&= 
4(\nu_{1} - \nu_{3} - \lambda + 1) 
(\nu_{2} - \nu_{3} - \lambda + 1) 
\\
& \not= 0, 
\end{align*}
namely if $\lambda \not= \nu_{1} - \nu_{3} + 1$ and 
$\lambda \not= \nu_{2} - \nu_{3} + 1$, 
then $c(q) = c(q-2) = 0$ for 
$q=-\lambda+3, -\lambda+4, \dots, \lambda-1$. 
Therefore, if $\lambda \geq 3$, 
then 
$c(q) = 0$ for $-\lambda+1 \leq q \leq \lambda-1$, 
and if $\lambda = 2$, then 
$c(1)=c(-1)=0$. 

Suppose first that $\lambda \geq 3$, 
$\lambda \not= \nu_{1} - \nu_{3} + 1$ and 
$\lambda \not= \nu_{2} - \nu_{3} + 1$.  
In this case, 
\[
(\nu_{1} - \nu_{2} - \lambda + 1)\, c(-\lambda) 
= (\nu_{1} - \nu_{2} - \lambda + 1)\, c(\lambda) 
= 0.
\]
So, if $\lambda \not= \nu_{1} - \nu_{2} + 1$, 
then $c(\lambda) = c(-\lambda) = 0$, 
and if $\lambda = \nu_{1} - \nu_{2} + 1$, 
then $c(\lambda)$ and $c(-\lambda)$ are arbitrary. 

Secondly, suppose that $\lambda=2$, 
$\lambda \not= \nu_{1} - \nu_{3} + 1$ and 
$\lambda \not= \nu_{2} - \nu_{3} + 1$. 
In this case, $c(-1) = c(1) = 0$, and 
$c(2), c(0), c(-2)$ satisfies 
\begin{align*}
& 
(\nu_{1} + \nu_{2} - 2 \nu_{3} - 3)\, c(0) 
- 
6 (\nu_{1} - \nu_{2} - 1)\, c(-2) = 0, 
\\
& 
(\nu_{1} + \nu_{2} - 2 \nu_{3} - 3)\, c(0) 
- 
6 (\nu_{1} - \nu_{2} - 1)\, c(2) = 0. 
\end{align*}
Since we are assuming $2 = \lambda \not= \nu_{2}-\nu_{3}+1$, 
if $\nu_{1} - \nu_{2} = 1$, 
then $\nu_{1} + \nu_{2} - 2 \nu_{3} - 3 
= (\nu_{1}-\nu_{2}-1)+2(\nu_{2}-\nu_{3}-1) \not= 0$. 
Therefore, $c(0) = 0$ and $c(2)$, $c(-2)$ are arbitrary. 
Note that the condition $\nu_{1}-\nu_{2}=1$ is equivalent to 
$\lambda = 2 = \nu_{1}-\nu_{2}+1$, 
and this result is the same as that of the above case for 
$\lambda \geq 3$ and 
$\lambda = \nu_{1}-\nu_{2}+1$. 
If $\nu_{1} - \nu_{2} \not= 1$, namely 
$\lambda = 2 \not= \nu_{1} - \nu_{2} + 1$, 
then 
$\ds 
c(2) = c(-2) 
= 
\frac{1}{6} 
\frac{\nu_{1} - \nu_{2} - 2 \nu_{3} - 3}{\nu_{1} - \nu_{2} - 1} 
c(0)$. 

We obtained 
\begin{proposition}\label{proposition:4.1}
Suppose $\lambda \not= \nu_{1} - \nu_{3} + 1$ and 
$\lambda \not= \nu_{2} - \nu_{3} + 1$. 
\begin{enumerate}
\item
If $\lambda = \nu_{1} - \nu_{2} + 1$, then 
the solution space of $P_{-1} \phi_{\lambda}(e) = 0$ and 
$P_{-2} \phi_{\lambda}(e) = 0$ is spanned by 
$v_{\lambda}^{\lambda} + (-1)^{i}v_{-\lambda}^{\lambda}$, $i=0,1$. 
These vectors are contained in 
$V_{\lambda}^{SO(3)}
(\sigma_{i,\lambda+i})$, 
respectively.  
\item
If $\lambda \geq 3$ and $\lambda\not= \nu_{1} - \nu_{2} + 1$, 
then 
the solution space of $P_{-1} \phi_{\lambda}(e) = 0$ and 
$P_{-2} \phi_{\lambda}(e) = 0$ is 
$\{0\}$. 
\item
If $\lambda = 2$ and $\lambda\not= \nu_{1} - \nu_{2} + 1$, 
then 
the solution space of $P_{-1} \phi_{\lambda}(e) = 0$ and 
$P_{-2} \phi_{\lambda}(e) = 0$ is spanned by 
$\ds 
\frac{1}{6} 
\frac{\nu_{1} + \nu_{2} - 2 \nu_{3} - 3}{\nu_{1} - \nu_{2} -1}
(v_{2}^{2} + v_{-2}^{2})
+ 
v_{0}^{2}$. 
This vector is contained in $V_{2}^{SO(3)}(\sigma_{0,0})$. 
\end{enumerate}
\end{proposition}

\subsection{The case $\lambda = \nu_{2} - \nu_{3} + 1$}

Next, suppose that $\lambda = \nu_{2} - \nu_{3} + 1$. 
Then \eqref{eq:c(q-2) and c(q)} is 
\begin{align}
(\nu_{1} - \nu_{2} + q - 1) 
& \times 
\left\{
(\lambda - q + 2)_{\downarrow (2)} c(q-2) 
- (\lambda + q)_{\downarrow (2)} c(q) 
\right\} = 0, 
\notag
\\
& (-\lambda + 2 \leq q \leq \lambda - 1), 
\notag
\end{align}
and \eqref{eq:c(q) and c(q+2)'} is 
\begin{align}
(\nu_{1} - \nu_{2}  - q + 1) 
&\times 
\left\{
(\lambda - q + 2)_{\downarrow (2)} c(q-2) 
- (\lambda + q)_{\downarrow (2)} c(q) 
\right\} = 0, 
\notag
\\
& (-\lambda + 3 \leq q \leq \lambda).
\notag
\end{align}
Since $\nu$ is regular, $\nu_{1}$ and $ \nu_{2}$ are different, 
so if $\nu_{1} - \nu_{2} + q - 1$ is zero, 
then $\nu_{1} - \nu_{2} - q + 1$ is not zero. 
It follows that \eqref{eq:c(q-2) and c(q)} and 
\eqref{eq:c(q) and c(q+2)'} are equivalent to 
\begin{align} 
& (\nu_{1} - \nu_{2} - \lambda + 1) 
\left\{
1 \cdot 2\,  c(\lambda-2) 
- 
2 \lambda (2 \lambda - 1)\, c(\lambda) 
\right\} = 0, 
\label{eq:nu2-nu3+1-1}
\\
& (\lambda - q + 2)_{\downarrow (2)} c(q-2) 
= (\lambda + q)_{\downarrow (2)} c(q), 
\quad 
(-\lambda + 3 \leq q \leq \lambda - 1), 
\label{eq:nu2-nu3+1-2}
\\
& (\nu_{1} - \nu_{2} - \lambda + 1) 
\left\{
2\lambda (2\lambda - 1)\, c(-\lambda) 
- 1 \cdot 2\, c(-\lambda + 2) 
\right\} = 0. 
\label{eq:nu2-nu3+1-3}
\end{align}
Therefore, if $\lambda \not= \nu_{1} - \nu_{2} + 1$, then 
\[
c(q-2k) 
= 
\frac{(\lambda+q)_{\downarrow (2k)}}{(\lambda-q+2k)_{\downarrow (2k)}}
c(q), 
\]
so 
\begin{align}
& 
c(\lambda - 2k) 
= 
\frac{(2\lambda)_{\downarrow (2k)}}{(2k)!} c_{0}, 
&
& (0 \leq k \leq \lambda), 
\label{eq:c(lambda-2k), SL(3), l=nu2-nu3+1}
\\
& 
c(\lambda - 1 - 2k) 
= 
\frac{(2\lambda-1)_{\downarrow (2k)}}{(2k+1)!} c_{1}, 
&
& (0 \leq k \leq \lambda-1) 
\label{eq:c(lambda-1-2k), SL(3), l=nu2-nu3+1}
\end{align}
for some constants $c_{0}, c_{1}$. 
If $\lambda = \nu_{1} - \nu_{2} + 1$, then 
$c(\lambda)$ and $c(-\lambda)$ are arbitrary and 
other terms are given by 
\eqref{eq:c(lambda-2k), SL(3), l=nu2-nu3+1} 
and 
\eqref{eq:c(lambda-1-2k), SL(3), l=nu2-nu3+1}. 

\begin{proposition}\label{proposition:4.2}
\begin{enumerate}
\item
If $\lambda = \nu_{2} - \nu_{3} + 1$ and 
$\lambda \not= \nu_{1} - \nu_{2} + 1$, 
then the solution space of $P_{-1} \phi_{\lambda}(e) = 0$ and 
$P_{-2} \phi_{\lambda}(e) = 0$ is spanned by the two vectors 
$\sum_{k=0}^{\lambda-i} 
\frac{(2\lambda-i)_{\downarrow (2k)}}{(2k+i)!} 
v_{\lambda-i-2k}^{\lambda}$, 
$i=0,1$. 
These vectors are contained in 
$V_{\lambda}^{SO(3)}(\sigma_{i,\lambda})$, respectively.  
\item
If $\lambda = \nu_{2} - \nu_{3} + 1$ and 
$\lambda = \nu_{1} - \nu_{2} + 1$, 
then the solution space of $P_{-1} \phi_{\lambda}(e) = 0$ and 
$P_{-2} \phi_{\lambda}(e) = 0$ is spanned by the four vectors 
$v_{\lambda}^{\lambda}
+(-1)^{i} 
v_{-\lambda}^{\lambda}, 
\sum_{k=0}^{\lambda-i} 
\frac{(2\lambda-i)_{\downarrow (2k)}}{(2k+i)!}
v_{\lambda-i-2k}^{\lambda}$, 
$i=0,1$. 
These vectors are contained in 
$V_{\lambda}^{SO(3)}(\sigma_{i,\lambda+i})$ and 
$V_{\lambda}^{SO(3)}(\sigma_{i,\lambda})$, 
respectively. 
\end{enumerate}
\end{proposition}

\subsection{The case $\lambda = \nu_{1} - \nu_{3} + 1$.} 

Let us consider the case $\lambda = \nu_{1} - \nu_{3} + 1$. 
Note that $\lambda$ is not equal to $\nu_{1} - \nu_{2} + 1$ nor 
to $\nu_{2} - \nu_{3} + 1$ because of the nonsingularity of $\nu$.  

In this case, 
\begin{align*}
& 
\nu_{1} + \nu_{2} - 2 \nu_{3} - 2 \lambda + q + 1
= - \nu_{1} + \nu_{2} + q - 1
\quad \mbox{and} 
\\
& 
\nu_{1} + \nu_{2} - 2 \nu_{3} - 2 \lambda - q + 3
= - \nu_{1} + \nu_{2} - q + 1. 
\end{align*}
It follows that the equations 
\eqref{eq:c(q-2) and c(q)} and \eqref{eq:c(q) and c(q+2)'} are the same, 
so the conditions are 
\begin{align}
(\lambda - q + 2)_{\downarrow (2)}
& (\nu_{1} - \nu_{2} + q - 1)\, c(q-2) 
\notag
\\
& + (\lambda+q)_{\downarrow (2)} (\nu_{1} - \nu_{2} - q + 1)\, 
c(q) = 0, 
\label{eq:nu1-nu3+1}
\\
& (-\lambda + 2 \leq q \leq \lambda). 
\notag
\end{align}
As far as the coefficients in \eqref{eq:nu1-nu3+1} are non-zero, 
$c(q)$ and $c(q')$ with $q \equiv q' \ \modulo 2$ satisfy
\begin{align}
\frac{c(q)}{c(q')} 
&= 
\frac{\varGamma(\frac{\nu_{1}-\nu_{2}+1+q}{2})}
{\varGamma(\frac{\nu_{1}-\nu_{2}+1+q'}{2})} 
\frac{\varGamma(\frac{\nu_{2}-\nu_{1}+1+q'}{2})}
{\varGamma(\frac{\nu_{2}-\nu_{1}+1+q}{2})} 
\frac{(\lambda+q')! (\lambda-q')!}{(\lambda+q)! (\lambda-q)!} 
\label{eq:c(q) and c(q-2k)-2}
\\
&= 
(-1)^{(q-q')/2}
\frac{\varGamma(\frac{\nu_{2}-\nu_{1}+1+q'}{2})}
{\varGamma(\frac{\nu_{2}-\nu_{1}+1+q}{2})} 
\frac{\varGamma(\frac{\nu_{2}-\nu_{1}+1-q'}{2})}
{\varGamma(\frac{\nu_{2}-\nu_{1}+1-q}{2})} 
\frac{(\lambda+q')! (\lambda-q')!}{(\lambda+q)! (\lambda-q)!}. 
\label{eq:c(q) and c(q-2k)-3}
\end{align}
Some of the coefficients in \eqref{eq:nu1-nu3+1} can be 
zero only if $q - \nu_{1} + \nu_{2} \equiv 1 \ (\modulo 2)$. 
It follows that 
\begin{align}
& 
\sum_{
-\lambda \leq q \leq \lambda 
\atop 
q - \nu_{1}+\nu_{2} \equiv 0 \ \modulo 2
}
\frac{\varGamma(\frac{\nu_{1}-\nu_{2}+1+q}{2})}
{\varGamma(\frac{\nu_{2}-\nu_{1}+1+q}{2})}
\frac{1}{(\lambda+q)! (\lambda-q)!}
v_{q}^{\lambda} 
\label{eq:a solution of the eqn nu1-nu3+1-1}
\end{align}
is a solution of \eqref{eq:nu1-nu3+1}. 
If all the coefficients in \eqref{eq:nu1-nu3+1} are non-zero, 
\begin{align}
& 
\sum_{-\lambda \leq q \leq \lambda 
\atop q - \nu_{1}+\nu_{2} \equiv 1 \ \modulo 2}
\frac{(-1)^{q/2}}
{\varGamma(\frac{\nu_{2}-\nu_{1}+1+q}{2})
\varGamma(\frac{\nu_{2}-\nu_{1}+1-q}{2})}
\frac{1}{(\lambda+q)! (\lambda-q)!}
v_{q}^{\lambda} 
\label{eq:a solution of the eqn nu1-nu3+1-2}
\end{align}
is another solution. 
Here, for later convenience, 
we wrote this solution in a way different from 
\eqref{eq:a solution of the eqn nu1-nu3+1-1}. 

Consider the case when $\nu_{1}-\nu_{2} \pm q + 1$ can be zero. 
Since $\lambda = \nu_{2}-\nu_{3}+1$ and $-\lambda \leq q \leq \lambda$, 
this is the case when $\nu_{2} > \nu_{3}$ 
and 
$2 \nu_{1} \geq \nu_{2}+ \nu_{3}$. 
Therefore, $\nu$ must satisfy 
$\nu_{1} > \nu_{2} > \nu_{3}$ or $\nu_{2} > \nu_{1} > \nu_{3}$. 
Suppose first that $\nu_{1} > \nu_{2} > \nu_{3}$. 
In this case, 
$0 < \nu_{1}-\nu_{2}+1 \leq \nu_{1}-\nu_{3}+1 = \lambda$ 
and 
$c(\pm(\nu_{1}-\nu_{2}-1)) = 0$ by \eqref{eq:nu1-nu3+1}. 
By using \eqref{eq:nu1-nu3+1} again, we obtain 
$c(\nu_{1}-\nu_{2}-1-2k) = 0$ for $0 \leq k \leq \nu_{1}-\nu_{2}-1$. 
Therefore \eqref{eq:c(q) and c(q-2k)-2} implies that 
\begin{align*}
& 
\sum_{ q \geq \nu_{1}-\nu_{2}+1 
\atop 
q \equiv \nu_{1}-\nu_{2}+1 \ \modulo 2} 
\frac{\varGamma(\frac{\nu_{1}-\nu_{2}+1+q}{2})}
{\varGamma(\frac{\nu_{2}-\nu_{1}+1+q}{2})} 
\frac{1}{(\lambda+q)! (\lambda-q)!}
v_{q}^{\lambda} 
\qquad 
\mbox{and} 
\\
& 
\sum_{ q \geq \nu_{1}-\nu_{2}+1 
\atop 
q \equiv \nu_{1}-\nu_{2}+1 \ \modulo 2} 
\frac{\varGamma(\frac{\nu_{1}-\nu_{2}+1+q}{2})}
{\varGamma(\frac{\nu_{2}-\nu_{1}+1+q}{2})} 
\frac{1}{(\lambda+q)! (\lambda-q)!}
v_{-q}^{\lambda}
\end{align*}
are solutions of \eqref{eq:nu1-nu3+1}. 

\begin{proposition}\label{proposition:nu1-nu3+1, (2)}
If $\lambda = \nu_{1} - \nu_{3} + 1$ and $\nu_{1} > \nu_{2} > \nu_{3}$, 
then the solution space of $P_{-1} \phi_{\lambda}(e) = 0$ and 
$P_{-2} \phi_{\lambda}(e) = 0$ is 
\begin{align*}
\mathrm{Span}
\langle & 
v_{o} := 
\sum_{-\lambda \leq q \leq \lambda 
\atop q - \nu_{1}+\nu_{2} \equiv 0 \ \modulo 2}
\frac{
\varGamma(\frac{\nu_{1}-\nu_{2}+1+q}{2})}
{\varGamma(\frac{\nu_{2}-\nu_{1}+1+q}{2})}
\frac{1}{(\lambda+q)! (\lambda-q)!}
v_{q}^{\lambda}, 
\\
&
v_{+} := 
\sum_{ q \geq \nu_{1}-\nu_{2}+1 
\atop 
q \equiv \nu_{1}-\nu_{2}+1 \ \modulo 2} 
\frac{\varGamma(\frac{\nu_{1}-\nu_{2}+1+q}{2})}
{\varGamma(\frac{\nu_{2}-\nu_{1}+1+q}{2})} 
\frac{1}{(\lambda+q)! (\lambda-q)!}
v_{q}^{\lambda}, 
\\
& 
v_{-} := 
\sum_{ q \geq \nu_{1}-\nu_{2}+1 
\atop 
q \equiv \nu_{1}-\nu_{2}+1 \ \modulo 2} 
\frac{\varGamma(\frac{\nu_{1}-\nu_{2}+1+q}{2})}
{\varGamma(\frac{\nu_{2}-\nu_{1}+1+q}{2})} 
\frac{1}{(\lambda+q)! (\lambda-q)!}
v_{-q}^{\lambda}
\rangle 
\end{align*}
The vectors 
$v_{+} + (-1)^{i} v_{-}$, $i=0,1$, 
are contained in 
$V_{\lambda}^{SO(3)}(\sigma_{\nu_{2}-\nu_{3}+i, \nu_{1}-\nu_{3}+1+i})$ 
and 
$v_{o}$ is contained in 
$V_{\lambda}^{SO(3)}(\sigma_{\nu_{1}-\nu_{3}+1, \nu_{2}-\nu_{3}+1})$. 
\end{proposition}

Suppose that $\nu$ satisfies $\nu_{2} > \nu_{1} > \nu_{3}$ 
and $2 \nu_{1} \geq \nu_{2}+\nu_{3}$. 
In this case, $0 < \nu_{2}-\nu_{1}+1 \leq \nu_{1}-\nu_{3}+1 = \lambda$ 
and $c(\pm(\nu_{2}-\nu_{1}+1)) = 0$, by \eqref{eq:nu1-nu3+1}. 
By using \eqref{eq:nu1-nu3+1} again, 
we obtain $c(\pm(\nu_{2}-\nu_{1}+1+2k)) = 0$ 
for $k \geq 0$. 
Therefore, in this case, 
\eqref{eq:a solution of the eqn nu1-nu3+1-2} is 
a solution if we regard $1/\varGamma(-k) = 0$ for $k \in \Z_{\geq 0}$.

We have obtained all the solutions of \eqref{eq:nu1-nu3+1}. 
\begin{proposition}\label{proposition:nu1-nu3+1, (1)}
If $\lambda = \nu_{1} - \nu_{3} + 1$, 
and either $\nu_{2} > \nu_{1} > \nu_{3}$ or 
$\nu_{1} > \nu_{3} > \nu_{2}$, then 
the solution space of $P_{-1} \phi_{\lambda}(e) = 0$ and 
$P_{-2} \phi_{\lambda}(e) = 0$ is 
$\Span{w_{i} \mid i=0,1}$, where 
\begin{align*}
w_{0} 
&= 
\sum_{-\lambda \leq q \leq \lambda 
\atop q \equiv \nu_{1}-\nu_{2} \ \modulo 2}
\frac{\varGamma(\frac{\nu_{1}-\nu_{2}+1+q}{2})}
{\varGamma(\frac{\nu_{2}-\nu_{1}+1+q}{2})}
\frac{1}
{(\lambda+q)! (\lambda-q)!}
v_{q}^{\lambda},
\\
w_{1} 
&= 
\sum_{-\lambda \leq q \leq \lambda 
\atop q \equiv \nu_{1}-\nu_{2}+1 \ \modulo 2}
\frac{(-1)^{q/2}}
{\varGamma(\frac{\nu_{2}-\nu_{1}+1-q}{2}) 
\varGamma(\frac{\nu_{2}-\nu_{1}+1+q}{2})}
\frac{1}
{(\lambda+q)! (\lambda-q)!}
v_{q}^{\lambda}.
\end{align*}
These vectors $w_{i}$, $i=0,1$, are contained in 
$V_{\lambda}^{SO(3)}(\sigma_{\nu_{1}-\nu_{3}+1,\nu_{2}-\nu_{3}+1+i})$, 
respectively. 
\end{proposition}

\section{Candidates for irreducible submodules. $SL(3,\R)$ case.}
\label{section:candidates for irr sub. SL(3)}

In this section, we seek the candidates for the irreducible 
submodules of the principal series modules of $SL(3,\R)$. 

\subsection{The irreducible modules $\overline{X}(\gamma_{i})$, $i=0,1,2$}

The results in the previous section imply 
the possible principal series modules into which the irreducible 
modules $\overline{X}(\gamma_{i})$, $i=0,1,2$, are embedded. 

Let $\Lambda = (\Lambda_{1}, \Lambda_{2}, \Lambda_{3})$ be 
a dominant nonsingular integral infinitesimal character. 
Namely, it satisfies 
$\Lambda_{1}-\Lambda_{2}, \Lambda_{2}-\Lambda_{3} \in \Z_{>0}$. 
By Lemma~\ref{lemma:K-spectra of irreducibles, SL(3)} (1), 
we know that the minimal $K$-types of $\overline{X}(\gamma_{i})$, 
$i=0,1,2$, are 
$\lambda = \Lambda_{1}-\Lambda_{3}+1$, 
$\Lambda_{1}-\Lambda_{2}+1$ and 
$\Lambda_{2}-\Lambda_{3}+1$, respectively. 
Suppose that $\nu = (\nu_{1}, \nu_{2}, \nu_{3})$ is an element 
of the orbit $W(G,H_{s}) \cdot \Lambda$. 
By Propositions~\ref{proposition:4.1}, \ref{proposition:4.2}, 
\ref{proposition:nu1-nu3+1, (2)} and 
\ref{proposition:nu1-nu3+1, (1)}, 
we know the principal series modules 
of which 
$\overline{X}(\gamma_{i})$, $i=0,1,2$, may be an irreducible submodule. 
Moreover, by \eqref{eq:gamma_3}, \eqref{eq:gamma_4} and \eqref{eq:gamma_5}, 
we know which standard module is isomorphic to this principal series, 
in the Grothendieck group. 
We write these informations in a table. 

\begin{table}[hbt]
\caption{Possible embeddings of $\overline{X}(\gamma_{i})$, 
$i=0,1,2$}
\label{table:possible embeddings of X(gamma_i), i=0,1,2}
\begin{center}
\begin{tabular}{|c|c|c|c|c|}
\hline
Irred. Mod. & Propn \#
& \multicolumn{2}{|c|}{$I(\sigma_{i,j}, \nu)$} & Std Mod. 
\\ \hline 
 & & $\nu$ & $\sigma_{i,j}$ & $X(\gamma_{k})$ 
\\ \hline \hline 
$\overline{X}(\gamma_{0})$ 
& 
\ref{proposition:4.1} (1) 
&
$(\Lambda_{1}, \Lambda_{3}, \Lambda_{2})$ 
&
$\sigma_{\Lambda_{1}-\Lambda_{2}, \Lambda_{2}-\Lambda_{3}+1}$ 
&
$k=5$ 
\\ \hline 
& 
\ref{proposition:4.1} (1) 
&
$(\Lambda_{1}, \Lambda_{3}, \Lambda_{2})$ 
&
$\sigma_{\Lambda_{1}-\Lambda_{2}+1, \Lambda_{2}-\Lambda_{3}}$ 
&
$k=4$ 
\\ \hline 
&
\ref{proposition:4.2} (1) 
&
$(\Lambda_{2},\Lambda_{1},\Lambda_{3})$ 
&
$\sigma_{\Lambda_{2}-\Lambda_{3}, \Lambda_{1}-\Lambda_{3}+1}$ 
&
$k=4$ 
\\ \hline 
&
\ref{proposition:4.2} (1) 
&
$(\Lambda_{2},\Lambda_{1},\Lambda_{3})$ 
&
$\sigma_{\Lambda_{2}-\Lambda_{3}+1, \Lambda_{1}-\Lambda_{3}+1}$ 
&
$k=5$ 
\\ \hline 
& 
\ref{proposition:nu1-nu3+1, (2)} 
& 
$(\Lambda_{1},\Lambda_{2},\Lambda_{3})$ 
&
$\sigma_{\Lambda_{1}-\Lambda_{3}, \Lambda_{2}-\Lambda_{3}+1}$ 
&
$k=3$ 
\\ \hline 
& 
\ref{proposition:nu1-nu3+1, (2)} 
& 
$(\Lambda_{1},\Lambda_{2},\Lambda_{3})$ 
&
$\sigma_{\Lambda_{1}-\Lambda_{3}+1, \Lambda_{2}-\Lambda_{3}}$ 
&
$k=4$ 
\\ \hline 
& 
\ref{proposition:nu1-nu3+1, (2)} 
& 
$(\Lambda_{1},\Lambda_{2},\Lambda_{3})$ 
&
$\sigma_{\Lambda_{1}-\Lambda_{3}+1, \Lambda_{2}-\Lambda_{3}+1}$ 
&
$k=5$ 
\\ \hline \hline 
$\overline{X}(\gamma_{1})$ 
&
\ref{proposition:4.1} (1) 
&
$(\Lambda_{1},\Lambda_{2},\Lambda_{3})$ 
&
$\sigma_{\Lambda_{1}-\Lambda_{3},\Lambda_{2}-\Lambda_{3}+1}$ 
&
$k=3$ 
\\ \hline 
&
\ref{proposition:4.1} (1) 
&
$(\Lambda_{1},\Lambda_{2},\Lambda_{3})$ 
&
$\sigma_{\Lambda_{1}-\Lambda_{3}+1,\Lambda_{2}-\Lambda_{3}}$ 
&
$k=4$ 
\\ \hline 
&
\ref{proposition:4.2} (1) 
&
$(\Lambda_{3},\Lambda_{1},\Lambda_{2})$ 
&
$\sigma_{\Lambda_{2}-\Lambda_{3}+1,\Lambda_{1}-\Lambda_{2}+1}$ 
&
$k=3$ 
\\ \hline 
&
\ref{proposition:4.2} (1) 
&
$(\Lambda_{3},\Lambda_{1},\Lambda_{2})$ 
&
$\sigma_{\Lambda_{2}-\Lambda_{3},\Lambda_{1}-\Lambda_{2}+1}$ 
&
$k=4$ 
\\ \hline 
&
\ref{proposition:nu1-nu3+1, (1)} 
&
$(\Lambda_{1},\Lambda_{3},\Lambda_{2})$ 
&
$\sigma_{\Lambda_{1}-\Lambda_{2}+1,\Lambda_{2}-\Lambda_{3}+1}$ 
&
$k=3$ 
\\ \hline 
&
\ref{proposition:nu1-nu3+1, (1)} 
&
$(\Lambda_{1},\Lambda_{3},\Lambda_{2})$ 
&
$\sigma_{\Lambda_{1}-\Lambda_{2}+1,\Lambda_{2}-\Lambda_{3}}$ 
&
$k=4$ 
\\ \hline \hline 
$\overline{X}(\gamma_{2})$ 
&
\ref{proposition:4.1} (1) 
&
$(\Lambda_{2},\Lambda_{3},\Lambda_{1})$ 
&
$\sigma_{\Lambda_{1}-\Lambda_{2}+1,\Lambda_{1}-\Lambda_{3}}$ 
&
$k=3$ 
\\ \hline 
&
\ref{proposition:4.1} (1) 
&
$(\Lambda_{2},\Lambda_{3},\Lambda_{1})$ 
&
$\sigma_{\Lambda_{1}-\Lambda_{2},\Lambda_{1}-\Lambda_{3}+1}$ 
&
$k=5$ 
\\ \hline 
&
\ref{proposition:4.2} (1) 
&
$(\Lambda_{1},\Lambda_{2},\Lambda_{3})$ 
&
$\sigma_{\Lambda_{1}-\Lambda_{3},\Lambda_{2}-\Lambda_{3}+1}$ 
&
$k=3$ 
\\ \hline 
&
\ref{proposition:4.2} (1) 
&
$(\Lambda_{1},\Lambda_{2},\Lambda_{3})$ 
&
$\sigma_{\Lambda_{1}-\Lambda_{3}+1,\Lambda_{2}-\Lambda_{3}+1}$ 
&
$k=5$ 
\\ \hline 
&
\ref{proposition:nu1-nu3+1, (1)} 
&
$(\Lambda_{2},\Lambda_{1},\Lambda_{3})$ 
&
$\sigma_{\Lambda_{2}-\Lambda_{3}+1,\Lambda_{1}-\Lambda_{3}}$ 
&
$k=3$ 
\\ \hline 
&
\ref{proposition:nu1-nu3+1, (1)} 
&
$(\Lambda_{2},\Lambda_{1},\Lambda_{3})$ 
&
$\sigma_{\Lambda_{2}-\Lambda_{3}+1,\Lambda_{1}-\Lambda_{3}+1}$ 
&
$k=5$ 
\\ \hline 
\end{tabular}
\end{center}
\end{table}

We see that there exist principal series modules which may have 
two or more irreducible submodules. 
For example, 
$\overline{X}(\gamma_{0})$, $\overline{X}(\gamma_{1})$ and 
$\overline{X}(\gamma_{2})$ may be submodules of 
$I(\sigma_{\Lambda_{1}-\Lambda_{3},\Lambda_{2}-\Lambda_{3}+1}, 
(\Lambda_{1},\Lambda_{2},\Lambda_{3}))$. 
Let us explore whether this is true or not. 
For this purpose, we use the solutions obtained in 
Propositions~\ref{proposition:4.1}, \ref{proposition:4.2}, 
\ref{proposition:nu1-nu3+1, (2)} and 
\ref{proposition:nu1-nu3+1, (1)}. 

By the translation principle, we may put $\Lambda = \rho_{m} = (1,0,-1)$. 
Firstly, let us consider the case 
$I(\sigma_{\Lambda_{1}-\Lambda_{3},\Lambda_{2}-\Lambda_{3}+1}, 
(\Lambda_{1},\Lambda_{2},\Lambda_{3}))
=
I(\sigma_{0,0}, (1,0,-1))$. 
By \eqref{eq:principal series twisted by mu}, 
this principal series module has horizontal symmetry 
(Corollary~\ref{corollary:horizontal symmetry}). 
Since $\overline{X}(\gamma_{1})$ and $\overline{X}(\gamma_{2})$ are 
quasi-dual (Lemma~\ref{lemma:quasi-dual modules, SL(3)}), 
they appear in $I(\sigma_{0,0}, (1,0,-1))$ as a pair 
$\overline{X}(\gamma_{1}) \oplus \overline{X}(\gamma_{2})$. 
By Proposition~\ref{proposition:4.2} (2) and 
Table~\ref{table:possible embeddings of X(gamma_i), i=0,1,2}, 
the solution which corresponds to 
$\overline{X}(\gamma_{1}) \oplus \overline{X}(\gamma_{2}) 
\subset 
I(\sigma_{0,0}, (1,0,-1))$ is 
$c_{1}(v_{2}^{2}+v_{-2}^{2})+c_{2} v_{0}^{2}$, 
where $c_{1}, c_{2}$ are non-zero constants. 
In the case $\Lambda=\rho_{m}$, 
Lemma~\ref{lemma:K-spectra of irreducibles, SL(3)} (2) says that 
the minimal $K$-types of $\overline{X}(\gamma_{1})$ 
and $\overline{X}(\gamma_{2})$ are $\lambda=2$ and 
$\lambda=3$ is not a $K$-type of them. 
Therefore, if $\overline{X}(\gamma_{1}) \oplus \overline{X}(\gamma_{2})$ 
lies in the socle of $I(\sigma_{0,0}, (1,0,-1))$, 
then $P_{1}(c_{1}(v_{2}^{2}+v_{-2}^{2})+c_{2} v_{0}^{2}) = 0$ 
for any $c_{1}, c_{2}$. 
But, by \eqref{eq:P_1, SL(3)}, 
$P_{1}(c_{1}(v_{2}^{2}+v_{-2}^{2})+c_{2} v_{0}^{2}) 
= 
(24 c_{1}+4c_{2})\, (v_{2}^{3}-v_{-2}^{3})$. 
It follows that $\overline{X}(\gamma_{1}) \oplus \overline{X}(\gamma_{2})$ 
does not lie in the socle. 

Just in the same way, we can check that 
$\overline{X}(\gamma_{1})$ (resp. $\overline{X}(\gamma_{2})$) 
does not lie in the socle of 
$I(\sigma_{\Lambda_{1}-\Lambda_{3}+1,\Lambda_{2}-\Lambda_{3}}, 
(\Lambda_{1},\Lambda_{2},\Lambda_{3}))$, 
$I(\sigma_{\Lambda_{1}-\Lambda_{2}+1,\Lambda_{2}-\Lambda_{3}}, 
(\Lambda_{1},\Lambda_{3},\Lambda_{2}))$ 
(resp. 
$I(\sigma_{\Lambda_{1}-\Lambda_{3}+1,\Lambda_{2}-\Lambda_{3}+1}, 
(\Lambda_{1},\Lambda_{2},\Lambda_{3}))$, 
$I(\sigma_{\Lambda_{2}-\Lambda_{3}+1,\Lambda_{1}-\Lambda_{3}+1}, 
(\Lambda_{2},\Lambda_{1},\Lambda_{3}))$). 

We have obtained possible embeddings of $\overline{X}(\gamma_{i})$, 
$i=0,1,2$ into principal series modules. 
In the followings of this paper, 
$I(\sigma,\nu) \approx X(\gamma_{j})$ means that 
they are isomorphic as elements in the Grothendieck group. 
Since $I(\sigma,\nu)$ is specified if we know the character $\nu$ 
and the standard module $X(\gamma_{j})$ satisfying 
$I(\sigma, \nu) \approx X(\gamma_{j})$, 
we write only these informations in the following proposition. 
\begin{proposition}
\label{proposition:candidate where bar{X}(gamma_i), i=0,1,2, is sub}
Suppose that the nonsingular integral infinitesimal character 
is $\Lambda = (\Lambda_{1},\Lambda_{2},\Lambda_{3})$, 
$\Lambda_{1}-\Lambda_{2}, \Lambda_{2}-\Lambda_{3} \in \Z_{>0}$. 
\begin{enumerate}
\item
$\overline{X}(\gamma_{0})$ can be a submodule of $I(\sigma, \nu)$ 
only if one of the following conditions holds: 
\begin{enumerate}
\item
$I(\sigma,\nu) \approx X(\gamma_{3})$ 
and 
$\nu = (\Lambda_{1}, \Lambda_{2},\Lambda_{3})$. 
In this case, 
$\overline{X}(\gamma_{1}) \oplus \overline{X}(\gamma_{2})$ 
lies in the higher floor than $\overline{X}(\gamma_{0})$. 
\item
$I(\sigma,\nu) \approx X(\gamma_{4})$ 
and 
$\nu 
= 
(\Lambda_{1},\Lambda_{2},\Lambda_{3})$. 
In this case, 
$\overline{X}(\gamma_{1})$ 
lies in the higher floor than $\overline{X}(\gamma_{0})$. 
\item
$I(\sigma,\nu) \approx X(\gamma_{5})$ 
and 
$\nu 
= 
(\Lambda_{1},\Lambda_{2},\Lambda_{3})$. 
In this case, 
$\overline{X}(\gamma_{2})$ 
lies in the higher floor than $\overline{X}(\gamma_{0})$. 
\item
$I(\sigma,\nu) \approx X(\gamma_{4})$ 
and 
$\nu 
= 
(\Lambda_{2},\Lambda_{1},\Lambda_{3})$. 
\item
$I(\sigma,\nu) \approx X(\gamma_{5})$ 
and 
$\nu 
= 
(\Lambda_{2},\Lambda_{1},\Lambda_{3})$. 
In this case, 
$\overline{X}(\gamma_{2})$ 
lies in the higher floor than $\overline{X}(\gamma_{0})$. 
\item
$I(\sigma,\nu) \approx X(\gamma_{4})$ 
and 
$\nu 
= 
(\Lambda_{1},\Lambda_{3},\Lambda_{2})$. 
In this case, 
$\overline{X}(\gamma_{1})$ 
lies in the higher floor than $\overline{X}(\gamma_{0})$. 
\item
$I(\sigma,\nu) \approx X(\gamma_{5})$ 
and 
$\nu 
= 
(\Lambda_{1},\Lambda_{3},\Lambda_{2})$. 
\end{enumerate}
\item
$\overline{X}(\gamma_{1})$ can be a submodule of 
$I(\sigma, \nu)$ only if one of the following conditions holds: 
\begin{enumerate}
\item
$I(\sigma,\nu) \approx X(\gamma_{3})$ 
and 
$\nu 
= 
(\Lambda_{3},\Lambda_{1},\Lambda_{2})$. 
\item
$I(\sigma,\nu) \approx X(\gamma_{4})$ 
and 
$\nu 
= 
(\Lambda_{3},\Lambda_{1},\Lambda_{2})$. 
\item
$I(\sigma,\nu) \approx X(\gamma_{3})$ 
and 
$\nu 
= 
(\Lambda_{1},\Lambda_{3},\Lambda_{2})$. 
\end{enumerate}
\item
$\overline{X}(\gamma_{2})$ can be a submodule of 
$I(\sigma, \nu)$ only if one of the following conditions holds: 
\begin{enumerate}
\item
$I(\sigma,\nu) \approx X(\gamma_{3})$ 
and 
$\nu 
= 
(\Lambda_{2},\Lambda_{1},\Lambda_{3})$. 
\item
$I(\sigma,\nu) \approx X(\gamma_{3})$ 
and 
$\nu 
= 
(\Lambda_{2},\Lambda_{3},\Lambda_{1})$. 
\item
$I(\sigma,\nu) \approx X(\gamma_{5})$ 
and 
$\nu 
= 
(\Lambda_{2},\Lambda_{3},\Lambda_{1})$. 
\end{enumerate}
\end{enumerate}
\end{proposition}

\subsection{The irreducible module $\overline{X}(\gamma_{3})$} 

If $\Lambda$ is trivial, 
then $\overline{X}(\gamma_{3})$ is the trivial $\brgK$-module 
and then its $K$-type is $V_{0}^{SO(3)}$ alone. 
So an embedding corresponds to a non-trivial solution 
of $(P_{1} \phi_{0})(e) = 0$ and $(P_{2} \phi_{0})(e) = 0$. 

For $\lambda = 0$, the equation $(P_{1} \phi_{0})(e) = 0$ is trivial. 
By \eqref{eq:P_2, SL(3)}, $(P_{2} \phi_{0})(e) = 0$ is 
\[
\begin{cases}
(\nu_{1}-\nu_{2}+1)\, c(0) = 0
\\
(\nu_{1}+\nu_{2}-2\nu_{3}+3)\, c(0) = 0
\end{cases}. 
\]
This system of equations has non-trivial solution if and only if 
$\nu_{2}=\nu_{1}+1$ and $\nu_{3}=\nu_{1}+2$. 
Since we are considering the case of trivial infinitesimal character, 
the possible $\nu$ is only $(-1,0,1)$. 
By Casselman's subrepresentation theorem, every irreducible 
admissible $\brgK$-module is an submodule of at least one 
principal series module. 
So the unique solution obtained above really corresponds to an embedding. 
By translation principle, we obtained the following proposition. 

\begin{proposition}
\label{proposition:embedding of X(gamma_3)}
In the case when the infinitesimal character is 
$\Lambda = (\Lambda_{1}, \Lambda_{2}, \Lambda_{3})$, 
$\Lambda_{1}-\Lambda_{2}, \Lambda_{2}-\Lambda_{3} \in \Z_{>0}$, 
the irreducible finite dimensional module $\overline{X}(\gamma_{3})$ 
is embedded into only one principal series module 
$I(\sigma_{\Lambda_{1}-\Lambda_{3},\Lambda_{1}-\Lambda_{2}+1}, 
(\Lambda_{3}, \Lambda_{2}, \Lambda_{1}))$. 
\end{proposition}

\subsection{The irreducible modules $\overline{X}(\gamma_{4})$ and 
$\overline{X}(\gamma_{5})$} 
\label{section:embeddings of 4, 5, SL(3)} 

Suppose that $\Lambda$ is trivial. 
Then 
Lemma~\ref{lemma:K-spectra of irreducibles, SL(3)}(2) 
says that 
both $\overline{X}(\gamma_{4})$ and $\overline{X}(\gamma_{5})$ contain 
the $K$-type $\lambda=1$, 
with multiplicity one and do not contain 
$\lambda=2$. 
It follows that if 
$\Ind_{P_{m}}^{G}(\sigma_{i,j} \otimes e^{-\nu-\rho_{m}})$, 
$\nu \in W(G,H_{s}) \cdot \rho_{m}$, contains 
$\overline{X}(\gamma_{k})$, $k=4$ or $5$, as a submodule, 
then the equation $(P_{1} \phi_{1})(e) = 0$ has a non-zero solution. 
For 
$\phi_{1}(e) = \sum_{q=-1}^{1} c(q)\, v_{q}^{1}$,  
this equation is equivalent to 
\begin{align}
& 
(\nu_{1}-\nu_{2}+1)\, c(0) = 0, 
\label{eq:P_1=0,lambda=1,0}
\\
& 
(\nu_{1}-\nu_{2})\, c(-1) + (\nu_{1}+\nu_{2}-2\nu_{3}+2)\, c(1) = 0, 
\label{eq:P_1=0,lambda=1,1}
\\
& 
(\nu_{1}+\nu_{2}-2\nu_{3}+2)\, c(-1) + (\nu_{1}-\nu_{2})\, c(1) = 0, 
\label{eq:P_1=0,lambda=1,-1}
\end{align}
by \eqref{eq:P_1, SL(3)}. 

The equation \eqref{eq:P_1=0,lambda=1,0} has a non-zero solution 
if and only if $\nu_{2}=\nu_{1}+1$, i.e. 
$\nu = (0,1,-1)$ or $(-1,0,1)$, and the solution space is 
$\C v_{0}^{1}$. 
Since $v_{0}^{1} \in V_{1}^{SO(3)}(\sigma_{1,1})$, 
this vector corresponds to the composition factor 
$\overline{X}(\gamma_{4}) 
\subset 
I(\sigma_{1,1}, (0,1,-1)) 
\approx X(\gamma_{4})$ 
and 
$\overline{X}(\gamma_{5}) 
\subset 
I(\sigma_{1,1}, (-1,0,1)) 
\approx X(\gamma_{5})$. 
Here, we used \eqref{eq:gamma_4} and \eqref{eq:gamma_5}.

The system of equations \eqref{eq:P_1=0,lambda=1,1} and 
\eqref{eq:P_1=0,lambda=1,-1} has a non-zero solution 
if and only if 
\[
(\nu_{1}+\nu_{2}-2\nu_{3}+2)^{2}-(\nu_{1}-\nu_{2})^{2} 
= 
4(\nu_{1}-\nu_{3}+1) (\nu_{2}-\nu_{3}+1)
= 0.
\]
If $\nu_{3}=\nu_{1}+1$, i.e. 
$\nu = (-1,1,0)$ or $(0,-1,1)$, then the solution space is 
$\C(v_{1}^{1}+v_{-1}^{1})$. 
If $\nu_{3}=\nu_{2}+1$, i.e. 
$\nu = (-1,0,1)$ or $(1,-1,0)$, then the solution space is 
$\C(v_{1}^{1}-v_{-1}^{1})$. 
Since 
$v_{1}^{1}+v_{-1}^{1} 
\in V_{1}^{SO(3)}(\sigma_{0,1})$ 
and 
$v_{1}^{1}-v_{-1}^{1} 
\in V_{1}^{SO(3)}(\sigma_{1,0})$, 
the informations \eqref{eq:gamma_4} and \eqref{eq:gamma_5} 
imply that 
these vectors correspond to (the possibilities of the) 
composition factors 
$\overline{X}(\gamma_{4}) 
\subset 
I(\sigma_{0,1}, (0,-1,1)), I(\sigma_{1,0}, (-1,0,1)) 
\approx 
X(\gamma_{4})$ 
and 
$\overline{X}(\gamma_{5}) 
\subset 
I(\sigma_{0,1}, (-1,1,0)), I(\sigma_{1,0}, (1,-1,0)) 
\approx 
X(\gamma_{5})$. 
We have obtained the following results. 
\begin{proposition}
\label{proposition:embeddings of X(gamma_i), i=4,5} 
Suppose that the nonsingular integral infinitesimal character 
is $\Lambda = (\Lambda_{1},\Lambda_{2},\Lambda_{3})$, 
$\Lambda_{1}-\Lambda_{2}, \Lambda_{2}-\Lambda_{3} \in \Z_{>0}$. 
\begin{enumerate}
\item
$\overline{X}(\gamma_{4})$ can be a submodule of 
$I(\sigma,\nu)$ only if 
$I(\sigma,\nu) \approx X(\gamma_{4})$ and 
$\nu 
= 
(\Lambda_{2},\Lambda_{1},\Lambda_{3})$, 
$(\Lambda_{2},\Lambda_{3},\Lambda_{1})$ 
or 
$(\Lambda_{3},\Lambda_{2},\Lambda_{1})$. 
\item
$\overline{X}(\gamma_{5})$ can be a submodule of 
$I(\sigma,\nu)$ only if 
$I(\sigma,\nu) \approx X(\gamma_{5})$ and 
$\nu 
= 
(\Lambda_{3},\Lambda_{2},\Lambda_{1})$, 
$(\Lambda_{3},\Lambda_{1},\Lambda_{2})$ 
or 
$(\Lambda_{1},\Lambda_{3},\Lambda_{2})$. 
\end{enumerate}
\end{proposition}

\section{Determination of the socle filtration. $SL(3,\R)$ case.}
\label{section:main results for SL(3)}

From 
Propositions~
\ref{proposition:candidate where bar{X}(gamma_i), i=0,1,2, is sub}, 
\ref{proposition:embedding of X(gamma_3)} and 
\ref{proposition:embeddings of X(gamma_i), i=4,5}, 
we can determine the socle filtrations of the principal 
series modules. 

For the $SL(3,\R)$ case, 
the dual module of $I(\sigma_{i,j}, \nu)$ is 
$I(\sigma_{i,j}, -\nu)$ 
and the $\mu'$-twisted module 
(see \eqref{eq:principal series twisted by mu}) 
of the latter is isomorphic to 
$I(\sigma_{i,i+j}, w^{0} \cdot \nu)$, 
where $w^{0}$ is the longest element of $W(G,A_{m})$. 
It follows that the socle filtrations of 
\begin{align}
& 
I(\sigma_{i,j}, (\nu_{1},\nu_{2},\nu_{3})) 
&
&
\mbox{and}
&
&
I(\sigma_{i,i+j}, (\nu_{3},\nu_{2},\nu_{1})) 
\label{eq:upside down}
\end{align}
are upside down.

We make a list of the candidates for the irreducible submodules of 
$I(\sigma, \nu)$. 
In Table~\ref{table:candidates for irreducible submodules}, 
the first column means that the principal series in consideration 
is isomorphic to $X(\gamma_{i})$ in the Grothendieck group. 
The second and fourth columns are the "$\nu$" part of 
$I(\sigma, \nu)$, and the third and fifth columns are the 
candidates for irreducible submodules of $I(\sigma, \nu)$.

\begin{table}[h]
\caption{Candidates for irreducible submodules}
\label{table:candidates for irreducible submodules} 
\begin{center}
\begin{tabular}{|c||c|c||c|c|}
\hline 
& $\nu$ & Irred.Sub. & $\nu$ & Irred.Sub.
\\ \hline \hline 
$X(\gamma_{3})$ 
& 
$(\Lambda_{1},\Lambda_{2},\Lambda_{3})$ 
& 
$\overline{X}(\gamma_{0})$ 
& 
$(\Lambda_{3},\Lambda_{2},\Lambda_{1})$ 
&
$\overline{X}(\gamma_{3})$ 
\\ \hline
& 
$(\Lambda_{2},\Lambda_{1},\Lambda_{3})$ 
& 
$\overline{X}(\gamma_{2})$ 
& 
$(\Lambda_{3},\Lambda_{1},\Lambda_{2})$ 
&
$\overline{X}(\gamma_{1})$ 
\\ \hline
& 
$(\Lambda_{2},\Lambda_{3},\Lambda_{1})$ 
& 
$\overline{X}(\gamma_{2})$ 
& 
$(\Lambda_{1},\Lambda_{3},\Lambda_{2})$ 
&
$\overline{X}(\gamma_{1})$ 
\\ \hline \hline 
$X(\gamma_{4})$ 
& 
$(\Lambda_{1},\Lambda_{2},\Lambda_{3})$ 
& 
$\overline{X}(\gamma_{0})$ 
& 
$(\Lambda_{3},\Lambda_{2},\Lambda_{1})$ 
&
$\overline{X}(\gamma_{4})$ 
\\ \hline
& 
$(\Lambda_{2},\Lambda_{1},\Lambda_{3})$ 
& 
$\overline{X}(\gamma_{0})$, $\overline{X}(\gamma_{4})$ 
& 
$(\Lambda_{3},\Lambda_{1},\Lambda_{2})$ 
&
$\overline{X}(\gamma_{1})$ 
\\ \hline
& 
$(\Lambda_{2},\Lambda_{3},\Lambda_{1})$ 
& 
$\overline{X}(\gamma_{4})$ 
& 
$(\Lambda_{1},\Lambda_{3},\Lambda_{2})$ 
&
$\overline{X}(\gamma_{0})$ 
\\ \hline \hline 
$X(\gamma_{5})$ 
& 
$(\Lambda_{1},\Lambda_{2},\Lambda_{3})$ 
& 
$\overline{X}(\gamma_{0})$ 
& 
$(\Lambda_{3},\Lambda_{2},\Lambda_{1})$ 
&
$\overline{X}(\gamma_{5})$ 
\\ \hline
& 
$(\Lambda_{2},\Lambda_{1},\Lambda_{3})$ 
& 
$\overline{X}(\gamma_{0})$
& 
$(\Lambda_{3},\Lambda_{1},\Lambda_{2})$ 
&
$\overline{X}(\gamma_{5})$ 
\\ \hline
& 
$(\Lambda_{2},\Lambda_{3},\Lambda_{1})$ 
& 
$\overline{X}(\gamma_{2})$ 
& 
$(\Lambda_{1},\Lambda_{3},\Lambda_{2})$ 
&
$\overline{X}(\gamma_{0})$, $\overline{X}(\gamma_{5})$  
\\ \hline
\end{tabular}
\end{center}
\end{table}
We know the composition factors of $X(\gamma_{3})$, 
$X(\gamma_{4})$ and $X(\gamma_{5})$ 
(Theorem~\ref{theorem:composition factors for SL(3)}). 
We also know that the lengths (Section~\ref{subsection:parity of length}) 
of regular characters are 
\begin{align*}
&
\ell(\gamma_{0}) 
= 0, 
&
&
\ell(\gamma_{1}) 
= 
\ell(\gamma_{2}) = 1, 
&
&
\ell(\gamma_{3}) = \ell(\gamma_{4}) = \ell(\gamma_{5}) = 2. 
\end{align*}
From Table~\ref{table:candidates for irreducible submodules}, 
the parity condition (Corollary~\ref{corollary:parity condition}) and 
\eqref{eq:upside down}, 
all the socle filtrations of principal series modules are completely 
determined. 
For example, consider the principal series 
$I(\sigma_{\Lambda_{1}-\Lambda_{3},\Lambda_{2}-\Lambda_{3}+1}, 
(\Lambda_{1},\Lambda_{2},\Lambda_{3})) 
\approx 
X(\gamma_{3})$. 
The irreducible factors of this module are 
$\overline{X}(\gamma_{i})$, $i=0,1,2,3$, 
and the multiplicities are all one. 
The second row of Table~\ref{table:candidates for irreducible submodules} 
says that the candidates for the irreducible submodules 
of this module is $\overline{X}(\gamma_{0})$ 
only, so it is really the unique submodule. 
It also tells us that $\overline{X}(\gamma_{3})$ is the unique 
quotient module of it. 
The lengths of the other irreducible factors $\overline{X}(\gamma_{1})$ 
and $\overline{X}(\gamma_{2})$ are odd, but 
those of $\overline{X}(\gamma_{0})$ and $\overline{X}(\gamma_{3})$ are 
even. 
Then by the parity condition, 
the socle filtration of 
$I(\sigma_{\Lambda_{1}-\Lambda_{3},\Lambda_{2}-\Lambda_{3}+1}, 
(\Lambda_{1},\Lambda_{2},\Lambda_{3}))$ 
is 
\begin{equation}
\label{eq:example of socle filtration}
\begin{xy}
(0,6)*{\overline{X}(\gamma_{3})}="A_{1}", 
(0,0)*{\overline{X}(\gamma_{1}) \oplus \overline{X}(\gamma_{2})}
="A_{2}", 
(0,-6)*{\overline{X}(\gamma_{0})}="A_{3}", 
\end{xy}
\end{equation}
This diagram means that 
the socle of 
$I := I(\sigma_{\Lambda_{1}-\Lambda_{3},\Lambda_{2}-\Lambda_{3}+1}, 
(\Lambda_{1},\Lambda_{2},\Lambda_{3}))$ 
is $\overline{X}(\gamma_{0})$; 
the socle of 
$I/\overline{X}(\gamma_{0})$ 
is $\overline{X}(\gamma_{1}) \oplus \overline{X}(\gamma_{2})$; 
and the socle of 
$(I/\overline{X}(\gamma_{0})) / 
(\overline{X}(\gamma_{1}) \oplus \overline{X}(\gamma_{2}) 
/ \overline{X}(\gamma_{0}))$ 
is $\overline{X}(\gamma_{3})$. 
Hereafter, for notational convenience, 
we abbreviate \eqref{eq:example of socle filtration} as 
\begin{equation}
\label{eq:example of socle filtration'}
\begin{xy}
(0,4)*{\overline{3}}="A_{1}", 
(0,0)*{\overline{1} \oplus \overline{2}}
="A_{2}", 
(0,-4)*{\overline{0}}="A_{3}", 
\end{xy}
\end{equation}

In the same way, we obtain the main results for the $SL(3,\R)$ case. 
\begin{theorem}
\label{theorem:main results for SL(3)}
Let 
$\Lambda = (\Lambda_{1}, \Lambda_{2}, \Lambda_{3})$ be 
a nonsingular dominant integral infinitesimal character of $SL(3,\R)$. 
The socle filtrations of the principal series modules 
$I(\sigma,\nu)$ with $\nu \in W(G,A_{m}) \cdot \Lambda$ are 
as follows: 
\begin{enumerate}
\item
The cases $I(\sigma,\nu) \approx X(\gamma_{3})$. 
\begin{center}
\begin{tabular}{|c||c|}
\hline 
$I(\sigma_{\Lambda_{1}-\Lambda_{3},\Lambda_{2}-\Lambda_{3}+1}, 
(\Lambda_{1},\Lambda_{2},\Lambda_{3}))$ 
&
$I(\sigma_{\Lambda_{1}-\Lambda_{3},\Lambda_{1}-\Lambda_{2}+1}, 
(\Lambda_{3},\Lambda_{2},\Lambda_{1}))$ 
\\ \hline 
\
\vspace{-3mm}
&
\\
$\begin{xy}
(0,4)*{\overline{3}}="A_{1}", 
(0,0)*{\overline{1} \oplus \overline{2}}
="A_{2}", 
(0,-4)*{\overline{0}}="A_{3}", 
\end{xy}$ 
&
$\begin{xy}
(0,4)*{\overline{0}}="A_{1}", 
(0,0)*{\overline{1} \oplus \overline{2}}
="A_{2}", 
(0,-4)*{\overline{3}}="A_{3}", 
\end{xy}$ 
\\
\ \vspace{-3mm}
&
\\ \hline \hline 
$I(\sigma_{\Lambda_{2}-\Lambda_{3}+1,\Lambda_{1}-\Lambda_{3}}, 
(\Lambda_{2},\Lambda_{1},\Lambda_{3}))$, 
&
$I(\sigma_{\Lambda_{2}-\Lambda_{3}+1,\Lambda_{1}-\Lambda_{2}+1}, 
(\Lambda_{3},\Lambda_{1},\Lambda_{2}))$, 
\\
$I(\sigma_{\Lambda_{1}-\Lambda_{2}+1,\Lambda_{1}-\Lambda_{3}} 
(\Lambda_{2},\Lambda_{3},\Lambda_{1}))$ 
&
$I(\sigma_{\Lambda_{1}-\Lambda_{2}+1,\Lambda_{2}-\Lambda_{3}+1} 
(\Lambda_{1},\Lambda_{3},\Lambda_{2}))$ 
\\ \hline 
\
\vspace{-3mm}
&
\\
$\begin{xy}
(0,4)*{\overline{1}}="A_{1}", 
(0,0)*{\overline{0} \oplus \overline{3}}
="A_{2}", 
(0,-4)*{\overline{2}}="A_{3}", 
\end{xy}$ 
&
$\begin{xy}
(0,4)*{\overline{2}}="A_{1}", 
(0,0)*{\overline{0} \oplus \overline{3}}
="A_{2}", 
(0,-4)*{\overline{1}}="A_{3}", 
\end{xy}$ 
\\ \hline 
\end{tabular}
\end{center}
\item
The cases $I(\sigma,\nu) \approx X(\gamma_{4})$. 
\begin{center}
\begin{tabular}{|c||c|}
\hline 
$I(\sigma_{\Lambda_{1}-\Lambda_{3}+1,\Lambda_{2}-\Lambda_{3}}, 
(\Lambda_{1},\Lambda_{2},\Lambda_{3}))$, 
&
$I(\sigma_{\Lambda_{1}-\Lambda_{3}+1,\Lambda_{1}-\Lambda_{2}+1}, 
(\Lambda_{3},\Lambda_{2},\Lambda_{1}))$, 
\\
$I(\sigma_{\Lambda_{1}-\Lambda_{2}+1,\Lambda_{2}-\Lambda_{3}}, 
(\Lambda_{1},\Lambda_{3},\Lambda_{2}))$ 
&
$I(\sigma_{\Lambda_{1}-\Lambda_{2}+1,\Lambda_{1}-\Lambda_{3}+1}, 
(\Lambda_{2},\Lambda_{3},\Lambda_{1}))$ 
\\ \hline 
\
\vspace{-3mm}
&
\\
$\begin{xy}
(0,4)*{\overline{4}}="A_{1}", 
(0,0)*{\overline{1}}
="A_{2}", 
(0,-4)*{\overline{0}}="A_{3}", 
\end{xy}$ 
&
$\begin{xy}
(0,4)*{\overline{0}}="A_{1}", 
(0,0)*{\overline{1}}
="A_{2}", 
(0,-4)*{\overline{4}}="A_{3}", 
\end{xy}$ 
\\
\
\vspace{-3mm}
& 
\\ \hline \hline 
$I(\sigma_{\Lambda_{2}-\Lambda_{3},\Lambda_{1}-\Lambda_{3}+1}, 
(\Lambda_{2},\Lambda_{1},\Lambda_{3}))$ 
&
$I(\sigma_{\Lambda_{2}-\Lambda_{3},\Lambda_{1}-\Lambda_{2}+1}, 
(\Lambda_{3},\Lambda_{1},\Lambda_{2}))$ 
\\ \hline 
\ 
\vspace{-3mm}
&
\\
$\begin{xy}
(0,2)*{\overline{1}}="A_{1}", 
(0,-2)*{\overline{0} \oplus \overline{4}}
="A_{2}", 
\end{xy}$ 
&
$\begin{xy}
(0,2)*{\overline{0} \oplus \overline{4}}
="A_{2}", 
(0,-2)*{\overline{1}}="A_{3}", 
\end{xy}$ 
\\ \hline 
\end{tabular}
\end{center}
\item
The cases $I(\sigma,\nu) \approx X(\gamma_{5})$. 
\begin{center}
\begin{tabular}{|c||c|}
\hline 
$I(\sigma_{\Lambda_{1}-\Lambda_{3}+1,\Lambda_{2}-\Lambda_{3}+1}, 
(\Lambda_{1},\Lambda_{2},\Lambda_{3}))$, 
&
$I(\sigma_{\Lambda_{1}-\Lambda_{3}+1,\Lambda_{1}-\Lambda_{2}}, 
(\Lambda_{3},\Lambda_{2},\Lambda_{1}))$, 
\\
$I(\sigma_{\Lambda_{2}-\Lambda_{3}+1,\Lambda_{1}-\Lambda_{3}+1}, 
(\Lambda_{2},\Lambda_{1},\Lambda_{3}))$ 
&
$I(\sigma_{\Lambda_{2}-\Lambda_{3}+1,\Lambda_{1}-\Lambda_{2}}, 
(\Lambda_{3},\Lambda_{1},\Lambda_{2}))$ 
\\ \hline 
\ \vspace{-3mm} 
&
\\
$\begin{xy}
(0,4)*{\overline{5}}="A_{1}", 
(0,0)*{\overline{2}}
="A_{2}", 
(0,-4)*{\overline{0}}="A_{3}", 
\end{xy}$ 
&
$\begin{xy}
(0,4)*{\overline{0}}="A_{1}", 
(0,0)*{\overline{2}}
="A_{2}", 
(0,-4)*{\overline{5}}="A_{3}", 
\end{xy}$ 
\\
\ \vspace{-3mm} 
& 
\\ \hline \hline 
$I(\sigma_{\Lambda_{1}-\Lambda_{2},\Lambda_{2}-\Lambda_{3}+1}, 
(\Lambda_{1},\Lambda_{3},\Lambda_{2}))$ 
&
$I(\sigma_{\Lambda_{1}-\Lambda_{2},\Lambda_{1}-\Lambda_{3}+1}, 
(\Lambda_{2},\Lambda_{3},\Lambda_{1}))$ 
\\ \hline 
\ \vspace{-3mm} 
& 
\\ 
$\begin{xy}
(0,2)*{\overline{2}}="A_{1}", 
(0,-2)*{\overline{0} \oplus \overline{5}}
="A_{2}", 
\end{xy}$ 
&
$\begin{xy}
(0,2)*{\overline{0} \oplus \overline{5}}
="A_{2}", 
(0,-2)*{\overline{2}}="A_{3}", 
\end{xy}$ 
\\ \hline 
\end{tabular}
\end{center}
\end{enumerate}
\end{theorem}



\section{The group $Sp(2,\R)$} 
\label{section:the group Sp(2)}

From this section to Section~\ref{section:determination, X(11)}, 
we consider the group  
\[
G = 
Sp(2, \R) 
= 
\left\{
g \in GL(4, \R) 
\mid 
{}^{t}g 
\begin{pmatrix}
O & -I \\ I & O 
\end{pmatrix}
g = 
\begin{pmatrix}
O & -I \\ I & O 
\end{pmatrix}
\right\} 
\]
and investigate the socle filtrations of 
the principal series modules of $G$. 

Let 
\begin{equation}\label{eq:K=U(2)}
K 
:= 
\left\{ 
\begin{pmatrix} 
A & B \\ -B & A 
\end{pmatrix}
\in Sp(2, \R)
\right\} \simeq U(2), 
\quad 
\begin{pmatrix} 
A & B \\ -B & A 
\end{pmatrix}
\mapsto 
A + \I B 
\end{equation}
be a maximal compact subgroup of $G$. 

Let $\lier{g} = \lier{k} \oplus \lier{s}$ 
be the Cartan decomposition of $\lier{g}$ with respect to the 
Cartan involution $\theta(X) = -{}^{t}X$. 
Choose a maximal commutative subspace %
$(\lie{a}_{m})_{0}
= \{\diag(a_{1}, a_{2}, -a_{1}, -a_{2}) \in G 
\mid 
a_{i} \in \R\}$ of $\lier{s}$ and define 
\begin{align*}
A_{m} 
&= \{\diag(a_{1}, a_{2}, a_{1}^{-1}, a_{2}^{-1}) \in G \mid a_{i} > 0\} 
\qquad 
\mbox{ and }
\\
N_{m} 
&= 
\left\{
\begin{pmatrix}
1 & x_{12} & & 
\\
& 1 & & 
\\
& & 1 & 
\\
& & -x_{12} & 1  
\end{pmatrix}
\begin{pmatrix}
1 & & x_{13} & x_{14} 
\\
& 1 & x_{14} & x_{24} 
\\
& & 1 & 
\\
& & & 1 
\end{pmatrix} 
\mid x_{ij} \in \R
\right\}. 
\end{align*}
Then 
$G=KA_{m}N_{m}$ is an Iwasawa decomposition of $G$.

Define a basis $H_{1}, H_{2}$ of $(\lie{a}_{m})_{0}$ by 
\[
H_{i} := E_{ii} - E_{2+i,2+i}, \qquad i=1,2.
\] 
Let $f_{1}$, $f_{2}$, be the dual basis of $(\lie{a}_{m})_{0}^{\ast}$ 
defined by $f_{i}(H_{j}) = \delta_{i,j}$. 
The root system 
$\Sigma
= 
\Sigma(\lier{g}, (\lie{a}_{m})_{0})$ is 
$\{\pm f_{1} \pm f_{2}, \pm 2 f_{1}, \pm 2 f_{2}\}$. 
We choose $\Sigma^{+} = \{f_{1} \pm f_{2}, 2 f_{1}, 2 f_{2}\}$ 
as a positive system of it. 
Define root vectors for positive roots by 
\begin{align*}
& X_{f_{1}-f_{2}} 
:= E_{12} - E_{43}, 
&
& X_{f_{1}+f_{2}} 
:= E_{14} + E_{23}, 
&
& X_{2 f_{1}} 
:= E_{13}, 
&
& X_{2 f_{2}} 
:= E_{24},
\end{align*}
and for negative roots $-\alpha$ ($\alpha \in \Sigma^{+}$), 
define $X_{-\alpha} = {}^{t} X_{\alpha}$. 

As a basis of $\lie{k}$, we choose 
\begin{align*}
& 
D_{i} 
:= \I(X_{-2 f_{i}}-X_{2f_{i}}), \quad i=1,2, 
\\
& 
Y_{\pm} 
:= 
\frac{1}{2}
\{\mp (X_{-f_{1}+f_{2}} - X_{f_{1}-f_{2}})
+\I (X_{-f_{1}-f_{2}}-X_{f_{1}+f_{2}})
\}. 
\end{align*}
Under the identification \eqref{eq:K=U(2)}, 
these vectors correspond to $E_{ii}, E_{12}, E_{21}$ of 
$\lie{gl}(2,\C)$, 
respectively. 

Let $T := \{\exp(x_{1} D_{1} + x_{2} D_{2}) 
\mid x_{1}, x_{2} \in \R\}$ be a Cartan subgroup of $K$ (also of $G$). 
Define a basis $\{e_{1}, e_{2}\}$ of $\lie{t}^{\ast}$ 
by $e_{i}(D_{j}) = \delta_{i,j}$. 

Now, we fix a basis of an irreducible representation of $K=U(2)$. 
Let $\{v_{q}^{\lambda}\}$ be the basis of irreducible representation 
of $SU(2)$ with the highest weight $2\lambda$, 
which is defined in Section~\ref{section:representation of K}. 
Let $1_{\mu}$ be the basis of the one dimensional representation of 
the center of $U(2)$ defined by 
\[
\begin{pmatrix}
a & 0 \\ 0 & a 
\end{pmatrix} 
\cdot 1_{\mu} 
= 
a^{\mu}, 
\qquad 
\begin{pmatrix}
a & 0 \\ 0 & a 
\end{pmatrix} 
\in U(2). 
\]
Let $(\tau_{(\lambda_{1},\lambda_{2})}, V_{(\lambda_{1},\lambda_{2})}^{U(2)})$ 
be the irreducible representation 
of $K=U(2)$ with the highest weight $\lambda=(\lambda_{1},\lambda_{2})$. 
We define a basis 
$\{v_{i}^{(\lambda_{1},\lambda_{2})} \mid \lambda_{2} \leq i \leq \lambda_{1}\}$ 
of it by 
\begin{equation}\label{eq:basis of repn of U(2)}
v_{q}^{(\lambda_{1},\lambda_{2})} 
:= 
v_{q-(\lambda_{1}+\lambda_{2})/2}^{(\lambda_{1}-\lambda_{2})/2} 
\otimes 
1_{\lambda_{1}+\lambda_{2}}. 
\end{equation}
The action of elements in $\lie{k}$ is given by 
\begin{align*}
&
D_{1}\, v_{q}^{(\lambda_{1},\lambda_{2})} 
= 
q \, v_{q}^{(\lambda_{1},\lambda_{2})}, 
&
&
D_{2}\, v_{q}^{(\lambda_{1},\lambda_{2})} 
= 
(\lambda_{1}+\lambda_{2}-q)\, v_{q}^{(\lambda_{1},\lambda_{2})}, 
\\
&
Y_{+}\, v_{q}^{(\lambda_{1},\lambda_{2})} 
= 
(\lambda_{1}-q)\, v_{q+1}^{(\lambda_{1},\lambda_{2})}, 
&
&
Y_{-}\, v_{q}^{(\lambda_{1},\lambda_{2})} 
= 
(q-\lambda_{2})\, v_{q-1}^{(\lambda_{1},\lambda_{2})}. 
\end{align*}

\section{Irreducible $\brgK$-modules of $Sp(2,\R)$}
\label{section:irreducible modules of Sp(2)}

There are four conjugacy classes of Cartan subgroups in $G=Sp(2,\R)$. 
One of them is the split Cartan subgroup $H_{s} = M_{m} A_{m}$, 
where 
\[
M_{m} := Z_{K}(A_{m}) 
= 
\{I, m_{-1,1}^{sp}, m_{1,-1}^{sp}, m_{-1,-1}^{sp}\}, 
\quad 
m_{i,j}^{sp} := \diag(i,j,i,j). 
\]

Let $\widehat{M_{m}}$ be the set of equivalence classes of 
irreducible representations of $M_{m}$. 
Define $\sigma_{i,j} \in \widehat{M_{m}}$, $i,j \in \Z$, by 
\begin{align*}
& 
\sigma_{i,j}(m_{-1,1}^{sp}) = (-1)^{i}, 
&
&
\sigma_{i,j}(m_{1,-1}^{sp}) = (-1)^{j}. 
\end{align*}
Then $\widehat{M_{m}}$ consists of four elements $\sigma_{i,j}$, 
$i,j \in \{0,1\}$. 
The action of $W(G, H_{s}) \simeq W(B_{2})$ on 
$\widehat{M_{m}}$ is given by 
\begin{align}
&
r_{1,2} \cdot \sigma_{i,j} = \sigma_{j,i}, 
&
&
\epsilon_{1} \cdot \sigma_{i,j} 
= 
\epsilon_{2} \cdot \sigma_{i,j} 
= 
\sigma_{i,j}, 
\label{eq:action of W on hat{M}, Sp(2)}
\end{align}
where $r_{1,2}$ is the permutation of $1$ and $2$, 
and $\epsilon_{i}$, $i=1,2$, is the reflection with respect to 
 the root $2 f_{i}$. . 

Let 
\begin{align}
\label{eq:infinitesimal character Lambda, Sp(2)} 
&
\Lambda 
= 
\Lambda_{1} f_{1} + \Lambda_{2} f_{2}, 
&
&
\Lambda_{1}, \Lambda_{2} \in \Z, 
\quad 
\Lambda_{1} > \Lambda_{2} > 0 
\end{align}
be a nonsingular integral infinitesimal character. 
There are four conjugacy classes of regular characters of 
$H_{s}$ with a given infinitesimal character, 
which correspond to the ATLAS numbers 
"10", "11" in the block $PSO(3,2)$, 
"4" in the block $PSO(4,1)$ and "0" in the block $PSO(5)$. 
According to these numbers, we define regular characters 
$\gamma_{1}$, $\gamma_{11}$, $\gamma_{4'}$ and $\gamma_{0''}$ 
of $H_{s}$ by 
\begin{align}
& 
\gamma_{10} 
=   
(\sigma_{\Lambda_{1},\Lambda_{2}+1}, (\pm \Lambda_{1}, \pm \Lambda_{2})) 
\sim 
(\sigma_{\Lambda_{2}+1,\Lambda_{1}}, (\pm \Lambda_{2}, \pm \Lambda_{1})),  
\label{eq:gamma_{10}, Sp(2)} 
\\
& 
\gamma_{11} 
= 
(\sigma_{\Lambda_{2},\Lambda_{1}+1}, (\pm \Lambda_{2},\pm \Lambda_{1})) 
\sim 
(\sigma_{\Lambda_{1}+1,\Lambda_{2}}, (\pm \Lambda_{1},\pm \Lambda_{2})), 
\label{eq:gamma_{11}. Sp(2)} 
\\
& 
\gamma_{4'} 
= 
(\sigma_{\Lambda_{1}+1,\Lambda_{2}+1}, (\pm \Lambda_{1}, \pm \Lambda_{2})) 
\sim 
(\sigma_{\Lambda_{2}+1,\Lambda_{1}+1}, (\pm \Lambda_{2}, \pm \Lambda_{1})), 
\label{eq:gamma_{4'}, Sp(2)}
\\
& 
\gamma_{0''} 
= 
(\sigma_{\Lambda_{1},\Lambda_{2}}, (\pm \Lambda_{1}, \pm \Lambda_{2})) 
\sim 
(\sigma_{\Lambda_{2},\Lambda_{1}}, (\pm \Lambda_{2}, \pm \Lambda_{1})). 
\label{eq:gamma_{0''}, Sp(2)} 
\end{align}
Here, the notation is the same as in the case of $SL(3,\R)$. 
Note that the lengths of these regular characters 
are all three.

Next, consider the fundamental Cartan subgroup 
$H_{f} = T \simeq U(1)^{\times 2}$. 
There are four conjugacy classes of regular characters of $H_{f}$ 
with a given nonsingular integral infinitesimal character, 
which correspond to the ATLAS numbers "0", "1", "2" and "3" 
in the block $PSO(3,2)$. 
According to these numbers, we define regular characters 
$\gamma_{j}$, $j=0,1,2,3$, of $H_{f}$ by 
\begin{align*}
& 
\gamma_{0} 
= 
(\Lambda_{1}+1,-\Lambda_{2}), 
&
&
\gamma_{1} = (\Lambda_{2}, -\Lambda_{1}-1), 
\\
&
\gamma_{2} = (\Lambda_{1}+1,\Lambda_{2}+2), 
&
&
\gamma_{3} = (-\Lambda_{2}-2,-\Lambda_{1}-1). 
\end{align*}
Here we denoted $\gamma_{j} \in \widehat{H_{f}}$ by its differential. 
The standard module $X(\gamma_{2})$ and $X(\gamma_{3})$ are 
the holomorphic and anti-holomorphic discrete series, 
and $X(\gamma_{0})$ and $X(\gamma_{1})$ are large discrete series modules. 

There are two other Cartan subgroups of $Sp(2,\R)$. 
One is $H_{long} = T_{long} A_{long}$, 
which is the Cayley transform of the compact Cartan subgroup 
through a long non-compact imaginary root. 
The groups $T_{long}$ and $A_{long}$ 
are isomorphic to $U(1) \times \{\pm 1\}$ and $\R_{>0}$, respectively, 
so $H_{long}$ has two connected components. 
The Levi subgroup $L_{long} := Z_{G}(A_{long})$ 
is isomorphic to $SL(2,\R) \times \{\pm 1\} \times \R_{>0}$. 
The Weyl group $W(G,H_{long})$ is isomorphic to $\mathfrak{S}_{2}$, 
which is the reflection on $\lie{a}_{long}$. 
Therefore, when a nonsingular integral infinitesimal character 
is fixed, there are $8/2 \times 2 = 8$ conjugacy classes of 
regular characters attached to $H_{long}$. 
They correspond to the ATLAS numbers 
"5", "6", "7", "8" in the block $PSO(3,2)$ an 
"0", "1", "2", "3" in the block $PSO(4,1)$. 
According to these numbers, we define regular characters 
$\gamma_{j}$, 
$j=5, 6, 7, 8, 0', 1', 2', 3'$ by 
\begin{align*}
& 
\gamma_{5} = (\Lambda_{1}+1, \sgn^{\Lambda_{2}}, \Lambda_{2}), 
&
&
\gamma_{6} = (-\Lambda_{1}-1, \sgn^{\Lambda_{2}}, \Lambda_{2}), 
\\
&
\gamma_{7} = (\Lambda_{2}+1, \sgn^{\Lambda_{1}}, \Lambda_{1}), 
&
&
\gamma_{8} = (-\Lambda_{2}-1, \sgn^{\Lambda_{1}}, \Lambda_{1}), 
\\
&
\gamma_{0'} 
:= 
(\Lambda_{1}+1, \sgn^{\Lambda_{2}+1}, \Lambda_{2}), 
&
&
\gamma_{1'} 
:= 
(-\Lambda_{1}-1, \sgn^{\Lambda_{2}+1}, \Lambda_{2}), 
\\
&
\gamma_{2'} 
:= 
(\Lambda_{2}+1, \sgn^{\Lambda_{1}+1}, \Lambda_{1}), 
&
&
\gamma_{3'} 
:= 
(-\Lambda_{2}-1, \sgn^{\Lambda_{1}+1}, \Lambda_{1}). 
\end{align*}
Here, the expression 
$\gamma_{j} 
= 
(\epsilon_{j}(\Lambda_{a(j)}+1), 
\sgn^{b(j)}, \Lambda_{c(j)})$ 
means that the restriction of $\gamma_{j}$ to 
$T_{long} \simeq U(1) \times \{\pm 1\}$ is 
$(\epsilon_{j}(\Lambda_{a(j)}+1), \sgn^{b(j)}) 
\in \widehat{U(1)} \times \widehat{\{\pm 1\}}$ 
and its restriction to $\lie{a}_{long}$ is 
$\Lambda_{c(j)}$. 
The lengths of these regular characters are 
$\ell(\gamma_{5}) = \ell(\gamma_{6}) 
= \ell(\gamma_{0'}) = \ell(\gamma_{1'}) = 1$ 
and $\ell(\gamma_{7}) = \ell(\gamma_{8}) 
= \ell(\gamma_{2'}) = \ell(\gamma_{3'}) = 2$. 

The other intermediate Cartan subgroup is 
$H_{short} = T_{short} A_{short}$, 
which is the Cayley transform of 
the compact Cartan subgroup through a short non-compact imaginary root. 
The groups $T_{short}$ and $A_{short}$ are 
isomorphic to $SO(2)$ and $\R_{>0}$, respectively, 
so $H_{short}$ is connected. 
The Levi subgroup 
$L_{short} := Z_{G}(A_{short})$ is isomorphic to 
$GL(2,\R) \simeq SL^{\pm}(2,\R) \times \R_{>0}$. 
The Weyl group $W(G,H_{short})$ is isomorphic to 
$\mathfrak{S}_{2} \times \mathfrak{S}_{2}$, 
generated by the reflections on $\lie{t}_{short}$ and 
$\lie{a}_{short}$. 
Therefore, when a nonsingular integral infinitesimal character is fixed, 
there are $8/4=2$ conjugacy classes of regular characters. 
They correspond to the ATLAS numbers "4" and "9" in the block $PSO(3,2)$. 
According to these numbers, we define regular characters 
$\gamma_{4}$ and $\gamma_{9}$ of $H_{short}$ by 
\begin{align*}
&
\gamma_{4} 
= 
(\Lambda_{1}+\Lambda_{2}+1, \Lambda_{1}-\Lambda_{2}), 
&
&
\gamma_{9} 
= 
(\Lambda_{1}-\Lambda_{2}+1, \Lambda_{1}+\Lambda_{2}) 
\in 
\widehat{SO(2)} \times \lie{a}_{short}. 
\end{align*}
The length of $\gamma_{4}$ is one and that of $\gamma_{9}$ is two.

\section{$K$-types of the irreducible modules of $Sp(2,\R)$}
\label{section:K-spectra, Sp(2)} 

In this section, we calculate the $K$-spectra of 
the irreducible modules of $Sp(2,\R)$. 
For the determination of the socle filtration of principal series modules, 
we need them in the case when the infinitesimal character $\Lambda$ is trivial 
and the highest weight $\lambda=(\lambda_{1},\lambda_{2})$ of 
the irreducible representation $V_{\lambda}^{U(2)}$ of 
$K \simeq U(2)$ satisfies $|\lambda_{1}|, |\lambda_{2}| \leq 3$. 
The $K$-spectra of irreducible modules can be computed 
by using the KLV-conjecture and the Blattner formula. 

\begin{theorem}
\label{theorem:composition factors for Sp(2)}
In the Grothendieck group, the standard modules 
$X(\gamma_{i})$ decomposes into irreducible modules as follows: 
\begin{enumerate}
\item 
{\rm (Block $PSO(3,2)$)} 
\begin{align*}
&
X(\gamma_{i}) = \overline{X}(\gamma_{i}), 
\quad i=0,1,2,3, 
\\
& 
X(\gamma_{4}) = 
\overline{X}(\gamma_{0}) + \overline{X}(\gamma_{1}) 
+ \overline{X}(\gamma_{4}), 
\\
&
X(\gamma_{5}) = 
\overline{X}(\gamma_{0}) + \overline{X}(\gamma_{2}) 
+ \overline{X}(\gamma_{5}), 
\qquad
X(\gamma_{6}) = 
\overline{X}(\gamma_{1}) + \overline{X}(\gamma_{3}) 
+ \overline{X}(\gamma_{6}), 
\\
&
X(\gamma_{7}) = 
\overline{X}(\gamma_{0}) + \overline{X}(\gamma_{4})
+\overline{X}(\gamma_{5}) + \overline{X}(\gamma_{7}), 
\\
&
X(\gamma_{8}) 
= 
\overline{X}(\gamma_{1})+\overline{X}(\gamma_{4})
+\overline{X}(\gamma_{6})+\overline{X}(\gamma_{8}), 
\\
&
X(\gamma_{9})
= 
\overline{X}(\gamma_{0}) 
+ \overline{X}(\gamma_{1}) + \overline{X}(\gamma_{4})
+ \overline{X}(\gamma_{5}) + \overline{X}(\gamma_{6})
+ \overline{X}(\gamma_{9}), 
\\
&
X(\gamma_{10})
= 
\overline{X}(\gamma_{0}) + \overline{X}(\gamma_{1}) 
+ 2 \times \overline{X}(\gamma_{4}) + \overline{X}(\gamma_{5}) 
+ \overline{X}(\gamma_{6}) 
\\
& \hspace{20mm}
+ \overline{X}(\gamma_{7}) 
+ \overline{X}(\gamma_{8}) + \overline{X}(\gamma_{9}) 
+ \overline{X}(\gamma_{10}), 
\\
&
X(\gamma_{11})
= 
\overline{X}(\gamma_{0}) + \overline{X}(\gamma_{1}) 
+ \overline{X}(\gamma_{2}) + \overline{X}(\gamma_{3}) 
\\
& \hspace{20mm}
+ \overline{X}(\gamma_{4}) + \overline{X}(\gamma_{5}) 
+ \overline{X}(\gamma_{6}) + \overline{X}(\gamma_{9}) 
+ \overline{X}(\gamma_{11}). 
\end{align*}
\item 
{\rm (Block $PSO(4,1)$)} 
\begin{align*}
&
X(\gamma_{i}) = \overline{X}(\gamma_{i}), 
\quad i=0', 1', 
\\
& 
X(\gamma_{2'}) 
= 
\overline{X}(\gamma_{0'}) + \overline{X}(\gamma_{2'}), 
\qquad 
X(\gamma_{3'}) 
= 
\overline{X}(\gamma_{1'}) + \overline{X}(\gamma_{3'}), 
\\
&
X(\gamma_{4'}) 
= 
\overline{X}(\gamma_{0'}) + \overline{X}(\gamma_{1'})
+
\overline{X}(\gamma_{2'}) + \overline{X}(\gamma_{3'}) 
+ \overline{X}(\gamma_{4'}). 
\end{align*}
\item
{\rm (Block $PSO(5)$)} 
\[
X(\gamma_{0''}) = \overline{X}(\gamma_{0''}). 
\]
\end{enumerate}
\end{theorem}

Let us calculate the $K$-spectra of the irreducible modules. 
The multiplicities of $K$-types 
in a discrete series module $X(\gamma_{j})$, $j=0,1,2,3$, are given by 
the Blattner formula (\cite{Schmid}) 
\begin{equation}
\label{eq:Blattner formula}
m(\lambda) 
= 
\sum_{w \in W(K,T_{f})} 
\epsilon(w)\, Q(w (\lambda+\rho_{c})-\overline{\gamma}-\rho_{n}). 
\end{equation}

Secondly, consider the standard modules 
$X(\gamma_{4})$ and $X(\gamma_{9})$. 
These are the generalized principal 
series modules 
\begin{align*}
&
X(\gamma_{i}) 
= 
\Ind_{P_{short}}^{G}(
\pi_{DS}^{SL^{\pm}(2,\R)}(\Lambda_{1}+\epsilon_{i} \Lambda_{2}+1) 
\otimes e^{\Lambda_{1}-\epsilon_{i} \Lambda_{2}+1}),
&
&
\epsilon_{i} 
= 
\left\{
\begin{matrix}
1 & (i=4) 
\\
-1 & (i=9) 
\end{matrix}
\right. ,
\end{align*}
where $P_{short} = M_{short} A_{short} N_{short}$ is a parabolic subgroup 
of $G$ with the Levi factor $L_{short} = M_{short} A_{short}$, 
$\pi_{DS}^{SL^{\pm}(2,\R)}(\alpha)$ is the discrete series module 
of $SL^{\pm}(2,\R)$ with the minimal $K \cap L_{short} \simeq O(2)$-type 
$\alpha$, 
and $e^{\beta}$ is the one dimensional representation of 
$A_{short} \simeq \R_{> 0}$ defined by 
$\R_{>0} \ni a \mapsto e^{\beta \cdot \log a}$. 
Therefore, their $K$-spectra are 
\[
\widehat{K}(X(\gamma_{i})) 
= 
\{\lambda=(\lambda_{1},\lambda_{2}) \in \widehat{U(2)} 
\mid 
\lambda_{1}-\lambda_{2} 
\in 
\Lambda_{1}+\epsilon_{i} \Lambda_{2}+1 + 2 \Z_{\geq 0}
\}, 
\]
with the multiplicities 
\begin{equation}
\label{eq:multi. of X(gamma_{4})}
m_{\lambda} 
= 
\frac{1}{2}(
\lambda_{1}-\lambda_{2}-\Lambda_{1}-\epsilon_{i} \Lambda_{2}+1). 
\end{equation}

Thirdly, 
consider the standard modules $X(\gamma_{i})$, $i=5,6,7,8,0',1',2',3'$. 
These are generalized principal series modules induced from 
a parabolic subgroup 
$P_{long} = M_{long} A_{long} N_{long}$ of $G$, 
where $M_{long} \simeq SL(2,\R) \times \{\pm 1\}$ 
and $A_{long} \simeq \R_{>0}$. 
In Section~\ref{section:irreducible modules of Sp(2)}, 
we wrote the regular characters $\gamma_{j}$, $j=5,6,7,8,0',1',2',3'$, 
in a form like 
\[
\gamma_{j} = (\epsilon_{j} (\Lambda_{a(j)}+1), \sgn^{b(j)}, \Lambda_{c(j)}), 
\qquad 
\epsilon_{i} = +1 \mbox{ or } -1.
\]
Under this notation, 
\begin{align*}
X(\gamma_{j}) 
= 
\Ind_{P_{long}}^{G} 
(\pi_{DS}^{SL(2,\R)}(\epsilon_{j} (\Lambda_{a(j)}+1)) \otimes \sgn^{b(j)} 
\otimes e^{\Lambda_{c(j)}+1}). 
\end{align*}
Therefore, the $K$-spectra $\widehat{K}(X(\gamma_{j}))$ are 
\begin{align*}
\{\lambda=(\lambda_{1},\lambda_{2}) \in \widehat{U(2)} 
\mid 
& \exists l \in \Z_{\geq 0} \mbox{ s.t. } 
\lambda_{1} \geq \epsilon_{i} (\Lambda_{a(j)}+1+2l) \geq \lambda_{2}, 
\\ 
& \mbox{ and } \ 
\lambda_{1}+\lambda_{2} \equiv \Lambda_{a(j)}+b(j)+1 \ (\modulo 2)
\}
\end{align*}
and the multiplicities are 
\begin{align*}
m_{\lambda} 
= 
\#
\{l \in \Z_{\geq 0} 
\mid 
\lambda_{1} \geq \epsilon_{j}(\Lambda_{a(j)}+1+2l) \geq \lambda_{2}\}. 
\end{align*}

Finally, consider the standard modules $X(\gamma_{j})$, 
$j=10,11,4',0''$. 
These are the principal series modules induced from 
a minimal parabolic subgroup of $G$. 
In Section~\ref{section:irreducible modules of Sp(2)}, 
we wrote the regular characters $\gamma_{j}$, $j=10,11,4',0''$, 
in a form like 
\[
\gamma_{j} 
= 
(\sigma_{a,b}, \nu), 
\qquad 
\sigma_{a,b} \in \widehat{M_{m}}, 
\quad 
\nu \in W(G,A_{m}) \cdot \Lambda.  
\]
Under this notation, 
\[
X(\gamma_{j})
\approx
\Ind_{P_{m}}^{G}(\sigma_{a,b} \otimes e^{\nu+\rho_{m}}). 
\]

Under the identification \eqref{eq:K=U(2)}, 
the elements $m_{\epsilon_{1},\epsilon_{2}} \in M_{m}$, 
$\epsilon_{1}, \epsilon_{2} \in \{\pm 1\}$, correspond to 
$\diag(\epsilon_{1},\epsilon_{2}) \in U(2)$. 
Therefore, Lemma~\ref{lemma:the action of elements on repn of SU(2)} 
implies the following lemma. 
\begin{lemma}\label{lemma:the action of M on f.d. repn of K, SP(2)} 
\begin{enumerate}
\item
The action of $M_{m}$ on the bases of 
$V_{(\lambda_{1},\lambda_{2})}^{U(2)}$ is given by 
\[
\tau_{\lambda}(m_{\epsilon_{1}, \epsilon_{2}})\, 
v_{q}^{(\lambda_{1},\lambda_{2})} 
= 
\epsilon_{1}^{q} \epsilon_{2}^{\lambda_{1}+\lambda_{2}-q} 
v_{q}^{(\lambda_{1},\lambda_{2})}. 
\]
\item
If $\lambda_{1}+\lambda_{2}$ is even, 
then the restriction of $V_{(\lambda_{1},\lambda_{2})}^{U(2)}$ 
to $M_{m}$ decomposes into 
$\sigma_{0,0}$ and $\sigma_{1,1}$-isotypic subspaces: 
\[
V_{(\lambda_{1},\lambda_{2})}^{U(2)} 
= 
V_{\lambda}^{U(2)}(\sigma_{0,0}) 
\oplus 
V_{\lambda}^{U(2)}(\sigma_{1,1}), 
\]
where 
\begin{align*}
& 
V_{\lambda}^{U(2)}(\sigma_{0,0}) 
= 
\Span{
v_{q}^{(\lambda_{1}, \lambda_{2})} 
\mid 
q \mbox{ is even}},
&
& 
V_{\lambda}^{U(2)}(\sigma_{1,1}) 
= 
\Span{
v_{q}^{(\lambda_{1},\lambda_{2})} 
\mid 
q \mbox{ is odd}}. 
\end{align*}
\item
If $\lambda_{1}+\lambda_{2}$ is odd, 
then the restriction of $V_{(\lambda_{1},\lambda_{2})}^{U(2)}$ 
to $M_{m}$ decomposes into 
$\sigma_{0,1}$ and $\sigma_{1,0}$-isotypic subspaces: 
\[
V_{(\lambda_{1},\lambda_{2})}^{U(2)} 
= 
V_{\lambda}^{U(2)}(\sigma_{0,1}) 
\oplus 
V_{\lambda}^{U(2)}(\sigma_{1,0}), 
\]
where 
\begin{align*}
& 
V_{\lambda}^{U(2)}(\sigma_{0,1}) 
= 
\Span{
v_{q}^{(\lambda_{1}, \lambda_{2})} 
\mid 
q \mbox{ is even}},
&
& 
V_{\lambda}^{U(2)}(\sigma_{1,0}) 
= 
\Span{
v_{q}^{(\lambda_{1},\lambda_{2})} 
\mid 
q \mbox{ is odd}}. 
\end{align*}
\end{enumerate}
\end{lemma}

By this lemma, the $K$-spectra 
$\widehat{K}(X(\gamma_{i}))
=\widehat{K}(\Ind_{P_{m}}^{G}(\sigma_{a,b} \otimes e^{\nu+\rho_{m}}))$ 
are 
\begin{align*}
&
\{\lambda = (\lambda_{1}, \lambda_{2}) 
\mid 
\lambda_{1}=\lambda_{2} \in 2 \Z 
\ 
\mbox{ or } 
\
\lambda_{1}-\lambda_{2} \in 2 \Z_{>0}\} 
&
&
\mbox{if } 
\ 
a \equiv b \equiv 0 \ (\modulo 2), 
\\
&
\{\lambda = (\lambda_{1}, \lambda_{2}), 
\mid 
\lambda_{1}=\lambda_{2} \in 2\Z + 1 
\ 
\mbox{ or } 
\
\lambda_{1}-\lambda_{2} \in 2 \Z_{>0}\} 
&
&
\mbox{if } 
\ 
a \equiv b \equiv 1 \ (\modulo 2), 
\\
&
\{\lambda = (\lambda_{1}, \lambda_{2}) 
\mid 
\lambda_{1}-\lambda_{2} \in 1 + 2 \Z_{\geq 0}\} 
&
&
\mbox{if } 
\ 
a + b \equiv 1 \ (\modulo 2), 
\end{align*}
and the multiplicities are 
\begin{align*}
&
m_{\lambda} 
= 
\left\{
\begin{matrix}
\frac{1}{2}(\lambda_{1}-\lambda_{2})+1 
& 
\mbox{if } \ 
\lambda_{1}, \lambda_{2} \ \mbox{ are even}
\\
\frac{1}{2}(\lambda_{1}-\lambda_{2}) 
&
\mbox{if } \ 
\lambda_{1}, \lambda_{2} \ \mbox{ are odd}
\end{matrix}
\right. , 
&
&
\mbox{if } \ 
a \equiv b \equiv 0 \ (\modulo 2), 
\\
&
m_{\lambda} 
= 
\left\{
\begin{matrix}
\frac{1}{2}(\lambda_{1}-\lambda_{2}),  
& 
\mbox{if } \ 
\lambda_{1}, \lambda_{2} \ \mbox{ are even}
\\
\frac{1}{2}(\lambda_{1}-\lambda_{2}) + 1 
&
\mbox{if } \ 
\lambda_{1}, \lambda_{2} \ \mbox{ are odd}
\end{matrix}
\right. , 
&
&
\mbox{if } \ 
a \equiv b \equiv 1 \ (\modulo 2), 
\\
&
m_{\lambda} 
= 
\frac{1}{2}(\lambda_{1}-\lambda_{2}+1), 
&
&
\mbox{if } \ 
a+b \equiv 1 \ (\modulo 2). 
\end{align*}

From these data and Theorem~\ref{theorem:composition factors for Sp(2)}, 
we obtain the $K$-spectra of all irreducible modules. 

\begin{proposition}
\label{proposition:K-spectra of D.S. of Sp(2)}
Suppose that the infinitesimal character is trivial, 
namely $\Lambda=(\Lambda_{1},\Lambda_{2})=(2,1)$. 
The multiplicities of the $K$-types $(\tau_{\lambda}, V_{\lambda}^{U(2)})$ 
with $\lambda=(\lambda_{1},\lambda_{2})$, 
$|\lambda_{1}|, |\lambda_{2}| \leq 3$, 
in the irreducible modules are as follows: 
\begin{enumerate}
\item 
{\rm Block $PSO(3,2)$} 
\begin{center}
\begin{tabular}{|c|c||c|c|}
\hline
{\rm Irred.Mod.} & $\lambda$ & {\rm Irred.Mod.} & $\lambda$ 
\\ \hline \hline 
\ 
\vspace{-4mm} 
& & & 
\\
$\overline{X}(\gamma_{0})$ & $(3,-1)$, $(3,-3)$ 
&
$\overline{X}(\gamma_{1})$ & $(1,-3)$, $(3,-3)$ 
\\ \hline 
\ \vspace{-4mm} 
& & & 
\\
$\overline{X}(\gamma_{2})$ & $(3,3)$ 
&
$\overline{X}(\gamma_{3})$ & $(-3,-3)$ 
\\ \hline
\ \vspace{-3.9mm} 
& & & 
\\
$\overline{X}(\gamma_{4})$ & $(2,-2)$ 
&
$\overline{X}(\gamma_{5})$ & $(3,1)$ 
\\ \hline 
\ \vspace{-4mm} 
& & & 
\\
$\overline{X}(\gamma_{6})$ & $(-1,-3)$ 
&
$\overline{X}(\gamma_{7})$ & $(2,2)$, $(2,0)$ 
\\ \hline 
& & & $(2,0)$, $(1,-1)$, $(0,-2)$
\\ 
$\overline{X}(\gamma_{8})$ & $(-2,-2)$, $(0,-2)$ 
&
$\overline{X}(\gamma_{9})$ & $(3,-1)$, $(2,-2)$
\\
& & & 
$(1,-3)$, $(3,-3)$ 
\\ \hline 
\ \vspace{-4mm} 
& & & 
\\
$\overline{X}(\gamma_{10})$ & $(0,0)$ 
&
$\overline{X}(\gamma_{11})$ & $\pm (1,1)$, $(3,\pm 1)$, $(\pm 1, -3)$ 
\\
& & & $(1,-1)$, $(3,-3)$ 
\\ \hline 
\end{tabular}
\end{center}
The multiplicities of these $K$-types are all one. 
\item
{\rm Block $PSO(4,1)$} 
\begin{center}
\begin{tabular}{|c|c||c|c|}
\hline
{\rm Irred.Mod.} & $\lambda$ & {\rm Irred.Mod.} & $\lambda$ 
\\ \hline \hline 
\ \vspace{-4mm} 
& & & 
\\ 
$\overline{X}(\gamma_{0'})$ & $(3,\pm 2)$, $(3,0)$ 
&
$\overline{X}(\gamma_{1'})$ & $(\pm 2,-3)$, $(0,-3)$  
\\ \hline 
\ \vspace{-4mm} 
& & & 
\\ 
$\overline{X}(\gamma_{2'})$ & $(2,\pm 1)$, $(2,-3)$ 
&
$\overline{X}(\gamma_{3'})$ & $(\pm 1,-2)$, $(3,-2)$ 
\\ \hline 
$\overline{X}(\gamma_{4'})$ & 
\multicolumn{3}{c|}
{$(1,0)$, $(0,-1)$, $(3,0)$, $(2,-1)$}
\\ 
& \multicolumn{3}{c|}
{$(1,-2)$, $(0,-3)$, 
$(3,-2)$, $(2,-3)$} 
\\ \hline 
\end{tabular}
\end{center}
The multiplicities of these $K$-types are all one. 
\end{enumerate}
\end{proposition}

\section{Shift operators for $Sp(2,\R)$}
\label{section:shift operators, Sp(2)}
In this section, we write down the shift operators of $K$ types 
explicitly, in the case of $G = Sp(2,\R)$.

The adjoint representation of $\lie{k}$ on $\lie{s}$ decomposes 
into two irreducible representations $\lie{s}_{\pm}$, 
whose highest weights are $(2,0)$ and $(0,-2)$, respectively. 
Define 
\begin{align*}
& 
X_{\pm 2 e_{i}}
:= 
H_{i} \pm \I (X_{2f_{i}}+X_{-2f_{i}}) 
= 
H_{i} \pm D_{i} \pm 2 \I X_{2f_{i}}, 
\quad i=1,2, 
\\
&
X_{\pm (e_{1}+e_{2})}
:= 
\pm(X_{f_{1}-f_{2}}+X_{-f_{1}+f_{2}}) 
+ \I 
(X_{f_{1}+f_{2}}+X_{-f_{1}-f_{2}})
\\
& 
\hspace{16.1mm}
= 2(Y_{\mp} \pm X_{f_{1}-f_{2}} + \I X_{f_{1}+f_{2}}). 
\end{align*}
Then, 
$X_{\pm 2e_{1}}, X_{\pm 2 e_{2}}, X_{\pm(e_{1}+e_{2})}$ is a basis 
of $\lie{s}_{\pm}$, and we may identify 
\begin{align}
& 
v_{2}^{(2,0)}
= X_{2e_{1}}, 
&
&
v_{1}^{(2,0)}
= \frac{1}{2} 
X_{e_{1}+e_{2}}, 
&
&
v_{0}^{(2,0)}
=X_{2e_{2}},  
& &
\label{eq:basis of s-1, Sp(2)}\\
& 
v_{0}^{(0,-2)} 
= X_{-2e_{2}},
& 
&
v_{-1}^{(0,-2)} 
=\frac{1}{2} 
X_{-e_{1}-e_{2}},  
&
&
v_{-2}^{(0,-2)} 
=X_{-2e_{1}}. 
\label{eq:basis of s-2, Sp(2)}
\end{align}

Define an invariant bilinear form $\langle \ , \ \rangle$ on $\lie{g}$ 
by 
$\langle X, Y \rangle = \tr(XY)$. 
Then 
\begin{align*}
& 
\langle X_{2e_{1}}, X_{-2e_{1}} \rangle 
= 4, 
&
&
\langle X_{2e_{2}}, X_{-2e_{2}} \rangle 
= -8, 
&
&
\langle X_{e_{1}+e_{2}}, X_{-e_{1}-e_{2}} \rangle 
= 4. 
\end{align*}
The operator $\nabla$ defined in Section~\ref{section:outline} is 
\begin{align}
4 \nabla
\phi 
= 
& 
\sum_{i=1}^{2} 
\sum_{\epsilon \in \{\pm 1\}} 
L(H_{i} - \epsilon D_{i} - 2 \I \epsilon X_{2f_{i}}) \phi 
\otimes 
X_{2 \epsilon e_{i}}
\label{eq:nabla for Sp(2)}
\\
& 
- 
\sum_{\epsilon \in \{\pm 1\}} 
L(Y_{\epsilon} 
- \epsilon X_{f_{1}-f_{2}} + \I X_{f_{1}+f_{2}}) \phi 
\otimes 
X_{\epsilon (e_{1}+e_{2})}. 
\notag
\end{align}

The contragredient representation of
$(\tau_{(\lambda_{1},\lambda_{2})}, V_{(\lambda_{1},\lambda_{2})}^{U(2)})$ 
is 
$(\tau_{(-\lambda_{2},-\lambda_{1})}, 
V_{(-\lambda_{2},-\lambda_{1})}^{U(2)})$. 
Let $\mathrm{pr}_{(i,j)}$, $(i,j) = \pm (2,0), \pm (1,1)$ or $\pm(0,2)$, 
be the natural projection 
\[
(\tau_{(\lambda_{1},\lambda_{2})})^{\ast} 
\otimes 
\lie{s} 
\to 
\tau_{(-\lambda_{2}-j,-\lambda_{1}-i)} 
\simeq 
(\tau_{(\lambda_{1}+i,\lambda_{2}+j)})^{\ast}. 
\]
We define the shift operators of $K$-types 
$P_{(i,j)}$ by 
\begin{align*}
P_{(i,j)} 
&= 4 \times \pr_{(i,j)} \circ \nabla : 
\\
&C_{(\tau_{(\lambda_{1},\lambda_{2})})^{\ast}}^{\infty}
(K \backslash G / A_{m} N_{m}; e^{\nu+\rho}) 
\rightarrow 
C_{(\tau_{(\lambda_{1}+i,\lambda_{2}+j)})^{\ast}}^{\infty}
(K \backslash G / A_{m} N_{m}; e^{\nu+\rho}). 
\end{align*}

An element $\phi_{\lambda}(g)$ in 
$C^{\infty}_{(\tau_{(\lambda_{1},\lambda_{2})})^{\ast}}
(K \backslash G / A_{m} N_{m}; e^{\nu+\rho})$ is determined by its value 
at $g=e$ by Iwasawa decomposition. 
So we write it as 
\[
\phi_{\lambda}(e) 
= 
\sum_{q=-\lambda_{1}}^{-\lambda_{2}} 
c(q)\, v_{q}^{(-\lambda_{2},-\lambda_{1})}. 
\]
For notational convenience, 
we put $c(q) = 0$ if $q > -\lambda_{2}$ or 
$q < -\lambda_{1}$. 
By Lemma~\ref{lemma:irred. decomp. of tau otimes Ad} and 
\eqref{eq:basis of repn of U(2)}, 
the irreducible decompositions of 
$\tau_{(-\lambda_{2},-\lambda_{1})} \otimes \tau_{(2,0)}$ and 
$\tau_{(-\lambda_{2},-\lambda_{1})} \otimes \tau_{(0,-2)}$ 
are given by 
\begin{align*}
& 
v_{q}^{(-\lambda_{2},-\lambda_{1})} 
\otimes 
v_{i}^{(2,0)} 
= 
\sum_{j=0}^{2}
d^{+}(q,i;j)\, 
v_{q+i}^{(-\lambda_{2}+j,-\lambda_{1}+2-j)}, 
\\
& 
v_{q}^{(-\lambda_{2},-\lambda_{1})} 
\otimes 
v_{i}^{(0,-2)} 
= 
\sum_{j=-2}^{0}
d^{-}(q,i;j)\, 
v_{q+i}^{(-\lambda_{2}+j,-\lambda_{1}-2-j)}, 
\end{align*}
where 
\begin{align*}
& 
d^{+}(q,i;2) = 1 \quad (0 \leq i \leq 2), 
\qquad 
d^{-}(q,i;0) 
= 1 \quad (-2 \leq i \leq 0), 
\\
&
d^{+}(q,2;1) = d^{-}(q,0;-1) = -(\lambda_{2}+q), 
\\
&
d^{+}(q,2;0) = d^{-}(q,0;-2) = (\lambda_{2}+q+1)_{\downarrow (2)}, 
\\
& d^{+}(q,1;1) = d^{-}(q,-1;-1) 
= 
-\frac{1}{2}(\lambda_{1}+\lambda_{2}+2q), 
\\
&
d^{+}(q,1;0) = d^{-}(q,-1;-2) 
= (\lambda_{1}+q)(\lambda_{2}+q), 
\\
& 
d^{+}(q,0;1) = d^{-}(q,-2;-1) 
= -(\lambda_{1}+q), 
\\
& d^{+}(q,0;0) = d^{-}(q,-2;-2) 
= 
(\lambda_{1}+q)_{\downarrow (2)}. 
\end{align*}
Therefore, by 
\eqref{eq:basis of s-1, Sp(2)}, \eqref{eq:basis of s-2, Sp(2)} 
and \eqref{eq:nabla for Sp(2)}, 
we obtain the explicit form of the shift operators $P_{(i,j)}$ of 
$K$-types: 
\begin{align}
P_{(2,0)} \phi_{\lambda}(e) 
= 
\sum_{q=-\lambda_{1}-2}^{-\lambda_{2}} 
\{&
\left(\nu_{1}+2\lambda_{1}+4+q \right) 
c(q+2) 
\notag
\\
&
+(\nu_{2}+\lambda_{1}+\lambda_{2}+1+q)\, 
c(q)
\}\,
v_{q}^{(-\lambda_{2},-\lambda_{1}-2)}, 
\label{eq:p.s., P_{(2,0)}}
\end{align}
\begin{align}
P_{(1,1)} \phi_{\lambda}(e) 
= 
- 
\sum_{q=-\lambda_{1}-1}^{-\lambda_{2}-1}
\{&
(\lambda_{1}+q+2) 
(\nu_{1}+\lambda_{1}+\lambda_{2}+q+2 )\, 
c(q+2) 
\notag\\
&+(\lambda_{2}+q) 
(\nu_{2}+\lambda_{1}+\lambda_{2}+q+1)\, 
c(q)
\}\,
v_{q}^{(-\lambda_{2}-1,-\lambda_{1}-1)}, 
\label{eq:p.s., P_{(1,1)}}
\end{align}
\begin{align}
P_{(0,2)} \phi_{\lambda}(e) 
= 
\sum_{q=-\lambda_{1}}^{-\lambda_{2}-2} 
\{&
(\lambda_{1}+q+2)_{\downarrow (2)} 
(\nu_{1}+2\lambda_{2}+q+2 )\, 
c(q+2) 
\notag\\
& 
+ (\lambda_{2}+q+1)_{\downarrow (2)} 
(\nu_{2}+\lambda_{1}+\lambda_{2}+q+1)\, 
c(q)
\}\,
v_{q}^{(-\lambda_{2}-2,-\lambda_{1})}, 
\label{eq:p.s., P_{(0,2)}}
\end{align}
\begin{align}
P_{(0,-2)} \phi_{\lambda}(e) 
= 
\sum_{q=-\lambda_{1}}^{-\lambda_{2}+2} 
\{&
(\nu_{1}-2\lambda_{2}-q+4)\, 
c(q-2) 
\notag\\
& 
+ (\nu_{2}-\lambda_{1}-\lambda_{2}-q+1)\, 
c(q)
\}\,
v_{q}^{(-\lambda_{2}+2,-\lambda_{1})}, 
\label{eq:p.s., P_{(0,-2)}}
\end{align}
\begin{align}
P_{(-1,-1)} \phi_{\lambda}(e) 
= 
- \sum_{q=-\lambda_{1}+1}^{-\lambda_{2}+1} 
\{&
(\lambda_{2}+q-2) 
(\nu_{1}-\lambda_{1}-\lambda_{2}-q+2)\, 
c(q-2) 
\notag\\
& 
+ (\lambda_{1}+q) 
(\nu_{2}-\lambda_{1}-\lambda_{2}-q+1)\, 
c(q)
\}\,
v_{q}^{(-\lambda_{2}+1,-\lambda_{1}+1)}, 
\label{eq:p.s., P_{(-1,-1)}}
\end{align}
\begin{align}
P_{(-2,0)} \phi_{\lambda}(e) 
= 
\sum_{q=-\lambda_{1}+2}^{-\lambda_{2}} 
\{&
(\lambda_{2}+q-1)_{\downarrow (2)} 
(\nu_{1}-2\lambda_{1}-q+2)\, 
c(q-2) 
\notag\\
& 
+ (\lambda_{1}+q)_{\downarrow (2)} 
(\nu_{2}-\lambda_{1}-\lambda_{2}-q+1)\, 
c(q)
\}\,
v_{q}^{(-\lambda_{2},-\lambda_{1}+2)}. 
\label{eq:p.s., P_{(-2,0)}}
\end{align}

\section{Candidates for irreducible submodules. $Sp(2,\R)$ case.}
\label{section:candidates for submodules of PS}

In this section, we seek candidates for 
irreducible submodules of principal series modules. 
By the translation principle, 
we will restrict our calculation to the case 
when the infinitesimal character 
$\Lambda = (\Lambda_{1}, \Lambda_{2})$ 
is trivial, namely $\Lambda = \rho_{m} = (2,1)$. 

Firstly, consider the block $PSO(3,2)$. 
The regular characters of the split Cartan subgroup 
which are contained in this block are 
$\gamma_{10}$ and $\gamma_{11}$. 
If the infinitesimal character is trivial, they are 
\begin{align*}
& 
\gamma_{10} = (\sigma_{0,0}, (2,1)), 
&
&
\gamma_{11} = (\sigma_{1,1}, (2,1)), 
\end{align*}
so we consider the socle filtrations of the principal series modules 
\[
I(\sigma_{0,0}, w \cdot (2,1)) 
\quad \mbox{and} \quad 
I(\sigma_{1,1}, w \cdot (2,1)), 
\qquad 
w \in W(G,H_{s}) \simeq W(B_{2}). 
\]
For $G=Sp(2,\R)$, 
the automorphism $\mu$ in section~\ref{subsection:horizontal symmetry} 
is given by the involution 
\[
\mu : G \rightarrow G, 
\qquad 
\mu(g) = 
x_{\mu}\,  
g\, 
(x_{\mu})^{-1}, 
\qquad 
x_{\mu} 
:= 
\diag(1,1,-1,-1). 
\]
This automorphism acts on $H_{f}$ as the inverse map. 
Therefore 
$\gamma_{0} \circ \mu$ and $\gamma_{2} \circ \mu$ are 
conjugate to $\gamma_{1}$ and $\gamma_{3}$, respectively. 
Next, $\mu$ acts on $A_{long}$ trivially and on $T_{long}$ as 
the inverse map, up to inner automorphisms. 
Therefore, 
$\gamma_{5} \circ \mu = \gamma_{6}$, 
$\gamma_{7} \circ \mu = \gamma_{8}$, 
$\gamma_{0'} \circ \mu = \gamma_{1'}$ and 
$\gamma_{2'} \circ \mu = \gamma_{3'}$. 
Finally, $\mu$ acts on $H_{s}$ and $H_{short}$ trivially, 
up to inner automorphisms. 
Therefore, $\overline{X}(\gamma_{i})$, 
$i=4, 9, 10, 11, 4', 0''$, are self-dual. 
Moreover, since $\mu$ stabilizes $N_{m}$, 
\[
I(\sigma, \nu)^{\mu} 
\simeq 
I(\sigma, \nu). 
\]
By these and the horizontal symmetry, we have obtained the followings: 
\begin{lemma}\label{lemma:horizontal symmetry in Sp(2)} 
\begin{enumerate}
\item
$\overline{X}(\gamma_{i})$, $i=4, 9, 10, 11, 4', 0''$ are self-dual. 
\item
If 
$(i,j) = (0,1), (2,3), (5,6), (7,8), (0',1')$ or $(2',3')$, 
then 
$\overline{X}(\gamma_{i})$ and $\overline{X}(\gamma_{j})$ are 
dual. 
\item
In the socle filtration of the principal series module 
$I(\sigma, \nu)$, 
the modules $\overline{X}(\gamma_{i})$ and $\overline{X}(\gamma_{j})$, 
$(i,j) = (0,1), (2,3), (5,6), (7,8), (0',1')$ or $(2',3')$, 
appear in the same floors. 
Namely, they appear as a pair 
$\overline{X}(\gamma_{i}) \oplus \overline{X}(\gamma_{j})$. 
\end{enumerate}
\end{lemma}

\subsection{The irreducible modules 
$\overline{X}(\gamma_{2})$ and $\overline{X}(\gamma_{3})$.} 

If the infinitesimal character is trivial, 
then the minimal $K$-type of the holomorphic discrete series 
$\overline{X}(\gamma_{2})$ 
is $\lambda=(\Lambda_{1}+1,\Lambda_{2}+2) = (3,3)$, 
and $\lambda=(3,1)$ is not a $K$-type 
of it. 
Therefore, 
if $\overline{X}(\gamma_{2})$ is 
a submodule of $I(\sigma_{1,1}\otimes e^{-\nu-\rho})$, 
then there exists a non-zero function 
\[
\phi_{(3,3)}^{2}(kan) 
= 
a^{-\nu-\rho} 
\tau_{(-3,-3)}(k^{-1})\, 
c(-3)\, 
v_{-3}^{(-3,-3)} 
\]
which satisfies 
$P_{(0,-2)}\phi_{(3,3)}^{2} = 0$. 
By \eqref{eq:p.s., P_{(0,-2)}}, 
this is equivalent to 
\[
(P_{(0,-2)} \phi_{(3,3)}^{2})(e) 
= 
c(-3)
\left\{
(\nu_{1}-1)\, 
v_{-1}^{(-1,-3)} 
+
(\nu_{2}-2)\, 
v_{-3}^{(-1,-3)} 
\right\}
= 0. 
\]
Therefore, the principal series 
of which $\overline{X}(\gamma_{2})$ can be a submodule 
is $I(\sigma_{1,1}, (1,2))$. 
Since there is no other possibility, 
Casselman's subrepresentation theorem implies that 
$\overline{X}(\gamma_{2})$ is really a submodule of $I(\sigma_{1,1}, (1,2))$. 

By Lemma~\ref{lemma:horizontal symmetry in Sp(2)}, 
we have obtained the next lemma. 
\begin{lemma}\label{lemma:2,3 is sub}
If $\Lambda$ is trivial, 
$\overline{X}(\gamma_{2})$ and $\overline{X}(\gamma_{3})$ are submodules of 
$I(\sigma_{1,1}, (1,2)) = I(\sigma_{1,1}, (\Lambda_{2}, \Lambda_{1}))$, 
and there is no other principal series representation of which 
these irreducible modules $\overline{X}(\gamma_{2}), 
\overline{X}(\gamma_{3})$ are submodules. 
\end{lemma}

\subsection{The irreducible modules 
$\overline{X}(\gamma_{0})$ and $\overline{X}(\gamma_{1})$.} 

In the same way as in the previous subsection, 
we determine the principal series representations of which 
$\overline{X}(\gamma_{0})$ is a submodule.

If the infinitesimal character is trivial, 
then the minimal $K$-type of the large discrete series 
$\overline{X}(\gamma_{0})$ is 
$\lambda = (\Lambda_{1}+1,-\Lambda_{2}) = (3,-1)$, 
and $\lambda = (3,1), (1,-1), (2,-2)$ are not $K$-types of it. 
Therefore, if $\overline{X}(\gamma_{0})$ is a submodule of 
$I(\sigma_{0,0}, (\nu_{1},\nu_{2})) 
\oplus 
I(\sigma_{1,1}, (\nu_{1},\nu_{2}))$, 
then there exists a non-zero function 
\[
\phi_{(3,-1)}^{0}(kan) 
= 
a^{-\nu-\rho} \tau_{(1,-3)}(k^{-1}) 
\sum_{q=-3}^{1} 
c(q)\, 
v_{q}^{(1,-3)} 
\]
which satisfies 
\begin{align*}
&
P_{(0,2)} \phi_{(3,-1)}^{0} 
= 0 
&
& \mbox{and} 
&
&
P_{(-2,0)} \phi_{(3,-1)}^{0} 
= 0 
&
& \mbox{and} 
&
&
P_{(-1,-1)} \phi_{(3,-1)}^{0} 
= 0. 
\end{align*}
By \eqref{eq:p.s., P_{(0,2)}}, \eqref{eq:p.s., P_{(-2,0)}} and 
\eqref{eq:p.s., P_{(-1,-1)}}, 
this system of equations is equivalent to 
\begin{align}
& 
\begin{cases}
6(\nu_{1}-2)\, c(0) + 6(\nu_{2}+1)\, c(-2) = 0, 
\\
6(\nu_{1}-4)\, c(-2) + 6(\nu_{2}-1)\, c(0) = 0, 
\\
-(\nu_{1}-2)\, c(0) = 0, 
\\
-3 \nu_{1}\, c(-2) + 3(\nu_{2}-1)\, c(0) = 0, 
\\
(\nu_{2}+1)\, c(-2) = 0; 
\end{cases}
\label{eq:0 is sub,1}\\
& 
\begin{cases}
12(\nu_{1}-1)\, c(1) + 2(\nu_{2}+2)\, c(-1) = 0, 
\\
2(\nu_{1}-3)\, c(-1) + 12 \nu_{2}\, c(-3) = 0, 
\\
2(\nu_{1}-5)\, c(-1) + 12(\nu_{2}-2)\, c(1) = 0, 
\\
12(\nu_{1}-3)\, c(-3) + 2 \nu_{2}\, c(-1) = 0, 
\\
-2(\nu_{1}-1)\, c(-1) + 4(\nu_{2}-2)\, c(1) = 0, 
\\
-4(\nu_{1}+1)\, c(-3) + 2 \nu_{2}\, c(-1) = 0. 
\end{cases}
\label{eq:0 is sub,2}
\end{align}
\begin{lemma}\label{lemma:0 is sub}
\begin{enumerate}
\item
The system of equations \eqref{eq:0 is sub,1} 
has a non-zero solution if and only if 
$\nu = (2,1) = (\Lambda_{1},\Lambda_{2})$ or 
$(2,-1) = (\Lambda_{1},-\Lambda_{2})$. 
In these cases, the solutions are 
\begin{align*}
& 
\phi_{(3,-1)}^{0}(e) 
= \alpha_{1}  
v_{0}^{(1,-3)}, 
&
&
\phi_{(3,-1)}^{0}(e) 
= \alpha_{2}  
(v_{0}^{(1,-3)} - v_{-2}^{(1,-3)}), 
\end{align*}
$\alpha_{1}, \alpha_{2} \in \C$, 
respectively. 
\item
The system of equations \eqref{eq:0 is sub,2} 
has a non-zero solution if and only if 
$\nu = (2,1) = (\Lambda_{1},\Lambda_{2})$, 
$(2,-1) = (\Lambda_{1},-\Lambda_{2})$ or 
$(1,2) = (\Lambda_{2}, \Lambda_{1})$. 
In these cases, the solutions are 
\begin{align*}
& 
\phi_{(1,-3)}^{0}(e) 
= 
\alpha \times 
\begin{cases}
3 v_{1}^{(1,-3)} - 6 v_{-1}^{(1,-3)} - v_{-3}^{(1,-3)}, 
\\
v_{1}^{(1,-3)} - 6 v_{-1}^{(1,-3)} + v_{-3}^{(1,-3)}, 
\\
v_{1}^{(1,-3)}, 
\end{cases}
&
&
\alpha \in \C, 
\end{align*}
respectively. 
\end{enumerate}
\end{lemma}

By Lemmas~\ref{lemma:the action of M on f.d. repn of K, SP(2)}, 
\ref{lemma:horizontal symmetry in Sp(2)} and \ref{lemma:0 is sub}, 
we have obtained the next corollary. 
\begin{corollary}\label{corollary:0,1 are sub}
If the infinitesimal character $\Lambda$ is trivial, then 
$\overline{X}(\gamma_{0}) \oplus \overline{X}(\gamma_{1})$ is a submodule of 
$I(\sigma, \nu)$ only if 
\begin{align*}
(\sigma,\nu) = 
&(\sigma_{0,0}, (\Lambda_{1},\Lambda_{2})), 
(\sigma_{0,0}, (\Lambda_{1},-\Lambda_{2})), 
\\
&(\sigma_{1,1}, (\Lambda_{1},\Lambda_{2})), 
(\sigma_{1,1}, (\Lambda_{1},-\Lambda_{2})) 
\ \mbox{or} \  
(\sigma_{1,1}, (\Lambda_{2},\Lambda_{1})). 
\end{align*}
\end{corollary}

\subsection{The irreducible modules $\overline{X}(\gamma_{5})$ and 
$\overline{X}(\gamma_{6})$.} 

If the infinitesimal character is trivial, 
the representation $\overline{X}(\gamma_{5})$ has a $K$-type 
$\lambda=(3,1)$ with multiplicity one, 
and $\lambda=(3,3), (3,-1), (2,0), (1,1)$ are not $K$-types of it. 
Therefore, if $\overline{X}(\gamma_{5})$ is a submodule of 
$I(\sigma, \nu)$, then 
there exists a non-zero vector 
\[
\phi_{(3,1)}^{5}(kan) 
= 
a^{-\nu-\rho} \tau_{(-1,-3)}(k^{-1}) 
\sum_{q=-3}^{-1} 
c(q)\, 
v_{q}^{(-1,-3)} 
\]
which satisfies 
\begin{align*}
&
P_{(0,2)} \phi_{(3,1)}^{5} = 0, 
&
&
P_{(0,-2)} \phi_{(3,1)}^{5} = 0, 
&
&
P_{(-1,-1)} \phi_{3,1}^{5} = 0, 
& 
& 
\mbox{and}  
&
& 
P_{(-2,0)} \phi_{3,1}^{5} = 0. 
\end{align*}
By \eqref{eq:p.s., P_{(0,2)}}, 
\eqref{eq:p.s., P_{(0,-2)}}, \eqref{eq:p.s., P_{(-1,-1)}} 
and \eqref{eq:p.s., P_{(-2,0)}}, 
this system of equations is equivalent to 
\begin{align}
& 
\begin{cases}
(\nu_{1}+2)\, c(-2) = 0, 
\qquad 
(\nu_{2}-1)\, c(-2) = 0, 
\\
(\nu_{1}-2)\, c(-2) = 0, 
\qquad 
(\nu_{2}-1)\, c(-2) = 0, 
\end{cases}
\label{eq:5 is sub,1}
\\
& 
\begin{cases}
2(\nu_{1}+1)\, c(-1)+2(\nu_{2}+2)\, c(-3) = 0, 
\\
(\nu_{1}+1)\, c(-1) = 0, 
\\
(\nu_{1}+3)\, c(-3)+(\nu_{2}-2)\, c(-1) = 0, 
\\
\nu_{2}\, c(-3) = 0, 
\\
-2(\nu_{1}-1)\, c(-3)+2(\nu_{2}-2)\, c(-1) = 0, 
\\
2(\nu_{1}-3)\, c(-3)+2(\nu_{2}-2)\, c(-1) = 0. 
\end{cases}
\label{eq:5 is sub,2} 
\end{align}
\begin{lemma}\label{lemma:5 is sub}
\begin{enumerate}
\item
The solution of the 
system of equations \eqref{eq:5 is sub,1} 
is $\{0\}$. 
\item
The system of equations \eqref{eq:5 is sub,2} 
has a non-trivial solution if and only if 
$(\nu_{1},\nu_{2})=(-1,2)=(-\Lambda_{2},\Lambda_{1})$. 
In this case, the solution is 
\[
\phi_{(-1,-3)}^{5}(e) 
= 
\alpha\, v_{-1}^{(-1,-3)}, 
\qquad \alpha \in \C. 
\]
\item
If the infinitesimal character is trivial, 
then the module $\overline{X}(\gamma_{5}) \oplus \overline{X}(\gamma_{6})$ 
is a submodule of 
$I(\sigma, \nu)$ if and only if 
$(\sigma,\nu) = (\sigma_{1,1}, (-\Lambda_{2},\Lambda_{1}))$. 
\end{enumerate}
\end{lemma}
\begin{proof}
(3) By Lemma~\ref{lemma:horizontal symmetry in Sp(2)}, 
$\overline{X}(\gamma_{5})$ and $\overline{X}(\gamma_{6})$ appear 
as a pair. 
Therefore, we may consider only the embedding of $\overline{X}(\gamma_{5})$. 
The "only if" part follows from (1) and (2). 
By the subrepresentation theorem, $\overline{X}(\gamma_{5})$ 
is a submodule of at least one principal series representation. 
As a result of (1), (2), 
there is only one possibility, so it is the unique embedding. 
This completes the proof of (3). 
\end{proof}

\subsection{The irreducible module $\overline{X}(\gamma_{4})$.}

If the infinitesimal character is trivial, the irreducible 
representation $\overline{X}(\gamma_{4})$ 
has a $K$-type $\lambda=(2,-2)$ with multiplicity one, and 
$\lambda=(3,-1), (2,0), (1,-3), (0,-2)$ are not a $K$-type of it. 
Therefore, if $\overline{X}(\gamma_{4})$ is a submodule of 
$I(\sigma,\nu)$, 
then there exists a non-zero function 
\[
\phi_{(2,-2)}^{4}(kan) 
= 
a^{-\nu-\rho} \tau_{(2,-2)}(k^{-1}) 
\sum_{q=-2}^{2} 
c(q)\, 
v_{q}^{(2,-2)} 
\]
which satisfies 
\begin{align*}
& P_{(1,1)} \phi_{(2,-2)}^{4} = 0, 
&
& P_{(0,2)} \phi_{(2,-2)}^{4} = 0, 
&
& P_{(-1,-1)} \phi_{(2,-2)}^{4} = 0, 
&
& P_{(-2,0)} \phi_{(2,-2)}^{4} = 0. 
\end{align*}
By \eqref{eq:p.s., P_{(1,1)}}, \eqref{eq:p.s., P_{(0,2)}}, 
\eqref{eq:p.s., P_{(-1,-1)}} 
and \eqref{eq:p.s., P_{(-2,0)}}, 
this system of equations is equivalent to 
\begin{align}
& 
\begin{cases}
-4(\nu_{1}+2)\, c(2) + 2(\nu_{2}+1)\, c(0) = 0, 
\\
-2 \nu_{1}\, c(0) + 4(\nu_{2}-1)\, c(-2) = 0, 
\\
12(\nu_{1}-2)\, c(2)+2(\nu_{2}+1)\, c(0) = 0, 
\\
2(\nu_{1}-4)\, c(0) + 12(\nu_{2}-1)\, c(-2) = 0, 
\\
2 \nu_{1}\, c(0) - 4(\nu_{2}-1)\, c(2) = 0, 
\\
4(\nu_{1}+2)\, c(-2) - 2(\nu_{2}+1)\, c(0) = 0, 
\\
2(\nu_{1}-4)\, c(0) + 12(\nu_{2}-1)\, c(2) = 0, 
\\
12(\nu_{1}-2)\, c(-2) + 2(\nu_{2}+1)\, c(0) = 0, 
\end{cases}
\label{eq:4 is sub,1}
\\
&
\begin{cases}
(\nu_{2}+2)\, c(1) = 0, 
\\
-3(\nu_{1}+1)\, c(1) + 3 \nu_{2}\, c(-1) = 0, 
\\
-(\nu_{1}-1)\, c(-1) = 0, 
\\
6(\nu_{1}-3)\, c(1) + 6 \nu_{2}\, c(-1) = 0, 
\\
(\nu_{1}-1)\, c(1) = 0, 
\\
3(\nu_{1}+1)\, c(-1) - 3 \nu_{2}\, c(1) = 0, 
\\
-(\nu_{2}+2)\, c(-1) = 0, 
\\
6(\nu_{1}-3)\, c(-1) + 6 \nu_{2}\, c(1) = 0. 
\end{cases}
\label{eq:4 is sub,2}
\end{align}
\begin{lemma}\label{lemma:4 is sub}
\begin{enumerate}
\item
The system of equations \eqref{eq:4 is sub,1} has 
a non-zero solution if and only if 
$(\nu_{1},\nu_{2}) = (1,2) = (\Lambda_{2},\Lambda_{1})$ 
or 
$(1,-2) = (\Lambda_{2},-\Lambda_{1})$. 
In these cases, the solutions are 
\[
\phi_{(2,-2)}^{4}(e) 
= 
\alpha \times 
\begin{cases}
v_{2}^{(2,-2)} + 2 v_{0}^{(2,-2)} + v_{-2}^{(2,-2)} 
\\
v_{2}^{(2,-2)} - 6 v_{0}^{(2,-2)} + v_{-2}^{(2,-2)} 
\end{cases}, 
\qquad \alpha \in \C, 
\]
respectively. 
\item
The system of equations \eqref{eq:4 is sub,2} has 
a non-zero solution if and only if 
$(\nu_{1},\nu_{2})$ 
$=(1,-2)=(\Lambda_{2},-\Lambda_{1})$. 
In this case, the solution is 
\[
\phi_{(2,-2)}^{4}(e) 
= 
\alpha (v_{1}^{(2,-2)} - v_{-1}^{(2,-2)}), 
\qquad 
\alpha \in \C. 
\]
\item
If the infinitesimal character is trivial, 
then 
$\overline{X}(\gamma_{4})$ can be a submodule of $I(\sigma,\nu)$ only if 
\[
(\sigma,\nu) 
= 
(\sigma_{0,0},(\Lambda_{2},\Lambda_{1})), 
(\sigma_{0,0},(\Lambda_{2},-\Lambda_{1})) 
\quad \mbox{or} \quad 
(\sigma_{1,1},(\Lambda_{2},-\Lambda_{1})). 
\]
\end{enumerate}
\end{lemma}

\subsection{The irreducible modules  
$\overline{X}(\gamma_{7})$ and $\overline{X}(\gamma_{8})$.}

If the infinitesimal character is trivial, 
then $\lambda=(2,0)$ is a $K$-type of 
the irreducible representation $\overline{X}(\gamma_{7})$ 
with multiplicity one, 
and $\lambda=(3,1), (2,-2), (1,-1), (0,0)$ are not. 
Therefore, if $\overline{X}(\gamma_{7})$ is a submodule of 
$I(\sigma,\nu)$, then there exists a non-zero function 
\[
\phi_{(2,0)}^{7}(kan) 
= 
a^{-\nu-\rho} \tau_{(0,-2)}(k^{-1}) 
\sum_{q=-2}^{0} 
c(q)\, 
v_{q}^{(0,-2)} 
\]
which satisfies 
\begin{align*}
&
P_{(1,1)} \phi_{(2,0)}^{7} = 0, 
&
&
P_{(0,-2)} \phi_{(2,0)}^{7} = 0, 
&
&
P_{(-1,-1)} \phi_{(2,0)}^{7} = 0, 
&
&
P_{(-2,0)} \phi_{(2,0)}^{7} = 0. 
\end{align*}
Since $\overline{X}(\gamma_{7})$ is not a composition factor of 
$I(\sigma_{1,1},\nu) \approx X(\gamma_{11})$ 
(Theorem~\ref{theorem:composition factors for Sp(2)}), 
we may assume that $c(q)=0$ if $q$ is odd. 
Then, by \eqref{eq:p.s., P_{(1,1)}}, 
\eqref{eq:p.s., P_{(0,-2)}}, 
\eqref{eq:p.s., P_{(-1,-1)}} and \eqref{eq:p.s., P_{(-2,0)}}, 
this system of equations is equivalent to 
\begin{align}
& 
\begin{cases}
-2 (\nu_{1}+2)\, c(0) + 2(\nu_{2}+1)\, c(-2) = 0, 
\\
(\nu_{1}+2)\, c(0) = 0, 
\qquad 
(\nu_{2}+1)\, c(-2) = 0, 
\\
(\nu_{1}+4)\, c(-2) + (\nu_{2}-1)\, c(0) = 0, 
\\
2 \nu_{1}\, c(-2) - 2(\nu_{2}-1)\, c(0) = 0, 
\\
2 (\nu_{1}-2)\, c(-2) + 2(\nu_{2}-1)\, c(0) = 0. 
\end{cases}
\label{eq:7 is sub}
\end{align}
\begin{lemma}\label{lemma:7 is sub}
\begin{enumerate}
\item
There exists a non-trivial solution of 
\eqref{eq:7 is sub} if and only if 
$(\nu_{1},\nu_{2})=(-2,1)=(-\Lambda_{1},\Lambda_{2})$. 
In this case, the solution is 
\[
\phi_{(0,-2)}^{7}(e)
=
\alpha \, 
v_{0}^{(0,-2)}, 
\quad 
\alpha \in \C. 
\]
\item
If the infinitesimal character is trivial, 
the the representation 
$\overline{X}({\gamma}_{7}) \oplus \overline{X}(\gamma_{8})$ 
is a submodule of 
$I(\sigma,\nu)$ if and only if 
$(\sigma,\nu) = (\sigma_{0,0}, (-\Lambda_{1},\Lambda_{2}))$. 
\end{enumerate}
\end{lemma}
\begin{proof}
The proof of (2) is the same as that of Lemma~\ref{lemma:5 is sub}. 
\end{proof}

\subsection{The irreducible module $\overline{X}(\gamma_{9})$.} 
If the infinitesimal character is trivial, 
then $\lambda=(0,-2)$ is a $K$-type of the irreducible representation 
$\overline{X}(\gamma_{9})$ with multiplicity one, 
but $\lambda=(0,0), (-1,-3), (-2,-2)$ are not. 
Therefore, if $\overline{X}(\gamma_{9})$ is a subrepresentation of 
$I(\sigma,\nu)$, 
then there exists a non-zero function 
\[
\phi_{(0.-2)}^{9}(kan) 
= 
a^{-\nu-\rho} \tau_{(2,0)}(k^{-1}) 
\sum_{q=0}^{2} 
c(q)\, v_{q}^{(2,0)} 
\]
which satisfies 
\begin{align*}
&
P_{(0,2)} \phi_{(0,-2)}^{9} = 0, 
&
&
P_{(-1,-1)} \phi_{(0,-2)}^{9} = 0, 
&
&
P_{(-2,0)} \phi_{(0,-2)}^{9} = 0. 
\end{align*}
By \eqref{eq:p.s., P_{(0,2)}}, 
\eqref{eq:p.s., P_{(-1,-1)}} 
and \eqref{eq:p.s., P_{(-2,0)}}, 
this system of equations is equivalent to 
\begin{align}
& 
\begin{cases}
2(\nu_{1}-2)\, c(2) + 2(\nu_{2}-1)\, c(0) = 0, 
\\
2(\nu_{1}+2)\, c(0) - 2(\nu_{2}+1)\, c(2) = 0, 
\\
2 \nu_{1}\, c(0) + 2(\nu_{2}+1)\ c(2) = 0, 
\end{cases}
\label{eq:9 is sub,1}
\\
&
(\nu_{1}+1)\, c(1) = 0, 
\qquad 
-(\nu_{2}+2)\, c(1) = 0. 
\label{eq:9 is sub,2}
\end{align}
\begin{lemma}\label{lemma:9 is sub}
\begin{enumerate}
\item
The system of equations \eqref{eq:9 is sub,1} 
has a non-zero solution if and only if 
$(\nu_{1},\nu_{2}) 
= (2,-1) = (\Lambda_{1},-\Lambda_{2})$, 
$(-1,2) = (-\Lambda_{2},\Lambda_{1})$ or 
\\
$(-1,-2) = (-\Lambda_{2},-\Lambda_{1})$. 
In these cases, the solutions are 
\begin{align*}
\phi_{(0,-2)}^{9}(e) 
= 
& 
\alpha_{1}\, 
v_{2}^{(2,0)}, 
\qquad 
\alpha_{2} 
(v_{2}^{(2,0)} + 3 v_{0}^{(2,0)}) 
\quad 
\mbox{and} 
\quad 
\alpha_{3}
(v_{2}^{(2,0)} - v_{0}^{(2,0)}), 
\end{align*}
$\alpha_{1}, \alpha_{2}, \alpha_{3} \in \C$, respectively. 
\item
The system of equations \eqref{eq:9 is sub,2} 
has a non-zero solution if and only if 
$(\nu_{1},\nu_{2})$ 
$=(-1,-2)=(-\Lambda_{2},-\Lambda_{1})$. 
In this case, the solution is 
\[
\phi_{(0,-2)}^{9}(e) 
= 
\alpha_{4}\, 
v_{1}^{(2,0)}, 
\quad \alpha_{4} \in \C. 
\]
\item
If the infinitesimal character $\Lambda$ is trivial, 
then the irreducible representation $\overline{X}(\gamma_{9})$ 
can be a submodule of $I(\sigma,\nu)$ only if 
\begin{align*}
(\sigma,\nu) 
= 
&
(\sigma_{0,0},(\Lambda_{1},-\Lambda_{2})), 
(\sigma_{0,0},(-\Lambda_{2},\Lambda_{1})), 
\\
&(\sigma_{0,0},(-\Lambda_{2},-\Lambda_{1})) 
\ \mbox{or} \ 
(\sigma_{1,1},(-\Lambda_{2},-\Lambda_{1})). 
\end{align*}
\end{enumerate}
\end{lemma}

\subsection{The irreducible module $\overline{X}(\gamma_{11})$.} 

If the infinitesimal character is trivial, 
the irreducible representation $\overline{X}(\gamma_{11})$ has the $K$-types 
$\lambda=(-1,-1)$ and $(-1,-3)$ with multiplicity one, 
and $\lambda=(-3,-3), (0,-2)$ are not $K$-types of it. 
Therefore, if $\overline{X}(\gamma_{11})$ is a submodule of 
$I(\sigma_{1,1},\nu)$, then there exists a non-zero function 
\[
\phi_{(-1,-1)}^{11}(kan) 
= 
a^{-\nu-\rho} \tau_{(1,1)}(k^{-1}) 
c\, 
v_{1}^{(1,1)}
\]
which satisfies 
\begin{align*}
& 
P_{(-2,0)} \circ P_{(0,-2)} \phi_{(-1,-1)}^{11} = 0, 
&
&
P_{(1,1)} \circ P_{(0,-2)} \phi_{(-1,-1)}^{11} = 0. 
\end{align*}
By \eqref{eq:p.s., P_{(0,-2)}}, \eqref{eq:p.s., P_{(-2,0)}} 
and 
\eqref{eq:p.s., P_{(1,1)}}, 
this system of equations is equivalent to 
\[
\begin{cases}
\{2(\nu_{1}+1)(\nu_{2}+2) + 2(\nu_{2}+2)(\nu_{1}+3)\} c 
= 
4(\nu_{1}+2)(\nu_{2}+2)\, c = 0, 
\label{eq:11 is sub,1}
\\ 
\{-2(\nu_{1}-1)(\nu_{1}+3) + 2(\nu_{2}-2)(\nu_{2}+2)\} c 
= 
2\{-(\nu_{1}+1)^{2}+\nu_{2}^{2}\} c = 0. 
\end{cases}
\]
Since we are assuming the infinitesimal character $\Lambda$ is 
$\rho_{m}=(2,1)$, 
the parameter $\nu=(\nu_{1},\nu_{2})$ is conjugate to $\rho_{m}$ 
under the action of Weyl group $W(B_{2})$. 
It follows that this system of equations has a non-zero 
solution if and only if 
$(\nu_{1},\nu_{2}) = (-2,1), (-2,-1)$ or $(1,-2)$. 

\begin{lemma}\label{lemma:11 is sub}
If the infinitesimal character $\Lambda$ is trivial, 
then the irreducible representation 
$\overline{X}(\gamma_{11})$ can be a submodule of 
$I(\sigma,\nu)$ only if 
\[
(\sigma,\nu) 
= 
(\sigma_{1,1},(-\Lambda_{1},\Lambda_{2})), 
(\sigma_{1,1},(-\Lambda_{1},-\Lambda_{2})) 
\ \mbox{or} \ 
(\sigma_{1,1},(\Lambda_{2},-\Lambda_{1})). 
\]
In these cases, the corresponding vectors $P_{(0,-2)} \phi_{(-1,-1)}^{11}$ 
are given by 
\begin{align*}
P_{(0,-2)} \phi_{(-1,-1)}^{11}(e) 
= 
& 
\alpha_{1} 
(v_{3}^{(3,1)} + 3 v_{1}^{(3,1)}), 
\qquad 
\alpha_{2} 
(v_{3}^{(3,1)} + v_{1}^{(3,1)}), 
\quad \mbox{and} \quad 
\alpha_{3} v_{3}^{(3,1)}, 
\end{align*}
$\alpha_{1}, \alpha_{2}, \alpha_{3} \in \C$, respectively. 
\end{lemma}

\subsection{The irreducible module $\overline{X}(\gamma_{10})$}
If the infinitesimal character is trivial, then 
the irreducible representation $\overline{X}(\gamma_{10})$ 
is the trivial representation of $G$. 
Its $K$-type is $\lambda=(0,0)$ alone. 
If it is a submodule of $I(\sigma,\nu)$, then 
there exists a non-zero function 
\[
\phi_{(0,0)}^{10}(kan) 
= 
a^{-\nu-\rho} \tau_{(0,0)}(k^{-1}) 
c\, v_{0}^{(0,0)}
\]
which satisfies 
$P_{(2,0)} \phi_{(0,0)}^{10} = 0$ and 
$P_{(0,-2)} \phi_{(0,0)}^{10} = 0$. 
By \eqref{eq:p.s., P_{(2,0)}} and \eqref{eq:p.s., P_{(0,-2)}}, 
this system of equations is equivalent to 
\[
(\nu_{1}+2)\, c = (\nu_{2}+1)\, c = 0. 
\]
\begin{lemma}\label{lemma:10 is sub} 
If the infinitesimal character $\Lambda$ is trivial, 
the the irreducible module $\overline{X}(\gamma_{10})$ is a submodule 
of $I(\sigma,\nu)$ if and only if 
$(\sigma,\nu) 
= 
(\sigma_{0,0},(-\Lambda_{1},-\Lambda_{2}))$. 
\end{lemma}

\subsection{The irreducible modules $\overline{X}(\gamma_{0'})$ 
and $\overline{X}(\gamma_{1'})$.}

Since $\overline{X}(\gamma_{0'})$ and $\overline{X}(\gamma_{1'})$ are dual, 
they appear in a pair, by the horizontal symmetry. 

We know that $\lambda=(3,2)$ is a $K$-type 
of $\overline{X}(\gamma_{0'})$ but $\lambda=(2,1)$ is not. 
So if $\overline{X}(\gamma_{0'})$ is a submodule of 
$I(\sigma, (\nu_{1},\nu_{2}))$, 
then there exists a non-zero vector 
\[
\phi_{(3,2)}^{0'}(kan)
= 
a^{-\nu-\rho} \tau_{(-2,-3)}(k^{-1}) 
\sum_{q=-3}^{-2} c(q)\, v_{q}^{(-2,-3)} 
\] 
which satisfies 
\[
P_{(-1,-1)} \phi_{(3,2)}^{0'}(e) = 0. 
\]
By \eqref{eq:p.s., P_{(-1,-1)}}, this system of equations is equivalent to  
\begin{equation}
\label{eq:0' is sub} 
(\nu_{1}-2)\, c(-3) = 0, 
\qquad 
-(\nu_{2}-2)\, c(-2) = 0. 
\end{equation}
\begin{lemma}\label{lemma:0' is sub} 
Suppose that the infinitesimal character is trivial. 
\begin{enumerate}
\item
The system of equations \eqref{eq:0' is sub} has a non-zero solution if 
and only if $(\nu_{1},\nu_{2}) = (2,\pm1) = (\Lambda_{1},\pm \Lambda_{2})$ 
or $(\pm 1, 2) = (\pm \Lambda_{2}, \Lambda_{1})$. 
In these cases, the corresponding vectors $\phi_{(3,2)}^{0'}(e)$ are 
constant multiple of 
$v_{-3}^{(-2,-3)}$ and $v_{-2}^{(-2,-3)}$, respectively. 
\item
The module 
$\overline{X}(\gamma_{0'}) \oplus \overline{X}(\gamma_{1'})$ can be 
a submodule of $I(\sigma, \nu)$ only if 
$(\sigma,\nu) 
= 
(\sigma_{1,0},(\Lambda_{1},\pm \Lambda_{2}))$ or 
$(\sigma_{0,1},(\pm \Lambda_{2}, \Lambda_{1}))$. 
\end{enumerate}
\end{lemma}

\subsection{The irreducible modules 
$\overline{X}(\gamma_{2'})$ and $\overline{X}(\gamma_{3'})$.}

We know that $\lambda=(2,1)$ is a $K$-type 
of $\overline{X}(\gamma_{2'})$ but $\lambda=(3,2)$ and 
$(1,0)$ are not. 
So if $\overline{X}(\gamma_{2'})$ is a submodule of 
$I(\sigma, (\nu_{1},\nu_{2}))$, 
then there exists a non-zero vector 
\[
\phi_{(2,1)}^{2'}(kan)
= 
a^{-\nu-\rho} \tau_{(-1,-2)}(k^{-1}) 
\sum_{q=-2}^{-1} c(q)\, v_{q}^{(-1,-2)} 
\] 
which satisfies 
\[
P_{(1,1)} \phi_{(2,1)}^{2'}(e) = 0, 
\qquad 
P_{(-1,-1)} \phi_{(2,1)}^{2'}(e) = 0,
\]
By \eqref{eq:p.s., P_{(1,1)}} and \eqref{eq:p.s., P_{(-1,-1)}}, 
this system of equations is equivalent to  
\begin{align}
& 
(\nu_{2}+2)\, c(-2) = (\nu_{1}-1) c(-2) = 0,
\label{eq:2' is sub-1}
\\
&
-(\nu_{1}+2) c(-1) = -(\nu_{2}-1) v(-1) = 0. 
\label{eq:2' is sub-2}
\end{align} 
\begin{lemma}\label{lemmq:2' is sub} 
Suppose that the infinitesimal character is trivial. 
\begin{enumerate}
\item
The system of equations \eqref{eq:2' is sub-1} has a non-zero solution if 
and only if $(\nu_{1},\nu_{2}) = (1,-2) = (\Lambda_{2},-\Lambda_{1})$.  
In this case, the corresponding solution is a constant multiple of 
$v_{-2}^{(-1,-2)}$. 
\item
The system of equations \eqref{eq:2' is sub-2} has a non-zero solution if 
and only if $(\nu_{1},\nu_{2}) = (-2,1) = (-\Lambda_{1},\Lambda_{2})$.  
In this case, the corresponding solution is a constant multiple of 
$v_{-1}^{(-1,-2)}$. 
\item
The module 
$\overline{X}(\gamma_{2'}) \oplus \overline{X}(\gamma_{3'})$ can be 
a submodule of $I(\sigma, \nu)$ only if 
$(\sigma,\nu) 
= 
(\sigma_{0,1},(\Lambda_{2}, -\Lambda_{1}))$ or 
$(\sigma_{1,0},(-\Lambda_{1},\Lambda_{2}))$. 
\end{enumerate}
\end{lemma}

\subsection{The irreducible module $\overline{X}(\gamma_{4'})$.} 
The irreducible module $\overline{X}(\gamma_{4'})$ has a $K$-type 
$\lambda=(1,0)$ with 
multiplicity one, 
and $\lambda=(2,1)$ is not a $K$-type of it. 
Therefore, if $\overline{X}(\gamma_{4'})$ is a submodule of 
$I(\sigma,\nu)$, 
then there are non-zero vectors 
\begin{align*}
&
\phi_{(1,0)}^{4'}(kan)  
= a^{-\nu-\rho} \tau_{(0,-1)}(k^{-1}) 
\sum_{q=-1}^{0} c(q)\, v_{q}^{(0,-1)}  
\end{align*}
which satisfies 
\begin{equation}
\label{eq:4' is sub}
P_{(1,1)} \phi_{(1,0)}^{4'}(e) 
=
(\nu_{2}+1)\, c(-1)\, v_{-1}^{(-1,-2)} 
- 
(\nu_{1}+1)\, c(0)\, v_{-2}^{(-1,-2)} = 0.  
\end{equation}

\begin{lemma}\label{lemmq:4' is sub} 
Suppose that the infinitesimal character is trivial. 
\begin{enumerate}
\item
The system of equations \eqref{eq:4' is sub} has a non-zero solution if 
and only if $(\nu_{1},\nu_{2}) = (\pm 2,-1) 
= (\pm \Lambda_{1}, -\Lambda_{2})$ 
or $(-1, \pm 2) = (-\Lambda_{2}, \pm \Lambda_{1})$. 
In these cases, the corresponding vectors $\phi_{(1,0)}^{4'}(e)$ are 
constant multiples of 
$v_{-1}^{(0,-1)}$ and $v_{0}^{(0,-1)}$, respectively. 
\item
$\overline{X}(\gamma_{4'})$ can be 
a submodule of $I(\sigma, \nu)$ only if 
$(\sigma,\nu) 
= 
(\sigma_{1,0},(\pm \Lambda_{1}, -\Lambda_{2}))$ or 
$(\sigma_{0,1},(-\Lambda_{2}, \pm \Lambda_{1}))$. 
\end{enumerate}
\end{lemma}



\subsection{Summary of the candidates for submodules of 
principal series.} 

We summarize the results obtained in this section. 
\begin{proposition}\label{proposition:candidates for sub}
Suppose that the infinitesimal character $\Lambda$ is trivial. 
\begin{enumerate}
\item 
(Block $PSO(3,2)$) 
The candidates for irreducible submodules of the principal series 
modules $I(\sigma_{i,i}, \nu)$, 
$i=0,1$, $\nu \in W(B_{2}) \cdot \Lambda$ are as follows: 
\begin{center}
\begin{tabular}{|c|c|c||c|c|}
\hline
$\sigma$ & $\nu$ & \rm{Irred.Mod} $\overline{X}(\gamma_{j})$ 
& $\nu$ & \rm{Irred.Mod} $\overline{X}(\gamma_{j})$ 
\\ \hline\hline
$\sigma_{0,0}$ & $(\Lambda_{1},\Lambda_{2})$ 
& $j=0,1$ 
& $(-\Lambda_{1},-\Lambda_{2})$ & $j=10$ 
\\ \hline
& $(\Lambda_{1},-\Lambda_{2})$ & $j=0,1,9$ 
& $(-\Lambda_{1},\Lambda_{2})$ & $j=7,8$ 
\\ \hline 
& $(\Lambda_{2},\Lambda_{1})$ & $j=4$ 
& $(-\Lambda_{2},-\Lambda_{1})$ & $j=9$ 
\\ \hline 
& $(\Lambda_{2},-\Lambda_{1})$ & $j=4$ 
& $(-\Lambda_{2},\Lambda_{1})$ & $j=9$ 
\\ \hline \hline 
$\sigma_{1,1}$ & $(\Lambda_{1},\Lambda_{2})$ 
& $j=0,1$ 
& $(-\Lambda_{1},-\Lambda_{2})$ & $j=11$ 
\\ \hline
& $(\Lambda_{1},-\Lambda_{2})$ & $j=0,1$ 
& $(-\Lambda_{1},\Lambda_{2})$ & $j=11$ 
\\ \hline 
& $(\Lambda_{2},\Lambda_{1})$ & 
$j=0,1,2,3$ 
& $(-\Lambda_{2},-\Lambda_{1})$ & $j=9$ 
\\ \hline 
& $(\Lambda_{2},-\Lambda_{1})$ & $j=4, 11$ 
& $(-\Lambda_{2},\Lambda_{1})$ & $j=5,6$ 
\\ \hline 
\end{tabular}
\end{center}
\item
The candidates in (1) are actually submodules 
at least except for the following two cases; 
\begin{enumerate}
\item
$(\sigma,\nu)=(\sigma_{0,0} (\Lambda_{1},-\Lambda_{2}))$, 
irreducible factor $\overline{X}(\gamma_{9})$. 
\item
$(\sigma,\nu)=(\sigma_{1,1},(\Lambda_{2},-\Lambda_{1}))$, 
irreducible factors $\overline{X}(\gamma_{4})$ 
and $\overline{X}(\gamma_{11})$. 
\end{enumerate}
\item (Block $PSO(4,1)$) 
The candidates for submodules of the principal series 
representations $I(\sigma, \nu)$, 
$(\sigma,\nu) \in W(B_{2}) \cdot (\sigma_{1,0}, \Lambda)$ are as follows: 
\begin{center}
\begin{tabular}{|c|c|c||c|c|}
\hline
$\sigma$ & $\nu$ & \rm{Irred.Mod} $\overline{X}(\gamma_{j})$ 
& $\nu$ & \rm{Irred.Mod} $\overline{X}(\gamma_{j})$ 
\\ \hline\hline
$\sigma_{1,0}$ & $(\Lambda_{1},\Lambda_{2})$ 
& $j = 0', 1'$ 
& $(-\Lambda_{1},-\Lambda_{2})$ & $j = 4'$ 
\\ \hline
& $(\Lambda_{1},-\Lambda_{2})$ & $j = 0', 1', 4'$ 
& $(-\Lambda_{1}, \Lambda_{2})$ & $j = 2', 3'$ 
\\ \hline \hline 
$\sigma_{0,1}$ & $(\Lambda_{2},\Lambda_{1})$ 
& $j = 0', 1'$ 
& $(-\Lambda_{2},-\Lambda_{1})$ & $j = 4'$ 
\\ \hline
& $(\Lambda_{2},-\Lambda_{1})$ & $j = 2', 3'$ 
& $(-\Lambda_{2},\Lambda_{1})$ & $j = 0', 1', 4'$ 
\\ \hline 
\end{tabular}
\end{center}
\end{enumerate}
\end{proposition}

\begin{proof}
(1) and (3) are summaries of the results obtained in this section. 

(2) 
Firstly,  
for a fixed principal series, if there is only one candidate 
for its submodule, then it is really the unique submodule. 
From the table in (1), 
we see that 
$\overline{X}(\gamma_{10})$ in 
$I(\sigma_{0,0},(-\Lambda_{1},-\Lambda_{2}))$; 
$\overline{X}(\gamma_{9})$ in 
$I(\sigma_{i,i},(-\Lambda_{2},-\Lambda_{1}))$, $i=0,1$, and 
$I(\sigma_{0,0},(-\Lambda_{2},\Lambda_{1}))$; 
$\overline{X}(\gamma_{4})$ in 
$I(\sigma_{0,0},(\Lambda_{2},\pm \Lambda_{1}))$; 
and 
$\overline{X}(\gamma_{11})$ in 
$I(\sigma_{1,1},(-\Lambda_{1},\pm \Lambda_{2}))$ are submodules.

Secondly, 
if an irreducible module can be a submodule of 
only one principal series, then it is actually a submodule because of 
the Casselman's subrepresentation theorem. 
It follows that 
$\overline{X}(\gamma_{7}), \overline{X}(\gamma_{8})$ 
in $I(\sigma_{0,0},(-\Lambda_{1},\Lambda_{2}))$; 
and 
$\overline{X}(\gamma_{5}), \overline{X}(\gamma_{6})$ 
in $I(\sigma_{1,1},(-\Lambda_{2},\Lambda_{1}))$ 
are submodules. 

Thirdly, by a theorem of Yamashita (\cite{Y}), 
embeddings of discrete series modules 
into principal series modules are characterized by 
the system of equations 
$P_{-\alpha} \phi_{\lambda} = 0$, 
if its infinitesimal character is far from the walls.  
Here $\alpha$ runs through the set of positive noncompact roots and 
$\lambda$ is the minimal $K$-type of this discrete series. 
Though we calculated only the case of trivial infinitesimal character, 
it is easy to see that results analogous to (1) 
holds for every discrete series module and every infinitesimal character. 
It follows that the $\overline{X}(\gamma_{0}), \overline{X}(\gamma_{1}), 
\overline{X}(\gamma_{2})$ and $\overline{X}(\gamma_{3})$ 
in the list in (1) are submodules. 
\end{proof}

Afterwords, we will show that all irreducible factors in (2) are 
submodules.

In order to determine the socle filtration of principal series 
representations, we use the tools explained in 
Section~\ref{section:other tools}. 

First of all, for $G=Sp(2,\R)$, 
th dual module $I(\sigma,\nu)^{\ast}$ is isomorphic to 
$I(\sigma,-\nu)$.

Secondly, the integral intertwining operators for $G=Sp(2,\R)$ are 
\begin{align}
I(\sigma_{i,j},(\Lambda_{1},\Lambda_{2})) 
& \rightarrow I(\sigma_{j,i},(\Lambda_{2},\Lambda_{1})) 
\rightarrow I(\sigma_{j,i},(\Lambda_{2},-\Lambda_{1})) 
\label{eq:integral intertwining operator,1}
\\
& 
\rightarrow I(\sigma_{i,j}(-\Lambda_{1},\Lambda_{2})) 
\rightarrow I(\sigma_{i,j},(-\Lambda_{1},-\Lambda_{2}))
\quad \mbox{and}
\notag\\
I(\sigma_{i,j},(\Lambda_{1},\Lambda_{2})) 
& \rightarrow I(\sigma_{i,j},(\Lambda_{1},-\Lambda_{2})) 
\rightarrow I(\sigma_{j,i},(-\Lambda_{2},\Lambda_{1})) 
\label{eq:integral intertwining operator,2}\\
& 
\rightarrow I(\sigma_{j,i}(-\Lambda_{2},-\Lambda_{1})) 
\rightarrow I(\sigma_{i,j},(-\Lambda_{1},-\Lambda_{2})). 
\notag
\end{align}

By using these intertwining operators, we can deduce some information 
about the socle filtrations. 
\begin{lemma}\label{lemma:intertwining operator}
Suppose that the infinitesimal character is trivial. 
\begin{enumerate}
\item
The irreducible module $\overline{X}(\gamma_{9})$ lies 
in the socle of $I(\sigma_{0,0},(\Lambda_{1},-\Lambda_{2}))$. 
\item
The socle of $I(\sigma_{1,1},(\Lambda_{2},-\Lambda_{1}))$ is 
$\overline{X}(\gamma_{4})$ or 
$\overline{X}(\gamma_{4}) \oplus \overline{X}(\gamma_{11})$. 
\item
There exist isomorphisms 
$I(\sigma_{0,0},(\pm \Lambda_{2},\Lambda_{1})) 
\simeq 
I(\sigma_{0,0},(\pm \Lambda_{2},-\Lambda_{1}))$ and 
\\
$I(\sigma_{1,1},(\pm \Lambda_{1},\Lambda_{2})) 
\simeq 
I(\sigma_{1,1},(\pm \Lambda_{1},-\Lambda_{2}))$. 
\end{enumerate}
\end{lemma} 
\begin{proof}
(1) Suppose $\overline{X}(\gamma_{9})$ is not in the socle of 
$I(\sigma_{0,0},(\Lambda_{1},-\Lambda_{2}))$. 
Then by Proposition~\ref{proposition:candidates for sub} (1), 
the socles of $I(\sigma_{0,0},(\Lambda_{1},\Lambda_{2}))$ 
and $I(\sigma_{0,0},(\Lambda_{1},-\Lambda_{2}))$ are both 
$\overline{X}(\gamma_{0}) \oplus \overline{X}(\gamma_{1})$, 
and the multiplicities of 
these irreducible factors in the principal series modules are both one. 
Since there is a non-zero intertwining operators between these two 
principal series modules, they must be isomorphic. 
On the other hand, the socles of their dual modules 
$I(\sigma_{0,0},(-\Lambda_{1},-\Lambda_{2}))$ and 
$I(\sigma_{0,0},(-\Lambda_{1},\Lambda_{2}))$ are different; 
the socle of the former is $\overline{X}(\gamma_{10})$ but 
that of the latter is 
$\overline{X}(\gamma_{7}) \oplus \overline{X}(\gamma_{8})$. 
This is a contradiction, so $\overline{X}(\gamma_{9})$ 
lies in the socle of 
$I(\sigma_{0,0},(\Lambda_{1},-\Lambda_{2}))$.

(2) is proved just in the same way as in (1). 
There is a non-zero intertwining operator 
$I(\sigma_{1,1},(\Lambda_{2},-\Lambda_{1})) 
\to 
I(\sigma_{1,1},(-\Lambda_{1},\Lambda_{2}))$. 
Compare the submodules of these principal series modules and 
those of their dual modules. 
Then we know that the socles of them must be different, 
so the socle of $I(\sigma_{1,1},(\Lambda_{2},-\Lambda_{1}))$ does not 
consist of $\overline{X}(\gamma_{11})$ alone. 

(3) There are non-zero intertwining operators between these pairs of 
principal series modules. 
For each pair, Proposition~\ref{proposition:candidates for sub} 
says that the two principal series modules have the same socles. 
Moreover, the multiplicities of 
$\overline{X}(\gamma_{j})$, $j=0,1,4,11$, in these 
principal series modules are all one. 
It follows that these intertwining operators are isomorphisms. 
\end{proof}

Next, consider the parity condition 
(Corollary~\ref{corollary:parity condition}) for $G = Sp(2,\R)$. 
The lengths of irreducible modules of $Sp(2,\R)$ are as follows: 
\begin{center}
\begin{tabular}{|c||c|}
\hline 
\rm{Length} & \rm{Irred. Mod. $\overline{X}(\gamma_{j})$}
\\ \hline \hline 
$0$ & $j=0,1,2,3$ 
\\ \hline 
$1$ & $j=4,5,6,0',1'$ 
\\ \hline 
$2$ & $j=7,8,9,2',3'$ 
\\ \hline 
$3$ & $j=10,11,4',0''$ 
\\
\hline 
\end{tabular}
\end{center}
By Proposition~\ref{proposition:candidates for sub}, 
the lengths of the irreducible factors in the socle 
of each principal series module have the same parity. 
So we can apply Corollary~\ref{corollary:parity condition} 
to these cases. 
Especially, the socle filtrations 
of the principal series modules 
$I(\sigma,\nu)$ in the block $PSO(4,1)$ are 
completely determined by the parity condition. 

For notational convenience, 
we use the notation \eqref{eq:example of socle filtration'}. 

\begin{theorem}\label{theorem:mail result, PSO(4,1)}
Let $\Lambda=(\Lambda_{1},\Lambda_{2})$ be a dominant nonsingular integral 
infinitesimal character of $G=Sp(2,\R)$. 
Then the socle filtrations of the principal series modules 
$I(w \cdot (\sigma_{\Lambda_{1}+1,\Lambda_{2}+1}, 
(\Lambda_{1},\Lambda_{2})))$, 
$w \in W(B_{2})$, are 
\begin{center}
\begin{tabular}{|c||c|}
\hline 
$I(\sigma_{\Lambda_{1}+1,\Lambda_{2}+1}, (\Lambda_{1},\Lambda_{2}))$, 
& 
$I(\sigma_{\Lambda_{1}+1,\Lambda_{2}+1}, (-\Lambda_{1},-\Lambda_{2}))$, 
\\
$I(\sigma_{\Lambda_{2}+1,\Lambda_{1}+1}, (\Lambda_{2},\Lambda_{1}))$ 
& 
$I(\sigma_{\Lambda_{2}+1,\Lambda_{1}+1}, (-\Lambda_{2},-\Lambda_{1}))$ 
\\\hline
\
\vspace{-3.5mm}
& 
\\
$\begin{xy}
(0,4)*{\overline{4'}}="A_{1}",
(0,0)*{\overline{2'} 
\oplus 
\overline{3'}}="A_{2}",
(0,-4)*{\overline{0'} 
\oplus 
\overline{1'}}="A_{3}",
\end{xy}$
\quad
&
$\begin{xy}
(0,-4)*{\overline{4'}}="A_{1}",
(0,0)*{\overline{2'} 
\oplus 
\overline{3'}}="A_{2}",
(0,4)*{\overline{0'} 
\oplus 
\overline{1'}}="A_{3}",
\end{xy}$
\\ 
\ \vspace{-3.5mm} 
& 
\\ \hline\hline
$I(\sigma_{\Lambda_{1}+1,\Lambda_{2}+1}, (\Lambda_{1},-\Lambda_{2}))$, 
&
$I(\sigma_{\Lambda_{1}+1,\Lambda_{2}+1}, (-\Lambda_{1},\Lambda_{2}))$, 
\\
$I(\sigma_{\Lambda_{2}+1,\Lambda_{1}+1}, (-\Lambda_{2},\Lambda_{1}))$ 
&
$I(\sigma_{\Lambda_{2}+1,\Lambda_{1}+1}, (\Lambda_{2},-\Lambda_{1}))$ 
\\\hline
\
\vspace{-3.5mm}
& 
\\
\quad
$\begin{xy}
(0,2)*{\overline{2'} 
\oplus 
\overline{3'}}="A_{1}",
(0,-2)*{\overline{0'} 
\oplus 
\overline{1'} 
\oplus 
\overline{4'}}="A_{2}",
\end{xy}$
\quad
&
\quad
$\begin{xy}
(0,2)*{\overline{0'} 
\oplus 
\overline{1'} 
\oplus 
\overline{4'}}="A_{1}",
(0,-2)*{\overline{2'} 
\oplus 
\overline{3'}}="A_{2}",
\end{xy}$
\\ \hline 
\end{tabular}
\end{center}
\end{theorem}
\begin{proof}
We will show this theorem in the case when the infinitesimal 
character $\Lambda$ is trivial, 
namely 
$\Lambda = (\Lambda_{1},\Lambda_{2}) = (2,1)$. 
If this is done, 
then the cases of general infinitesimal characters follow from  
the translation principle. 

Firstly, we will show that all the candidates 
in Proposition~\ref{proposition:candidates for sub} (3) 
are actually submodules. 
By the horizontal symmetry, 
$\overline{0'} \oplus \overline{1'}$ and 
$\overline{2'} \oplus \overline{3'}$ appear 
in the same places of the socle filtrations of principal series modules. 
Since every non-zero module has its socle (this is trivial!), 
the candidates in Proposition~\ref{proposition:candidates for sub} (3) 
are actually submodules at least 
$(\sigma,\nu) \not= 
(\sigma_{1,0}, (\Lambda_{1},-\Lambda_{2}))$ and 
$(\sigma,\nu) \not= (\sigma_{0,1}, (-\Lambda_{2},-\Lambda_{1}))$. 

For these remaining cases, 
the socles of their dual modules are 
$\overline{2'} \oplus \overline{3'}$. 
We know that the lengths of $\overline{X}(\gamma_{j})$ are odd 
if $j = 0', 1', 4'$ and even if $j = 2', 3'$.  
We also know from Theorem~\ref{theorem:composition factors for Sp(2)} 
that the irreducible factors of the principal series modules 
$I(w \cdot (\sigma_{1,0},(2,1)))$ are $\overline{X}(\gamma_{j})$, 
$j = 0', 1', 2', 3', 4'$, whose multiplicities are all one. 
It follows from the parity condition 
(Corollary~\ref{corollary:parity condition}) that 
the socles of both $I(\sigma_{1,0}, (\Lambda_{1},-\Lambda_{2}))$ 
and 
$I(\sigma_{0,1}, (-\Lambda_{2},-\Lambda_{1}))$ are 
$\overline{0'} 
\oplus 
\overline{1'} 
\oplus 
\overline{4'}$. 

Finally, for each principal series module, 
there is only one possibility of the socle filtration which satisfies 
Proposition~\ref{proposition:candidates for sub}, 
the parity condition (Corollary~\ref{corollary:parity condition}) 
and 
the dual principal series property 
(Section~\ref{subsection:dual principal series}). 
The possible filtration is the one indicated in the theorem. 
\end{proof}


\section{Socle filtrations of 
$I(w \cdot (\sigma_{\Lambda_{1}, \Lambda_{2}+1}, \Lambda))$, 
$w \in W(B_{2})$} 
\label{section:determination, X(00)} 

In this section, we determine the socle filtrations of 
the principal series modules which are isomorphic to $X(\gamma_{10})$ 
in the Grothendieck group. 

Firstly we summarize the results. 

\begin{theorem}
\label{theorem:main results for Sp(2), X(gamma_{10})} 
Let $\Lambda = (\Lambda_{1},\Lambda_{2})$ be a nonsingular 
dominant integral infinitesimal character of $Sp(2,\R)$. 
The socle filtrations of 
$I(w \cdot (\sigma_{\Lambda_{1}, \Lambda_{2}+1}, \Lambda))$, 
$w \in W(B_{2})$, are as follows: 
\begin{center}
\begin{tabular}{|c||c||c|} 
\hline 
{\rm (1)} 
& 
$I(\sigma_{\Lambda_{1},\Lambda_{2}+1},(\Lambda_{1},\Lambda_{2}))$ 
&
$I(\sigma_{\Lambda_{1},\Lambda_{2}+1},(-\Lambda_{1},-\Lambda_{2}))$ 
\\ \hline 
\ \vspace{-3.5mm} 
& & 
\\
&
$\begin{xy}
(0,6)*{\overline{10}}="A_{1}",
(0,2)*{\overline{7} \oplus \overline{8} \oplus \overline{9}}="A_{2}",
(0,-2)*{\overline{4} \oplus \overline{4} \oplus \overline{5} 
\oplus \overline{6}}="A_{3}",
(0,-6)*{\overline{0} \oplus \overline{1}}="A_{4}",
\end{xy}$ 
&
$\begin{xy}
(0,-6)*{\overline{10}}="A_{1}",
(0,-2)*{\overline{7} \oplus \overline{8} \oplus \overline{9}}="A_{2}",
(0,2)*{\overline{4} \oplus \overline{4} \oplus \overline{5} 
\oplus \overline{6}}="A_{3}",
(0,6)*{\overline{0} \oplus \overline{1}}="A_{4}",
\end{xy}$ 
\\
\ \vspace{-3.5mm} 
& & 
\\ \hline \hline 
{\rm (2)} 
& 
$I(\sigma_{\Lambda_{1},\Lambda_{2}+1}(\Lambda_{1},-\Lambda_{2}))$ 
& 
$I(\sigma_{\Lambda_{1},\Lambda_{2}+1}(-\Lambda_{1},\Lambda_{2}))$ 
\\ \hline 
\ \vspace{-3.5mm} 
& & 
\\
&
$\begin{xy}
(0,4)*{\overline{7} \oplus \overline{8}}="A_{1}",
(0,0)*{\overline{4} \oplus \overline{4} \oplus \overline{5} 
\oplus \overline{6} \oplus \overline{10}}="A_{2}",
(0,-4)*{\overline{0} \oplus \overline{1} \oplus \overline{9}}="A_{3}",
\end{xy}$ 
&
$\begin{xy} 
(0,-4)*{\overline{7} \oplus \overline{8}}="A_{1}",
(0,0)*{\overline{4} \oplus \overline{4} \oplus \overline{5} 
\oplus \overline{6} \oplus \overline{10}}="A_{2}",
(0,4)*{\overline{0} \oplus \overline{1} \oplus \overline{9}}="A_{3}",
\end{xy}$ 
\\
\ \vspace{-3.5mm} 
& & 
\\ \hline \hline 
{\rm (3)} 
&
$I(\sigma_{\Lambda_{2}+1,\Lambda_{1}},(\Lambda_{2},\pm \Lambda_{1}))$ 
&
$I(\sigma_{\Lambda_{2}+1,\Lambda_{1}},(-\Lambda_{2},\mp \Lambda_{1}))$ 
\\ \hline 
\ \vspace{-3.5mm} 
& & 
\\
& 
$\begin{xy}
(0,6)*{\overline{9}}="A_{1}",
(0,2)*{\overline{4} \oplus \overline{5} \oplus
    \overline{6} \oplus \overline{10}}="A_{2}",
(0,-2)*{\overline{0} \oplus \overline{1} \oplus \overline{7} 
\oplus \overline{8}}="A_{3}",
(0,-6)*{\overline{4}}="A_{4}",
\end{xy}$ 
&
$\begin{xy}
(0,-6)*{\overline{9}}="A_{1}",
(0,-2)*{\overline{4} \oplus \overline{5} \oplus
    \overline{6} \oplus \overline{10}}="A_{2}",
(0,2)*{\overline{0} \oplus \overline{1} \oplus \overline{7} 
\oplus \overline{8}}="A_{3}",
(0,6)*{\overline{4}}="A_{4}",
\end{xy}$ 
\\ \hline 
\end{tabular}
\end{center}
\end{theorem}


Let us prove this theorem. 

By Theorem~\ref{theorem:composition factors for Sp(2)}, 
the composition factors of this principal series module are 
\begin{align*}
\overline{0} 
\oplus 
\overline{1}, \
2 \times \overline{4}, \
\overline{5} 
\oplus 
\overline{6}, \
\overline{7} 
\oplus 
\overline{8}, \
\overline{9}, \
\overline{10}. 
\end{align*}
Here, $\overline{a} \oplus \overline{b}$ 
means that these two factors 
appear in a pair (cf. Corollary~\ref{corollary:horizontal symmetry}) 
and $2 \times \overline{4}$ means that the multiplicity of 
$\overline{4}$ is two (and those of others are one). 
The lengths of irreducible modules are indicated 
above Theorem~\ref{theorem:mail result, PSO(4,1)}.

Hereafter, we calculate the socle filtration of principal series modules 
in the case when the infinitesimal character $\Lambda$ is trivial, 
namely $\Lambda=(\Lambda_{1},\Lambda_{2})=(2,1)$. 
In this case, 
$\sigma_{\Lambda_{1},\Lambda_{2}+1}$ 
and 
$\sigma_{\Lambda_{2}+1,\Lambda_{1}}$ 
are both $\sigma_{0,0}$.

\subsection{Proof of 
Theorem~\ref{theorem:main results for Sp(2), X(gamma_{10})} (1)} 

By Proposition~\ref{proposition:candidates for sub} (1), 
the socle of the module 
$I(\sigma_{0,0},(\Lambda_{1},\Lambda_{2}))$ is $\overline{0} \oplus \overline{1}$ 
and 
that of 
$I(\sigma_{0,0},(-\Lambda_{1},-\Lambda_{2}))$ is $\overline{10}$. 
The remaining factors are $\overline{4} \times 2$, 
$\overline{5} \oplus \overline{6}$, 
whose lengths are odd, 
and 
$\overline{7} \oplus \overline{8}$, $\overline{9}$, 
whose lengths are even. 
Since the lengths of $\overline{0}$, $\overline{1}$ are even and 
$\overline{10}$ odd, 
Corollary~\ref{corollary:parity condition} says that 
the second floor of $I(\sigma_{0,0},(\Lambda_{1},\Lambda_{2}))$ 
consists of irreducible factors of odd length, 
the third floor of even length, and so on, and the top floor is $\overline{10}$ 
whose length is odd. 
We shall show 
that 
\begin{equation}\label{statement:I(0,0;2,1)}
\mbox{
$\overline{7} \oplus \overline{8}$ and $\overline{9}$ lie above the 
factors $\overline{4}$ and $\overline{5} \oplus \overline{6}$. 
}
\end{equation}
This forces that there is only one possibility of the socle filtration 
of $I(\sigma_{0,0},(\Lambda_{1},\Lambda_{2}))$, 
which is nothing but the one stated in this proposition. 
The socle filtration of $I(\sigma_{0,0},(-\Lambda_{1},-\Lambda_{2}))$ is 
obtained from it by the duality of principal series modules.  

Let us show \eqref{statement:I(0,0;2,1)} by direct computation. 
Since $M_{m}$ acts by $\sigma_{0,0}$, 
the vectors $\phi_{\lambda}^{a}(e)$ which correspond to 
the $\lambda$-isotypic subspace of the irreducible factors 
$\overline{a}$ are  
\begin{align*}
& 
\phi_{(3,1)}^{5}(e) 
= v_{-2}^{(-1,-3)},  
&
& 
\phi_{(2,2)}^{7}(e) 
= 
v_{-2}^{(-2,-2)}, 
\\
& 
\phi_{(-2,-2)}^{8}(e) 
= 
v_{2}^{(2,2)},  
&
& 
\phi_{(1,-1)}^{9}(e) 
= 
v_{0}^{(1,-1)}, 
\end{align*}
respectively (cf. Proposition~\ref{proposition:K-spectra of D.S. of Sp(2)}). 
Since 
\begin{align*}
& 
P_{(1,1)} P_{(0,-2)} \phi_{(2,2)}^{7}(e) 
= -8 v_{-2}^{(-1,-3)} = - 8 \phi_{(3,1)}^{5}(e) 
\quad \mbox{and} 
\\
& 
P_{(1,1)} P_{(1,1)} \phi_{(1,-1)}^{9}(e) 
= -16 v_{-2}^{(1,-3)} 
= -16 \phi_{(3,1)}^{5}(e), 
\end{align*}
there exist $\lie{g}$-actions which send 
non-zero elements of $\overline{7}$ and $\overline{9}$ to non-zero elements 
of $\overline{5}$. 
Therefore, in the socle filtration of $I(\sigma_{0,0},(2,1))$, 
the irreducible factors $\overline{7}$ and $\overline{9}$ 
lie above the irreducible factor $\overline{5}$. 
By the horizontal symmetry (Corollary~\ref{corollary:horizontal symmetry}), 
$\overline{7} \oplus \overline{8}$ and $\overline{9}$ 
lie above the factor $\overline{5} \oplus \overline{6}$. 

Next, consider the factors $2 \times \overline{4}$. 
Since the three vectors 
\begin{align*}
& 
P_{(0,-2)} P_{(1,1)} \phi_{(1,-1)}^{9}(e) 
= 
4(2 v_{2}^{(2,-2)} - 3 v_{0}^{(2,-2)} - v_{-2}^{(2,-2)}), 
\\
& 
P_{(2,0)} P_{(-1,-1)} \phi_{(1,-1)}^{9}(e) 
= 
4(v_{2}^{(2,-2)} - 3 v_{0}^{(2,-2)} - 2 v_{-2}^{(2,-2)}), 
\\
& 
P_{(-1,-1)} P_{(2,0)} \phi_{(1,-1)}^{9}(e) 
= 
12(3 v_{0}^{(2,-2)} - v_{-2}^{(2,-2)}) 
\end{align*}
span the three dimensional space 
$V_{(2,-2)}^{U(2)}(\sigma_{0,0})$, 
$\overline{9}$ lies above $2 \times \overline{4}$. 
Moreover, 
since the two vectors 
\begin{align*}
& 
P_{(0,-2)} P_{(0,-2)} \phi_{(2,2)}^{7}(e) 
= 
8 v_{2}^{(2,-2)} 
&
&
\mbox{and} 
&
&
P_{(2,0)} P_{(2,0)} \phi_{(-2,-2)}^{8}(e) 
= 
8 v_{-2}^{(2,-2)} 
\end{align*}
span two dimensional subspace of 
$V_{(2,-2)}^{U(2)}(\sigma_{0,0})$, 
at least two factors in $\overline{9}$, $2 \times \overline{4}$ 
lie below $\overline{7} \oplus \overline{8}$. 
But we know that $\overline{9}$ lies above $2 \times \overline{4}$. 
It follows that the two $\overline{4}$'s lie below 
$\overline{7} \oplus \overline{8}$, 
and \eqref{statement:I(0,0;2,1)} is proved.

\subsection{Proof of 
Theorem~\ref{theorem:main results for Sp(2), X(gamma_{10})} (2)} 
By Proposition~\ref{proposition:candidates for sub} and 
Lemma~\ref{lemma:intertwining operator}(1), 
the socle of $I(\sigma_{0,0}(\Lambda_{1},-\Lambda_{2}))$ is 
a direct sum 
$\overline{0} \oplus \overline{1} \oplus \overline{9}$ of 
irreducible modules of even length, 
and 
that of $I(\sigma_{0,0}(-\Lambda_{1},\Lambda_{2}))$ is a direct 
sum $\overline{7} \oplus \overline{8}$ of even length. 
The remaining irreducible factors are 
$2 \times \overline{4}$, $\overline{5} \oplus \overline{6}$ and 
$\overline{10}$, 
all of whose lengths are odd. 
It follows from Corollary~\ref{corollary:parity condition} that 
there is only one possibility of the socle filtration of 
$I(\sigma_{0,0}(\Lambda_{1},-\Lambda_{2}))$, 
which is the one stated in this proposition. 

\subsection{Proof of 
Theorem~\ref{theorem:main results for Sp(2), X(gamma_{10})} (3)} 
Since $I(\sigma_{0,0},(\pm \Lambda_{2},\Lambda_{1}))$ 
is isomorphic to 
$I(\sigma_{0,0},(\pm \Lambda_{2},-\Lambda_{1}))$ 
by Lemma~\ref{lemma:intertwining operator} (3), 
it suffices to determine the socle filtrations of 
$I(\sigma_{0,0},\pm (\Lambda_{2},\Lambda_{1}))$. 
By Proposition~\ref{proposition:candidates for sub}, 
the socle of $I(\sigma_{0,0}(\Lambda_{2},\Lambda_{1}))$ is 
$\overline{4}$, whose length is odd, 
and 
that of $I(\sigma_{0,0}(-\Lambda_{2},-\Lambda_{1}))$ is 
$\overline{9}$, whose length is even. 
The remaining composition factors are 
$\overline{0} \oplus \overline{1}$, $\overline{7} \oplus \overline{8}$, 
whose lengths are even, 
and 
$1 \times \overline{4}$, $\overline{5} \oplus \overline{6}$ 
and $\overline{10}$, 
whose lengths are odd. 
We shall show that, in the socle filtration of 
$I(\sigma_{0,0}, (\Lambda_{2},\Lambda_{1}))$, 
\begin{equation}
\label{statement:I(sigma,(1,2))}
\mbox{
$\overline{4}$, $\overline{5} \oplus \overline{6}$ 
lie above $\overline{0} \oplus \overline{1}$, 
$\overline{7} \oplus \overline{8}$, 
and $\overline{10}$ lies above 
$\overline{7} \oplus \overline{8}$.}
\end{equation}

By Proposition~\ref{proposition:K-spectra of D.S. of Sp(2)}, 
the vector 
$\phi_{(3,1)}^{5}(e) 
= 
v_{-2}^{(-1,-3)}$ 
corresponds to $\overline{5}$, 
and 
$V_{(3,-1)}^{U(2)}(\sigma_{0,0})^{\ast} 
\simeq 
V_{(1,-3)}^{SU(2)}(\sigma_{0,0})$ 
is a sum of one dimensional subspace 
which corresponds to the factor $\overline{9}$ and 
one dimensional subspace which corresponds to the factor $\overline{0}$. 
Since $\overline{9}$ is the unique quotient of 
$I(\sigma_{0,0},(\Lambda_{2},\Lambda_{1}))$, 
the vector 
\[
P_{(0,-2)} \phi_{(3,1)}^{5}(e) 
= 
3 v_{0}^{(1,-3)} 
+
v_{-2}^{(1,-3)} 
\]
corresponds to $\overline{0}$. 
Since this vector is non-zero, there is a $\lie{g}$-action which sends 
a vector in $\overline{5}$ to a non-zero vector in $\overline{0}$. 
It follows that $\overline{5}$ lies above $\overline{0}$ in the 
socle filtration of $I(\sigma_{0,0},(\Lambda_{2},\Lambda_{1}))$. 
By the horizontal symmetry, $\overline{5} \oplus \overline{6}$ lies 
above $\overline{0} \oplus \overline{1}$. 
Just in the same way, 
\[
P_{(-1,-1)} \phi_{(3,1)}^{5}(e) 
= 
v_{0}^{(0,-2)} 
+ 
v_{-2}^{(0,-2)}  
\not=0
\]
implies that $\overline{5} \oplus \overline{6}$ lies above 
$\overline{7} \oplus \overline{8}$, and 
\[
P_{(2,0)} \phi_{(0,0)}^{10}(e) 
= 
P_{(2,0)} v_{0}^{(0,0)} 
= 
3 
(v_{0}^{(0,-2)} 
+
v_{-2}^{(0,-2)} 
) 
\not=0
\]
implies that $\overline{10}$ lies above $\overline{7} \oplus \overline{8}$. 

Next, we will check that the factor $\overline{4}$, 
which is not in the socle, lies above $\overline{0}$. 

As mentioned above, 
we know that $V_{(1,-3)}^{U(2)}(\sigma_{0,0})$ 
consists of two $1$-dimensional subspaces, 
one of which corresponds to the factor $\overline{0}$ and  the other 
to $\overline{9}$. 
On the other hand, 
$V_{(2,-2)}^{U(2)}(\sigma_{0,0})$ 
consists of three $1$-dimensional subspaces 
one of which corresponds to $\overline{9}$ and the other two to 
the two factors isomorphic to $\overline{4}$. 
The image of the map 
\begin{align*}
P_{(1,1)}& : V_{(2,-2)}^{U(2)}(\sigma_{0,0})^{\ast} 
\simeq V_{(2,-2)}^{U(2)}(\sigma_{0,0}) 
 \to V_{(1,-3)}^{U(2)}(\sigma_{0,0}) 
\simeq V_{(3,-1)}^{U(2)}(\sigma_{0,0})^{\ast}, 
\\
& 
\sum_{j=-1}^{1} \alpha_{2j}\, v_{2j}^{(2,-2)} 
\mapsto 
-6(2 \alpha_{2} - \alpha_{0})\, v_{0}^{(1,-3)} 
-2(\alpha_{0} - 2\alpha_{-2})\, v_{-2}^{(1,-3)} 
\end{align*}
is two dimensional. 
Since $\overline{9}$ is the irreducible quotient module, 
there is a $\lie{g}$-action sending a vector in 
a factor isomorphic to $\overline{4}$ to a non-zero vector in $\overline{0}$. 
It follows that one of the factors isomorphic to $\overline{4}$ 
lies above $\overline{0}$, 
and, by the horizontal symmetry, 
it lies above $\overline{0} \oplus \overline{1}$. 

The proof of the fact that 
$\overline{4}$ lies above $\overline{7} \oplus \overline{8}$ 
is almost the same. 
Indeed, 
$V_{(2,0)}^{U(2)}(\sigma_{0,0})^{\ast} 
\simeq 
V_{(0,-2)}^{U(2)}(\sigma_{0,0})$ 
consists of two $1$-dimensional subspaces, 
one of which corresponds to  
$\overline{7}$ and the other to $\overline{9}$. 
The image of 
\begin{align*}
P_{(0,2)} & : 
V_{(2,-2)}^{U(2)}(\sigma_{0,0})^{\ast}  
\simeq 
V_{(2,-2)}^{U(2)}(\sigma_{0,0}) 
\to V_{(0,-2)}^{U(2)}(\sigma_{0,0}) 
\simeq 
V_{(2,0)}^{U(2)}(\sigma_{0,0})^{\ast}, 
\\
& 
\sum_{j=-1}^{1} 
\alpha_{2j}\, v_{2j}^{(2,-2)} 
\mapsto
6(-2\alpha_{2}+\alpha_{0})\, v_{0}^{(0,-2)} 
- 6(\alpha_{0} - 2\alpha_{-2})\, v_{-2}^{(0,-2)} 
\end{align*}
is two dimensional. 
Therefore, there exists a $\lie{g}$-action which sends a vector in 
a factor isomorphic to $\overline{4}$ to 
a non-zero vector in $\overline{7}$. 
This implies that one of the $\overline{4}$ 
lies above $\overline{7} \oplus \overline{8}$. 

There are two possibilities of the socle filtration of 
$I(\sigma_{0,0},(\Lambda_{2},\Lambda_{1}))$ 
which satisfy both \eqref{statement:I(sigma,(1,2))} and 
Corollary~\ref{corollary:parity condition}, 
namely 
\begin{align*} 
&
\begin{xy}
(0,6)*{\overline{9}}="A_{1}",
(0,2)*{\overline{4} \oplus \overline{5} \oplus
    \overline{6} \oplus \overline{10}}="A_{2}",
(0,-2)*{\overline{0} \oplus \overline{1} \oplus \overline{7} 
\oplus \overline{8}}="A_{3}",
(0,-6)*{\overline{4}}="A_{4}",
\end{xy}
&
&
\mbox{or} 
&
&
\begin{xy}
(0,10)*{\overline{9}}="A_{1}",
(0,6)*{\overline{4} \oplus \overline{5} \oplus
    \overline{6}}="A_{2}",
(0,2)*{\overline{0} \oplus \overline{1}}="A_{3}", 
(0,-2)*{\overline{10}}="A_{4}", 
(0,-6)*{\overline{7} \oplus \overline{8}}="A_{5}",
(0,-10)*{\overline{4}}="A_{6}"
\end{xy}
\end{align*}
But the latter is impossible. 
Indeed, since the $K$-spectrum of $\overline{10}$ is not adjacent to 
those of $\overline{0}$ and $\overline{1}$, 
$\Ext_{\gK}^{1}(\overline{10}, \overline{0} \oplus \overline{1}) = 0$. 
This completes the proof. 

\section{Socle filtration of 
$I(w \cdot (\sigma_{\Lambda_{1}+1,\Lambda_{2}},\Lambda)))$, 
$w \in W(B_{2})$.}
\label{section:determination, X(11)}

Finally, let us determine the socle filtrations of 
the principal series modules which are isomorphic to $X(\gamma_{11})$ 
in the Grothendieck group. 
As in the previous section, 
we first summarize the result. 

\begin{theorem}
\label{theorem:main results for Sp(2), X(gamma_{11})} 
Let $\Lambda = (\Lambda_{1},\Lambda_{2})$ be a nonsingular 
dominant integral infinitesimal character of $Sp(2,\R)$. 
The socle filtrations of 
$I(w \cdot (\sigma_{\Lambda_{1}+1, \Lambda_{2}}, \Lambda))$, 
$w \in W(B_{2})$, are as follows: 
\begin{center}
\begin{tabular}{|c||c||c|} 
\hline 
{\rm (1)} 
& 
$I(\sigma_{\Lambda_{1}+1,\Lambda_{2}},(\Lambda_{1},\pm \Lambda_{2}))$ 
&
$I(\sigma_{\Lambda_{1}+1,\Lambda_{2}},(-\Lambda_{1},\mp \Lambda_{2}))$  
\\ \hline 
\ \vspace{-3.5mm} 
& & 
\\
&
$\begin{xy}
(0,6)*{\overline{11}}="A_{1}",
(0,2)*{\overline{2} \oplus \overline{3} \oplus \overline{9}}="A_{2}",
(0,-2)*{\overline{4} \oplus \overline{5} \oplus \overline{6}}="A_{3}",
(0,-6)*{\overline{0} \oplus \overline{1}}="A_{4}",
\end{xy}$ 
&
$\begin{xy}
(0,-6)*{\overline{11}}="A_{1}",
(0,-2)*{\overline{2} \oplus \overline{3} \oplus \overline{9}}="A_{2}",
(0,2)*{\overline{4} \oplus \overline{5} \oplus \overline{6}}="A_{3}",
(0,6)*{\overline{0} \oplus \overline{1}}="A_{4}",
\end{xy}$ 
\\
\ \vspace{-3.5mm} 
& & 
\\ \hline \hline 
{\rm (2)} 
&
$I(\sigma_{\Lambda_{2},\Lambda_{1}+1},(\Lambda_{2},\Lambda_{1}))$ 
& 
$I(\sigma_{\Lambda_{2},\Lambda_{1}+1},(-\Lambda_{2},-\Lambda_{1}))$ 
\\ \hline 
\ \vspace{-3.5mm} 
& & 
\\
&
$\begin{xy}
(0,4)*{\overline{9}}="A_{1}",
(0,0)*{\overline{4} \oplus \overline{5} \oplus \overline{6} \oplus
    \overline{11}}="A_{2}",
(0,-4)*{\overline{0} \oplus \overline{1} \oplus \overline{2} 
\oplus \overline{3}}="A_{3}",
\end{xy}$ 
&
$\begin{xy}
(0,-4)*{\overline{9}}="A_{1}",
(0,0)*{\overline{4} \oplus \overline{5} \oplus \overline{6} \oplus
    \overline{11}}="A_{2}",
(0,4)*{\overline{0} \oplus \overline{1} \oplus \overline{2} 
\oplus \overline{3}}="A_{3}",
\end{xy}$ 
\\
\ \vspace{-3.5mm} 
& & 
\\ \hline \hline 
{\rm (3)}  
& 
$I(\sigma_{\Lambda_{2},\Lambda_{1}+1},(\Lambda_{2},-\Lambda_{1}))$ 
& 
$I(\sigma_{\Lambda_{2},\Lambda_{1}+1},(-\Lambda_{2},\Lambda_{1}))$ 
\\ \hline 
\ \vspace{-3.5mm} 
& & 
\\
&
$\begin{xy}
(0,4)*{\overline{5} \oplus \overline{6}}="A_{1}",
(0,0)*{\overline{0} \oplus \overline{1} \oplus \overline{2} \oplus
    \overline{3} \oplus \overline{9}}="A_{2}",
(0,-4)*{\overline{4} \oplus \overline{11}}="A_{3}",
\end{xy}$ 
& 
$\begin{xy}
(0,-4)*{\overline{5} \oplus \overline{6}}="A_{1}",
(0,0)*{\overline{0} \oplus \overline{1} \oplus \overline{2} \oplus
    \overline{3} \oplus \overline{9}}="A_{2}",
(0,4)*{\overline{4} \oplus \overline{11}}="A_{3}",
\end{xy}$ 
\\ \hline 
\end{tabular}
\end{center}
\end{theorem}

Let us prove this theorem. 

By Theorem~\ref{theorem:composition factors for Sp(2)}, 
the composition factors of this principal series module are 
\begin{align*}
\overline{0} 
\oplus 
\overline{1}, \ 
\overline{2} 
\oplus 
\overline{3}, \ 
\overline{4}, \ 
\overline{5} 
\oplus 
\overline{6}, \ 
\overline{9}, \ 
\overline{11}. 
\end{align*}

As in the last section, we consider the case when 
the infinitesimal character $\Lambda$ is trivial. 
In this case, 
$\sigma_{\Lambda_{1}+1,\Lambda_{2}}$ 
and 
$\sigma_{\Lambda_{2},\Lambda_{1}+1}$ 
are both $\sigma_{1,1}$. 
Note that the highest weight $\lambda = (\lambda_{1},\lambda_{2})$ of 
each $K$-type of the principal series module 
$I(\sigma_{1,1}, \nu)$ satisfies $\lambda_{1}+\lambda_{2} \in 2 \Z$, 
and the $\sigma_{1,1}$-isotypic subspace 
$V_{\lambda}^{U(2)}(\sigma_{1,1})$ is 
$\Span{v_{q}^{(\lambda_{1},\lambda_{2})} \mid q \mbox{ is odd}}$.

\subsection{Proof of 
Theorem~\ref{theorem:main results for Sp(2), X(gamma_{11})} (1)} 

By Lemma~\ref{lemma:intertwining operator} (3), 
$I(\sigma_{1,1},(\pm \Lambda_{1},\Lambda_{2}))$ 
and  $
I(\sigma_{1,1},(\pm \Lambda_{1},-\Lambda_{2}))$ 
are isomorphic, 
so it suffices to show the case 
$I(\sigma_{1,1},\pm(\Lambda_{1},\Lambda_{2}))$. 
By Proposition~\ref{proposition:candidates for sub}, 
the socle of $I(\sigma_{1,1},(\Lambda_{1},\Lambda_{2}))$ is 
$\overline{0} \oplus \overline{1}$ and 
that of its dual module 
$I(\sigma_{1,1},(-\Lambda_{1},-\Lambda_{2}))$ is 
$\overline{11}$. 
The remaining composition factors are 
$\overline{2} \oplus \overline{3}$ and $\overline{9}$, 
whose lengths are even, and 
$\overline{4}$ and $\overline{5} \oplus \overline{6}$, whose lengths are odd. 

We first show that, 
in the socle filtration of $I(\sigma_{1,1},(\Lambda_{1},\Lambda_{2}))$, 
$\overline{2} \oplus \overline{3}$ lies above 
$\overline{4} \oplus \overline{5}$; 
and $\overline{9}$ lies above 
$\overline{4}$ and $\overline{5} \oplus \overline{6}$. 

The space $V_{(3,3)}^{U(2)}(\sigma_{1,1})^{\ast} 
\simeq V_{(-3,-3)}^{U(2)}(\sigma_{1,1})$ corresponds to $\overline{2}$, 
and 
$V_{(3,1)}^{U(2)}(\sigma_{1,1})^{\ast} \simeq 
V_{(-1,-3)}^{U(2)}(\sigma_{1,1})$ 
consists of two $1$-dimensional subspaces 
one of which corresponds to $\overline{5}$ 
and the other to $\overline{11}$. 
Since $\overline{11}$ is the unique irreducible quotient of 
$I(\sigma_{1,1},(\Lambda_{1},\Lambda_{2}))$, the image of 
\begin{align*}
& 
P_{(0,-2)} : V_{(-3,-3)}^{U(2)}(\sigma_{1,1}) 
\to 
V_{(-1,-3)}^{U(2)}(\sigma_{1,1}), 
&
&
v_{-3}^{(-3,-3)} 
\mapsto 
v_{-1}^{(-1,-3)}  
-
v_{-3}^{(-1,-3)} 
\not= 0 
\end{align*}
corresponds to the factor $\overline{5}$. 
It follows that $\overline{2}$ lies above $\overline{5}$, 
so $\overline{2} \oplus \overline{3}$ lies above 
$\overline{5} \oplus \overline{6}$. 

For $\overline{9}$, we discuss analogously. 
The 
space 
$V_{(2,0)}^{U(2)}(\sigma_{1,1})^{\ast} 
\simeq 
V_{(0,-2)}^{U(2)}(\sigma_{1,1}) = \C v_{-1}^{(0,-2)}$ 
corresponds to $\overline{9}$. 
Since the image of 
\begin{align*}
&
P_{(1,1)} : V_{(0,-2)}^{U(2)}(\sigma_{1,1}) 
\to 
V_{(-1,-3)}^{U(2)}(\sigma_{1,1}), 
&
&
v_{-1}^{(0,-2)}  
\mapsto 
3(v_{-1}^{(-1,-3)} - v_{-3}^{(-1,-3)}) 
\end{align*}
is non-zero, 
the factor $\overline{9}$ lies above $\overline{5}$, 
so $\overline{9}$ lies above $\overline{5} \oplus \overline{6}$. 

The one-dimensional space 
$V_{(0,-2)}^{U(2)}(\sigma_{1,1})^{\ast} 
\simeq 
V_{(2,0)}^{U(2)}(\sigma_{1,1}) 
= \C v_{1}^{(2,0)}$ 
is also contained in $\overline{9}$. 
The two-dimensional space 
$V_{(2,-2)}^{U(2)}(\sigma_{1,1})$ consists of two $1$-dimensional subspaces, 
one of which corresponds to $\overline{9}$ and the other to $\overline{4}$. 
Since 
\begin{align*}
& 
P_{(0,-2)} 
v_{-1}^{(0,-2)} 
= 5 v_{1}^{(2,-2)} + v_{-1}^{(2,-2)} 
\qquad 
\mbox{and} 
\qquad 
P_{(2,0)} 
v_{1}^{(2,0)} 
= v_{1}^{(2,-2)} + 5 v_{-1}^{(2,-2)}  
\end{align*}
span the space $V_{(2,-2)}^{U(2)}(\sigma_{1,1})$, 
there is a $\lie{g}$-action which sends a vector in $\overline{9}$ to 
a non-zero vector in $\overline{4}$. 
It follows that $\overline{9}$ lies above $\overline{4}$. 

In order to complete the proof of 
Theorem~\ref{theorem:main results for Sp(2), X(gamma_{11})} (1), 
we shall show that the extension group 
$\Ext_{\gK}^{1}(\overline{4}, \overline{2} \oplus \overline{3})$ 
vanishes. 

By the Blattner formula \eqref{eq:Blattner formula}, 
the $K$-spectra of $\overline{2}$ and $\overline{3}$ are contained in 
$\{\lambda=(\lambda_{1}, \lambda_{2}) 
\mid \lambda_{2} \geq \Lambda_{2}+2\}$ 
and 
$\{\lambda=(\lambda_{1}, \lambda_{2}) 
\mid \lambda_{1} \leq -\Lambda_{2}-2\}$, respectively.  
On the other hand, by using the Blattner formula again, 
we see that, if $\lambda_{2} \geq -\Lambda_{2}$ 
(resp. $\lambda_{1} \leq \Lambda_{2}$), 
then the multiplicity of 
$(\tau_{(\lambda_{1}, \lambda_{2})}, 
V_{(\lambda_{1},\lambda_{2})}^{U(2)})$ 
in $\overline{0}$ (resp. $\overline{1}$) is 
$\frac{1}{2}(\lambda_{1}-\lambda_{2}-\Lambda_{1}-\Lambda_{2}+1)$, 
which is the same as 
the multiplicity in $X(\gamma_{4})$ 
(cf. \eqref{eq:multi. of X(gamma_{4})}). 
Since $X(\gamma_{4}) 
= \overline{X}(\gamma_{0}) 
+ \overline{X}(\gamma_{1}) 
+ \overline{X}(\gamma_{4})$, 
the $K$-spectrum of $\overline{X}(\gamma_{4})$ is contained in 
$\{\lambda=(\lambda_{1},\lambda_{2}) 
\mid \lambda_{1} \geq \Lambda_{2} + 1, 
\lambda_{2} \leq - \Lambda_{2} - 1\}$. 
It follows that 
the $K$-spectrum of $\overline{4}$ is not adjacent 
to those of $\overline{2}$ and 
$\overline{3}$, 
so 
$\Ext_{\gK}^{1}(\overline{4}, \overline{2} \oplus \overline{3}) 
=
0$. 

Finally, by a discussion analogous to the last part of the proof of 
Theorem~\ref{theorem:main results for Sp(2), X(gamma_{10})} (3), 
we see that there is only one possibility of the socle filtration of 
$I(\sigma_{1,1},(\Lambda_{1},\Lambda_{2}))$, 
which is nothing but the one stated in this proposition.

\subsection{
Proof of Theorem~\ref{theorem:main results for Sp(2), X(gamma_{11})} (2)} 
By Proposition~\ref{proposition:candidates for sub} again, 
the socle of $I(\sigma_{1,1},(\Lambda_{2},\Lambda_{1}))$ 
is 
$\overline{0} \oplus \overline{1} \oplus \overline{2} \oplus \overline{3}$, 
and that of 
$I(\sigma_{1,1},(-\Lambda_{2},-\Lambda_{1}))$ is 
$\overline{9}$. 
The lengths of these modules are all even. 
The remaining composition factors are 
$\overline{4}$, $\overline{5} \oplus \overline{6}$ and $\overline{11}$, 
whose lengths are all odd. 
Then by Corollary~\ref{corollary:horizontal symmetry}, 
the socle filtration is uniquely determined as in the statement of 
this proposition. 

\subsection{Proof of 
Theorem~\ref{theorem:main results for Sp(2), X(gamma_{11})} (3)} 
By Proposition~\ref{proposition:candidates for sub} and 
Lemma~\ref{lemma:intertwining operator} (2), 
the socle of 
$I(\sigma_{1,1},(\Lambda_{2},-\Lambda_{1}))$ is 
$\overline{4}$ or $\overline{4} \oplus \overline{11}$, 
and 
that of $I(\sigma_{1,1},(-\Lambda_{2},\Lambda_{1}))$ is 
$\overline{5} \oplus \overline{6}$. 

We first show that the socle of 
$I(\sigma_{1,1},(\Lambda_{2},-\Lambda_{1}))$ is 
$\overline{4} \oplus \overline{11}$. 
Assume that it is not, 
namely suppose that the socle consists of $\overline{4}$ alone.   
Then, the socle filtration of $I(\sigma_{1,1},(\Lambda_{2},\Lambda_{1}))$ 
(Theorem~\ref{theorem:main results for Sp(2), X(gamma_{11})} (2)) 
implies that 
the factor $\overline{11}$ in $I(\sigma_{1,1},(\Lambda_{2},\Lambda_{1}))$ 
is contained in the kernel of the intertwining operator 
\[
I(\sigma_{1,1},(\Lambda_{2},\Lambda_{1}))
\rightarrow 
I(\sigma_{1,1},(\Lambda_{2},-\Lambda_{1})).
\]
But it is well known that the composition 
\[
I(\sigma_{1,1},(\Lambda_{1},\Lambda_{2})) 
\rightarrow 
I(\sigma_{1,1},(-\Lambda_{1},-\Lambda_{2})) 
\]
of the sequence of intertwining operators 
\eqref{eq:integral intertwining operator,1} is a non-zero map and 
its image is $\overline{11}$. 
This is a contradiction, so the sole of 
$I(\sigma_{1,1},(\Lambda_{2},-\Lambda_{1}))$ is 
$\overline{4} \oplus \overline{11}$. 

The remaining factors are $\overline{0} \oplus \overline{1}$, 
$\overline{2} \oplus \overline{3}$ and $\overline{9}$, 
whose lengths are all even. 
As before, the possibility of socle filtration of this principal series 
is unique and it is the one stated in this proposition. 


\end{document}